\documentclass[11pt,a4paper,leqno,noamsfonts]{amsart}
\usepackage[english]{babel}
\usepackage[dvipsnames]{xcolor}
\usepackage{graphicx,pifont}
\usepackage[utopia]{mathdesign}
\usepackage[a4paper,top=3.2cm,bottom=3.2cm,left=3.5cm,right=3.5cm,marginparwidth=60pt]{geometry}
\usepackage[utf8]{inputenc}
\usepackage{braket,caption,comment,mathtools,stmaryrd,enumitem} 
\usepackage[usestackEOL]{stackengine}
\usepackage{multirow,booktabs,microtype,relsize,soul}
\usepackage[foot]{amsaddr}
\usepackage[colorlinks,bookmarks]{hyperref} %
      \hypersetup{colorlinks,%
            citecolor=britishracinggreen,%
            filecolor=black,%
            linkcolor=cobalt,%
            urlcolor=purple}
      \numberwithin{equation}{section}
\usepackage[capitalise]{cleveref}
\allowdisplaybreaks

\Crefname{conjecture}{Conjecture}{Conjectures}
\DeclareSymbolFont{cmarrows}{OMS}{cmsy}{m}{n}
\SetSymbolFont{cmarrows}{bold}{OMS}{cmsy}{b}{n}
\DeclareMathSymbol{\cmminus}{\mathbin}{cmarrows}{"00}

\DeclareMathSymbol{\leftrightarrow}{\mathrel}{cmarrows}{"24}
\DeclareMathSymbol{\leftarrow}{\mathrel}{cmarrows}{"20}
   
\DeclareMathSymbol{\rightarrow}{\mathrel}{cmarrows}{"21}
   \let\to=\rightarrow
\DeclareMathSymbol{\mapstochar}{\mathrel}{cmarrows}{"37}
   \def\mapsto{\mapstochar\rightarrow}
   
\DeclareSymbolFont{usualmathcal}{OMS}{cmsy}{m}{n}
\DeclareSymbolFontAlphabet{\mathcal}{usualmathcal}
\DeclareMathAlphabet\BCal{OMS}{cmsy}{b}{n}

\makeatletter
\newcommand{\mylabel}[2]{#2\def\@currentlabel{#2}\label{#1}}
\makeatother

\definecolor{cornellred}{rgb}{0.7, 0.11, 0.11}
\definecolor{britishracinggreen}{rgb}{0.0, 0.26, 0.15}
\definecolor{cobalt}{rgb}{0.0, 0.28, 0.67}

\usepackage[colorinlistoftodos]{todonotes} 
\usepackage{amsmath,float,soul}
\usetikzlibrary{decorations.markings}

\DeclareMathOperator{\rk}{rk}

\newcommand{\Gr}{\mathrm{Gr}}

\newcommand{\Coef}{\mathsf{Coef}}

\newcommand{\mult}{\mathrm{mult}}

\newcommand{\into}{\hookrightarrow}
\newcommand{\onto}{\twoheadrightarrow}
\newcommand{\HH}{\mathrm{H}}

\newcommand{\OO}{\mathscr O}
\newcommand{\OF}{\mathscr F}
\newcommand{\OI}{\mathscr I}

\newcommand{\OZ}{\mathscr Z}

\newcommand{\OH}{\mathscr H}

\newcommand{\boldit}[1]{\boldsymbol{#1}}

\newcommand{\CCoh}{\mathscr{C}\kern-0.25em {o}\kern-0.2em{h}}

\DeclareMathOperator{\bh}{{\boldit{h}}}
\DeclareMathOperator{\ba}{{\boldit{a}}}
\DeclareMathOperator{\by}{{\boldit{y}}}

\DeclareMathOperator{\Hilb}{Hilb}
\DeclareMathOperator{\Exp}{Exp}

\newcommand{\embdim}{\mathsf{ed}}
\DeclareMathOperator{\Sym}{Sym}

\DeclareMathOperator{\Var}{Var}

\DeclareMathOperator{\In}{In}

\DeclareMathOperator{\Quot}{Quot}

\DeclareMathOperator{\Mot}{Mot}
\DeclareMathOperator{\CHMot}{CHMot}

\DeclareMathOperator{\HS}{HS}

\DeclareMathOperator{\Spec}{Spec}
\DeclareMathOperator{\id}{id}

\DeclareMathOperator{\pt}{pt}

\newcommand{\BA}{{\mathbb{A}}}

\newcommand{\BC}{{\mathbb{C}}}

\newcommand{\BG}{{\mathbb{G}}}

\newcommand{\BL}{{\mathbb{L}}}

\newcommand{\BN}{{\mathbb{N}}}

\newcommand{\BP}{{\mathbb{P}}}
\newcommand{\BQ}{{\mathbb{Q}}}

\newcommand{\BZ}{{\mathbb{Z}}}

\newcommand{\CA}{{\mathcal A}}

\newcommand{\CZ}{{\mathcal Z}}

\newcommand{\Fm}{{\mathfrak{m}}}

\usepackage[all]{xy}
\usepackage{tikz}
\usepackage{tikz-cd}
\usepackage{adjustbox}
\usepackage{rotating}
\usepackage{comment}

\usetikzlibrary{matrix,shapes,intersections,arrows,decorations.pathmorphing}
\tikzset{commutative diagrams/.cd,
mysymbol/.style={start anchor=center,end anchor=center,draw=none}}

\tikzset{
shift up/.style={
to path={([yshift=#1]\tikztostart.east) -- ([yshift=#1]\tikztotarget.west) \tikztonodes}
}
}

\usepackage{youngtab} 

\usepackage{ytableau} 
  
\theoremstyle{definition}

\newtheorem*{lemma*}{Lemma}
\newtheorem*{theorem*}{Theorem}
\newtheorem*{example*}{Example}
\newtheorem*{fact*}{Fact}
\newtheorem*{notation*}{Notation}
\newtheorem*{definition*}{Definition}
\newtheorem*{prop*}{Proposition}
\newtheorem*{remark*}{Remark}
\newtheorem*{corollary*}{Corollary}

\newtheorem*{conventions*}{Conventions}

\newtheorem{definition}{Definition}[section]

\newtheorem{example}[definition]{Example}

\newtheorem{question}[definition]{Question}

\newtheorem{remark}[definition]{Remark}

\newtheoremstyle{thm} 
        {3mm}
        {3mm}
        {\slshape}
        {0mm}
        {\bfseries}
        {.}
        {1mm}
        {}
\theoremstyle{thm}
\newtheorem{theorem}[definition]{Theorem}
\newtheorem{corollary}[definition]{Corollary}
\newtheorem{lemma}[definition]{Lemma}
\newtheorem{prop}[definition]{Proposition}
\newtheorem{conjecture}[definition]{Conjecture}

\newtheorem{thm}{Theorem}

\newtheorem*{Acknowledgments*}{Acknowledgments}

\newcommand{\motive}[1]{ \left[ #1 \right]
}

\newcommand{\tartaglia}[1]{
\begin{tikzpicture}
\foreach \i in {0,...,{#1}}
    \foreach \j in {0,...,\i}
            \pgfmathtruncatemacro{\ciao}{factorial(\i)/(factorial(\j)*factorial(\i-\j))}
	       \node at ( \j*0.8-\i/2*0.8,-\i*0.8) {\textcolor{gray}{\ciao}};
\end{tikzpicture}
}

\title[The motive of the Hilbert scheme of points in all dimensions]{The motive of the Hilbert scheme of points \\ in all dimensions}
\author{Michele Graffeo, Sergej Monavari, Riccardo Moschetti, Andrea T. Ricolfi}

\begin{document}
\begin{abstract}
We prove a closed formula for the generating function $\mathsf Z_d(t)$ of the motives $[\Hilb^d(\BA^n)_0] \in K_0(\Var_{\BC})$ of punctual Hilbert schemes, summing over $n$, for fixed $d>0$. The result is an expression for  $\mathsf Z_d(t)$ as the product of the zeta function of $\BP^{d-1}$ and a polynomial $\mathsf P_d(t)$, which in particular implies that $\mathsf Z_d(t)$ is  a rational function. Moreover, we reduce the complexity of $\mathsf P_d(t)$ to the computation of $d-8$ initial data, and therefore give explicit formulas for $\mathsf Z_d(t)$ in the cases $d \leq 8$, which in turn yields a formula for $[\Hilb^{\leq 8}(X)]$ for any smooth variety $X$. We perform a similar analysis for the Quot scheme of points, obtaining explicit formulas for the full generating function (summing over all ranks and dimensions) for $d \leq 4$. In the  limit $n \to \infty$, 
we prove that the motives $[\Hilb^d(\BA^n)_0]$ stabilise to the class of the infinite Grassmannian $\Gr(d-1,\infty)$. Finally, exploiting our geometric methods, we conjecture (and partially confirm) a structural result on the `error' measuring the discrepancy between the count of higher dimensional partitions and MacMahon's famous guess. 
\end{abstract}

\maketitle
{\hypersetup{linkcolor=black}\tableofcontents}

\section{Introduction}
\subsection{Structure theorem for Hilbert series}
Let $X$ be a complex smooth quasiprojective variety. The \emph{Hilbert scheme of $d$ points} $\Hilb^d(X)$ is the fine moduli space of $0$-dimensional closed subschemes $Z \subset X$ whose length $\chi(\OO_Z)$ is equal to $d$ \cite{Grothendieck_Quot}. As soon as $\dim X > 2$, the geometry of $\Hilb^d(X)$ is to date largely inaccessible, due to the presence of (yet to be understood) singularities and (yet to be found but definitely present) unexpected components; it is nevertheless known that Murphy's Law \cite{Murphy} holds for $\Hilb^{\mathsf{pts}}(\BA^{n})$ up to retraction for $n\geq 16$ \cite{Jelisiejew-pathologies}. 

In this paper we study the motivic class $[\Hilb^d(X)]$ in the Grothendieck ring of varieties $K_0(\Var_{\BC})$. To state our first main result, we denote by $\Hilb^d(\BA^m)_0\subset \Hilb^d(\BA^m)$ the \emph{punctual Hilbert scheme}, parametrising subschemes $Z \subset \BA^m$ entirely supported at the origin $0 \in \BA^m$, and we form the generating function
\[
\mathsf Z_d(t) = \sum_{n\geq 0}\,[\Hilb^d(\BA^{n+1})_0] t^n.
\]

\begin{thm}[\Cref{main-in-body}]\label{MAIN-THEOREM-1}
Fix an integer $d>0$. Then
\[
\mathsf Z_d(t) = \zeta_{\BP^{d-1}}(t) \cdot \mathsf{P}_d(t)
\]
in $K_0(\Var_{\BC})\llbracket t \rrbracket$, where 
\[
\zeta_{\BP^{d-1}}(t) = \prod_{i=0}^{d-1}\,\frac{1}{1-\BL^it}
\]
is the zeta function of $\BP^{d-1}$ and $\mathsf P_d(t) \in K_0(\Var_{\BC})[t]$ is a polynomial such that
\begin{itemize}
\item [\mylabel{mainthm-1}{\normalfont{(1)}}] $\mathsf{P}_d(t)=1$ if $d\leq 3$, and,
\item [\mylabel{mainthm-2}{\normalfont{(2)}}] if $d>3$, then
\[
\mathsf{P}_d(t) = \sum_{i=0}^{d-2} a_i^{(d)}t^i
\] 
has degree at most $d-2$, and its coefficients are given by
\[
a_i^{(d)} = \sum_{\alpha=0}^i\,(-1)^\alpha [\Hilb^d(\BA^{i-\alpha+1})_0][\Gr(\alpha,d)]\BL^{\binom{\alpha}{2}}, \qquad 0\leq i \leq d-2.
\]
\end{itemize}
In particular, $\mathsf Z_d(t)$ is a rational function. 
\end{thm}

We now comment on some consequences and interpretations of \Cref{MAIN-THEOREM-1}, and we give an outline of its proof.

\subsubsection{Motives of singular Hilbert schemes}
The `first' singular Hilbert scheme is $\Hilb^4(\BA^3)$, and the first reducible one (in dimension $n>3$) is $\Hilb^8(\BA^n)$. Motives of singular or reducible varieties are known to be hard to compute in general. \Cref{MAIN-THEOREM-1} computes the motives of infinitely many reducible (or irreducible but singular) Hilbert schemes.

\subsubsection{Amplification to smooth $n$-folds}
The knowledge of the motives $[\Hilb^i(\BA^n)_0]$ for all $i \leq d$ is equivalent to the knowledge of the motives $[\Hilb^i(X)]$ for all $i \leq d$ for any given smooth $n$-fold $X$ (cf.~\Cref{sec:Exp-and-stuff}). Using the results in \Cref{sec:consequences of thm A}, where we reduce the complexity of $\mathsf P_d(t)$ to the computation of simply $d-8$ `data' (cf.~\Cref{ques:howmanydata} and \Cref{rmk:d-8-data}), we provide in \Cref{app:Pd-explicit} explicit closed formulas for $\mathsf P_{\leq 8}(t)$, showing they are polynomials in $\BL$ (see also \Cref{conj:L-subalgebra}). Therefore an immediate consequence of \Cref{MAIN-THEOREM-1} is that one can compute the motives $[\Hilb^{\leq 8}(X)]$ fully explicitly for \emph{any} smooth variety $X$, and these classes are polynomials in $\BL$ as soon as $[X] \in K_0(\Var_{\BC})$ is itself a polynomial in $\BL$.  

\subsubsection{Relations among the coefficients of $\mathsf P_d(t)$}
The coefficients of $\mathsf P_d(t)$  
share relations among themselves. This is for instance reflected in a recursion involving Hilb-classes, cf.~\Cref{cor:hilbrel}. The easiest of these relations is the identity $\mathsf P_d(1)=1$ in $K_0(\Var_{\BC})$, that we prove in \Cref{Cor:P(1)=1}. See \Cref{rem: relations from Y} for more on this.

\subsubsection{The advantage of fixing the number of points}
\Cref{MAIN-THEOREM-1} says that the knowledge of a \emph{finite} number of punctual motives allows one to compute infinitely many motivic classes (for any fixed $d$). This feature is achieved thanks to our approach fixing the number of points instead of the ambient dimension: in the latter case one is instead led to compute the generating function
\[
\mathsf{Hilb}_{n,0}(t) = \sum_{d \geq 0}\,[\Hilb^d(\BA^n)_0]t^d,
\]
but a closed formula for $\mathsf{Hilb}_{n,0}(t)$ is currently out of reach for $n>2$ (see \Cref{sec:curves-and-surfaces} for the case $n\leq 2$). Even at the level of Euler characteristics, if $n>3$ there is currently no known explicit formula for the generating function $\chi\mathsf{Hilb}_{n,0}(t)$. It is again thanks to this approach that we were able to formulate \Cref{conj:chi-omega-intro} (discussed in \Cref{sec:higher-dim-partitions}), which is a new structural prediction on the discrepancy between $\chi\mathsf{Hilb}_{n,0}(t)$ and the famous guess by MacMahon. 

\subsubsection{$\Omega$-classes}
We conjecture that the (effective) classes $\Omega^n_d \in K_0(\Var_{\BC})$ defined through the relation
\begin{equation}\label{omega-classes-intro}
\mathsf{Hilb}_{n,0}(t) = \Exp \left(\sum_{d>0} \Omega^{n}_{d}t^d\right)
\end{equation}
satisfy the same recursion as the classes of punctual Hilbert schemes (\Cref{cor:hilbrel}). This is the content of \Cref{conjecture-on-Omega-GF} or, equivalently, \Cref{FormulaOmega}, which we prove for $d\leq 8$ (see also \Cref{subsec:motiveomega} for explicit formulas).

\subsubsection{Strategy of the proof of \texorpdfstring{\Cref{MAIN-THEOREM-1}}{}}
We briefly sketch the main ingredients in the proof of \Cref{MAIN-THEOREM-1}.
The starting point is the stratification
\[
\Hilb^d(\BA^n)_0 = \coprod_{k=1}^{d-1} Y^n_{k,d},
\]
where $Y^n_{k,d} \subset \Hilb^d(\BA^n)_0$ are the strata consisting of fat points $Z \subset \BA^n$ of embedding dimension \emph{exactly} $k$. These strata obey the key recursion 
\begin{equation}\label{eqn:Y-strata}
[Y^n_{k,d}] = [\Gr(k,n)]\BL^{(n-k)(d-k-1)}[Y^k_{k,d}]\,\in\,K_0(\Var_{\BC}),
\end{equation}
cf.~\Cref{thm:motive-Y-stratum}. This relation is obtained by analysing a further  stratification of $Y^n_{k,d}$ into a suitable union of Hilbert--Samuel strata $H^n_{\bh} \subset \Hilb^d(\BA^n)_0$, parametrising fat points with prescribed Hilbert--Samuel function $\bh = (1,h_1,\ldots)$. Each Hilbert--Samuel stratum admits a morphism $H^n_{\bh} \to \Gr(n-h_1,n)$, which can be interpreted as the projection of a fat point onto its tangent space, and we prove that this morphism is a Zariski locally trivial fibration, cf.~\Cref{tau-is-ZLT}. \Cref{eqn:Y-strata} is finally obtained by relating the motive of a preferred  fiber of the latter fibration to some simpler Hilbert--Samuel strata.

The next crucial ingredient is the `inversion formula' (cf.~\Cref{lemma:inversion})
\begin{equation}\label{inversion-formula-intro}
[Y_{k,d}^k]=\sum_{j=1}^{k}\, (-1)^{k+j}[\Gr(j,k)]\BL^{(k-j) (d-j-1)-\binom{k-j}{2}}[\Hilb^d(\BA^j)_0]
\end{equation}
defining an inverse transform from Hilb-data back to $Y$-data. This is used in \Cref{prop:coefficient-of-p_d} to conveniently rewrite the coefficients $a_i^{(d)}$ of $\mathsf P_d(t)$ in terms of classes of Grassmannians and $Y$-classes. Finally, in \Cref{sec:motives-of-grass} we prove a series of motivic identities on zeta functions and classes of Grassmannians, that in \Cref{sec:proof-thm-A} we combine with \Cref{eqn:Y-strata} in order to obtain the proof of \Cref{MAIN-THEOREM-1}.

\subsection{Stabilisation theorem for \texorpdfstring{$[\Hilb^d(\BA^n)_0]$}{}}
Our second main result is about the \emph{stabilisation} of $[\Hilb^d(\BA^n)_0]$ as $n \to \infty$. We work in the $\BL$-adic completion $\widehat{K}_0(\Var_\BC)$, where, as we prove, the motivic class of the `infinite Grassmannian' and of the `infinite punctual Hilbert scheme'  
\begin{align*}
\Gr(d-1,\infty)&=\varinjlim \Gr(d-1,n),\\
\Hilb^d(\BA^\infty)_0&=\varinjlim \Hilb^d(\BA^n)_0,
\end{align*}
viewed naturally as ind-schemes, make sense. The classes of infinite Grassmannians are determined by the relation \cite{{GLMstacks}}
\[
\sum_{k=0}^\infty \,[\Gr(k,\infty)]t^k=\prod_{i=0}^\infty\frac{1}{1-\BL^{i}t}\in \widehat{K}_0(\Var_\BC)\llbracket t \rrbracket.
\]
We prove the following stabilisation theorem.
\begin{thm}[{\Cref{thm: stab motives}, \Cref{cor:limit-Hilb}, \Cref{cor:wt-polynomials}}]\label{MAIN-THM-STAB}
Let $d,n\geq 1$ be integers.
Then
\[
[\Hilb^d(\BA^n)_0] =[\Gr(d-1,n)] \in K_0(\Var_\BC)/(\BL^{n-d+2}).
\]
In particular, there is an identity
\[
\sum_{d\geq 0}\,[\Hilb^d(\BA^\infty)_0]t^d =t\cdot \prod_{i=0}^\infty\frac{1}{1-\BL^{i}t}\,\in\, \widehat{K}_0(\Var_\BC)\llbracket t \rrbracket.
\]
At the level of the weight polynomial specialisation, there is an identification with the signed Poincaré polynomial
\begin{align*}
\mathsf w(\Hilb^d(\BA^n)_0,z^{1/2}) 
&= \mathsf p(\Hilb^d(\BA^n)_0,-z^{1/2})\\
&=\frac{\prod_{k=1}^n (z^k-1)}{\prod_{k=1}^{d-1} (z^k-1) \prod_{k=1}^{n-d+1} (z^k-1)}
\end{align*}
in $\BZ[z^{1/2}]/z^{n-d+2}$.
\end{thm}
As we discuss in \Cref{sec:Stabilisation}, this stabilisation result is the motivic counterpart of the $\BA^1$-homotopy equivalence $\Hilb^d(\BA^\infty) \simeq \Gr(d-1,\infty)$ proved in \cite[Thm.~2.1]{Hilb^infinity}.

\subsection{Relation to other motivic theories}
Besides the classical motivic measures reviewed in \Cref{sec:K0(Var)}, the Grothendieck ring $K_0(\Var_{\BC})$ has natural maps into other K-theories, such as the K-ring $K_0(\CHMot_{\BC})$ of the category of Chow motives, the K-ring $K_0(\Mot_{\BC})$ of Voevodsky's geometrical effective motives \cite[Sec.~3.2.4]{Gillet-Soulé}, and the K-ring $K_0(\textrm{dg-Cat}_{\BC})$ of dg-categories \cite{Bondal-Larsen-Lunts}. Therefore all relations proved for the motivic classes $[-] \in K_0(\Var_{\BC})$ also hold in these other K-theories.

\subsection{Generalisations: arbitrary $n$-folds and Quot schemes}
We propose amplifications of our main formulas for punctual Hilbert schemes of $\BA^n$ in two directions: replacing $\BA^n$ with a smooth $n$-fold $X$, and considering higher rank moduli spaces, namely \emph{Quot schemes of points} (again on arbitrary $n$-folds, but starting from $\BA^n$). For the latter, we again consider the Hilbert--Samuel stratification of the punctual Quot scheme $\Quot_{\BA^n}(\OO^{\oplus r},d)_0$. Using apolarity for modules, we establish the following formulas.

\begin{thm}[\Cref{sec:quot}]\label{MAIN-THM-QUOT}
There are identities
\[
\sum_{{r}\ge 0}\sum_{n\ge 0} \,\bigl[\Quot_{\BA^{n+1}}(\OO^{\oplus r+1},d)_0 \bigr]x^ny^r = 
\begin{cases}
\zeta_{\BP^0}(x)\zeta_{\BP^1}(y) & \mbox{if }d=1\\
\zeta_{\BP^1}(x)\zeta_{\BP^2}(y)(1-\BL xy) & \mbox{if }d=2\\
\zeta_{\BP^2}(x)\zeta_{\BP^3}(y)(1-\BL xy)(1-\BL^2xy) & \mbox{if }d=3 \\
\zeta_{\BP^3}(x)\zeta_{\BP^4}(y)\zeta_{\BL^4}(x)\mathsf U_4(x,y) & \mbox{if }d=4
\end{cases}
\]
where $\mathsf U_4(x,y) \in \BZ[\BL,x,y]$ is an explicit polynomial, given in \eqref{U-polynomial}.
\end{thm}

Going back to Hilbert schemes, in \Cref{sec:Exp-and-stuff} we explain how to pass from $\Hilb^\bullet(\BA^n)_0$ to $\Hilb^\bullet(X)$ for an arbitrary smooth $n$-dimensional quasiprojective variety. This procedure naturally generalises to Quot schemes via the `$\Omega$-classes' \eqref{omega-classes-intro} (and their higher rank analogues), which govern the motivic Hilbert/Quot scheme theory of smooth $n$-folds via the theory of power structures \cite{GLMps}. We provide two worked out examples for $X=\BP^3$, leading to the explicit formulas
\begin{align*}
[\Hilb^8(\BP^3)] 
&=
\BL^{24}
 + 2 \BL^{23} 
 + 7 \BL^{22}
 + 18 \BL^{21} 
 + 48 \BL^{20} 
 + 111 \BL^{19} 
 + 251 \BL^{18} 
 + 498 \BL^{17} 
 + 891 \BL^{16} \\
&\quad 
+ 1368 \BL^{15} 
 + 1847 \BL^{14} 
 + 2132 \BL^{13} 
 + 2150 \BL^{12} 
 + 1853 \BL^{11} 
 + 1395 \BL^{10} 
 + 904 \BL^9 \\
&\quad
+ 522 \BL^8
+ 272 \BL^7
+ 136 \BL^6
+ 66 \BL^5
+ 32 \BL^4
+ 14 \BL^3
+ 6 \BL^2
+ 2 \BL
+1 \\
[\Quot_{\BP^3}(E,4)] &= \BL^{20} + 3\BL^{19} + 13\BL^{18} + 39\BL^{17} + 102\BL^{16} + 202\BL^{15} + 346L^{14} + 480\BL^{13} \\
&\quad+ 581\BL^{12} + 590\BL^{11} + 533\BL^{10} + 415\BL^9 + 297\BL^8 + 187\BL^7 
+ 113\BL^6 + 60\BL^5 \\
&\quad+ 32\BL^4 + 14\BL^3 + 6\BL^2 + 2\BL + 1,
\end{align*}
where the second formula (cf.~\Cref{subsubsec:Quot-P^3}) holds for any locally free sheaf $E$ of rank $3$ on $\BP^3$.

\subsection{Relation with existing work}
Let $X$ be a smooth quasiprojective variety of dimension $n$. The motive of $\Hilb^d(X)$, for $n\leq 2$, admits an explicit closed formula (only involving $\BL$ and $[X]$) for all $d$ (cf.~\Cref{sec:curves-and-surfaces}).
The motives of the Hilbert schemes of $\leq 3$ points can be computed by localisation (see \cite[Rem.~4.5]{BBS} and \cite{Cheah96}). The motive of $\Hilb^4(\BA^3)$ can be easily extracted from the Pfaffian description of the singularity of this Hilbert scheme \cite{Katz-Hilb4,Dimca_Szendroi}. See Zhan \cite{zhan2022punctual} for a recent computation of $[\Hilb^d(\BA^3)_0]$ for $d \leq 5$.\footnote{We were informed by A.~Gautam that he computed the motive $[\Hilb^6(\BA^3)_0]$.}  We are not aware of similar computations for higher dimension or number of points. We refer the reader to \cite{GLMps,GLMHilb,ricolfi2019motive,Fantechi-Ricolfi-motivic} for structural results on motivic generating functions. Variations of the Hilbert scheme of points have been studied, from a motivic point of view, in \cite{BFP19,mozgovoy2019motivic,MR_nested_Quot,double-nested-1, monavari2024hyperquot}. See also \cite{BBS,RefConifold,BR18,DavisonR,Cazzaniga:2020aa,cazzaniga2020higher,Virtual_Quot} for the \emph{virtual} point of view and the link with motivic Donaldson--Thomas theory.

\subsection{Open questions}
We close the introduction proposing a few open problems and conjectures around the topics discussed in this paper.

First of all, note that in \Cref{MAIN-THEOREM-1} the polynomial $\mathsf P_d(t)$ is stated to have degree \emph{at most} $d-2$. In all the examples we know, the coefficient $a^{(d)}_{d-2}$ does not vanish. We therefore propose the following. 

\begin{conjecture}\label{conj:degree-of-P_d}
For every $d>3$, the polynomial $\mathsf P_d(t)$ has degree \emph{exactly} $d-2$.    
\end{conjecture}

Using computer software SAGE and Macaulay2 \cite{sagemath,M2} and the database of partitions \cite{ThePartitionsProjectWebsite,GOVINDARAJAN2013600,Indiani-asintotici}, we were able to verify that \Cref{conj:degree-of-P_d} holds  true for $d \leq 26$.

\smallbreak
Another open problem is the following. It is not known (for $n>2$) whether the motive of the stack $\CCoh^d(\BA^n)_0$ of $0$-dimensional coherent sheaves of length $d$, supported at the origin, is a rational function in $\BL$ (cf.~\cite[Open Problem 5.9]{Fantechi-Ricolfi-motivic}). By \cite[Thm.~1.7]{Huang-Jiang}, this statement is equivalent to the polynomiality in $\BL$ of the motive of $\Hilb^d(\BA^n)_0$. We therefore ask the following.

\begin{question}\label{question-polynomiality}
For which $n$ and $d$ is $[\Hilb^d(\BA^n)_0]$ a polynomial in $\BL$?
\end{question}

It is worth mentioning that the recent work \cite{farkas2024irrational} proves the existence of irrational components on some Hilbert scheme of points $\Hilb^d(\BA^n)$. It would be interesting to explore how this fact is linked to \Cref{question-polynomiality}. A related (and in some sense harder) open question is the following.

\begin{question}\label{conj:L-subalgebra}
For which values of $d$ does the polynomial $\mathsf P_d(t)$ lie in the subalgebra $\BZ[\BL,t]\subset K_0(\Var_{\BC})[t]$? And if $\mathsf P_d(t) \in \BZ[\BL,t]$, is it true that $\Hilb^d(\BA^n)$ has no irrational components for every $n$?
\end{question}

We discuss in \Cref{subsec:indiani} a possible strategy to answer the following question.
\begin{question} \label{ques:howmanydata}
What is the smallest amount of \emph{data} required for computing the motives $[\Hilb^d(\BA^n)_0]$ for every $n\geq 1$ up to a certain $d$?
\end{question}

As a consequence of \Cref{cor:hilbrel} we see that, once we fix an integer $d\geq 1$, the motives $[\Hilb^d(\BA^n)_0]$ for $n\geq d$ can be computed from the first $d-1$ motives $[\Hilb^d(\BA^1)_0]$, $\dots$, $[\Hilb^d(\BA^{d-1})_0]$. As a consequence, the answer to \Cref{ques:howmanydata} is bounded above by $d(d-1)/2$, where `data' is taken to mean precisely the motives of  $[\Hilb^d(\BA^i)_0]$ for $i$ in a suitable range. By the inductive nature of this reasoning, we can in fact rephrase this in a better way, speaking of the \emph{number of new data} required for computing the motives $[\Hilb^d(\BA^n)_0]$ for $n\geq d$, where `new' refers to the assumption of having already computed $[\Hilb^{e}(\BA^n)_0]$ for $1 \leq e < d$. Again, thanks to \Cref{cor:hilbrel}, this number is bounded above by $d-1$.\footnote{In fact, \Cref{prop:easyY} shows that only $d-8$ new data are needed.} Our expectation is that this bound can be approximately halved.

\begin{conjecture} \label{conj:halfdata}
For $d>3$, the \emph{number of new data} required for computing the polynomial $\mathsf P_d(t)$, assuming $\mathsf P_{e}(t)$ is known for $1 \leq e <d$, is bounded above by $\left\lceil\frac{d-1}{2}\right\rceil$.
\end{conjecture}

More explicitly, `new data' refers precisely to the motives   $[\Hilb^d(\BA^i)]$, for $i=1,\ldots, \left\lceil\frac{d-1}{2}\right\rceil$. One should see  \Cref{conj:halfdata} as the motivic refinement of the main result of  \cite{GOVINDARAJAN2013600}, which says that the number of \emph{new data} needed to compute the number of $(n-1)$-dimensional partitions of size $d$ is bounded above by $\left\lceil\frac{d-1}{2}\right\rceil$.\footnote{Recall that the number of higher dimensional partitions are precisely the topological Euler characteristics $\chi(\Hilb^d(\BA^n))$ by a classical torus localisation argument.}

Finally, we propose in \Cref{sec:higher-dim-partitions} a qualitative description for the discrepancy between the two generating functions
\[
\chi\mathsf{Hilb}_{n,0}(t) = \prod_{d>0}\,\left(1-t^d\right)^{-\chi(\Omega^n_d)},\qquad  \prod_{d>0}\,\left(1-t^d\right)^{-\binom{d+n-3}{n-2}},
\]
where the second one was MacMahon's proposed product formula for the generating function $\Pi_n(t)$ of $(n-1)$-dimensional partitions. 

\begin{conjecture}[{\Cref{conj:erroreMac}}]\label{conj:chi-omega-intro}
For every $d\ge 6 $ there exists an \emph{irreducible} polynomial $r_d(t)\in \BQ[t]$  of degree $d-6$ such that
\[
\binom{d+n-3}{n-2}-\chi(\Omega^n_d) =\binom{n}{4}r_d(n)
\]
for all $n\ge 1$. 
\end{conjecture}

We prove the conjecture for $d\leq 26$ via computer software and the database of partitions \cite{ThePartitionsProjectWebsite,GOVINDARAJAN2013600,Indiani-asintotici}.

\subsection*{Acknowledgements}
We would like to thank Eugenio Bellini, Giorgio Gubbiotti, Tony Iarrobino,  Joachim Jelisiejew, Alexander Kuznetsov,  Paolo Lella, Danilo Lewanski, Alessio Sammartano, Israel Vainsencher, Tanguy Vernet, for useful discussions.  S.M. is supported by the Chair or Arithmetic Geometry, EPFL.  R.M. is partially supported by the PRIN project 2022L34E7W ``Moduli Spaces and Birational Geometry''. A.R. is partially supported by the PRIN project 2022BTA242 ``Geometry of algebraic structures: moduli, invariants, deformations''. All the authors are members of the GNSAGA - INdAM.

\section{Background material}

\subsection{Grothendieck ring of varieties}\label{sec:K0(Var)}

The Grothendieck ring of complex varieties $K_0(\Var_{\BC})$ is the free abelian group generated by isomorphism classes $[X]$ of $\BC$-varieties, modulo the \emph{scissor relations}, namely the relations 
\[
[X] = [Y] + [X \setminus Y]
\]
whenever $Y \subset X$ is a closed subvariety. The group structure agrees, on generators, with the disjoint union of varieties and has neutral element $0 = [\emptyset]$. The fibre product defines a ring structure on $K_0(\Var_{\BC})$, with neutral element $1 = [\Spec \BC]$. Every constructible subset $Z \subset X$ of a variety $X$ has a well-defined motivic class $[Z]$, independent upon the decomposition of $Z$ into locally closed subsets.  One could also perform a similar construction with schemes (or algebraic spaces) over $\BC$ instead of varieties, but the resulting Grothendieck rings would come out isomorphic to $K_0(\Var_\BC)$. In particular, the motive of a scheme agrees with the motive of its reduction.

The classes $[X]$ of honest $\BC$-varieties in $K_0(\Var_{\BC})$ are called \emph{effective}. A key example is the \emph{Lefschetz motive}
\[
\BL = [\BA^1]\in K_0(\Var_{\BC}).
\]

The ring $K_0(\Var_\BC)$ has the following universal property. Suppose $S$ is a ring and $w(-)$ is an $S$-valued invariant of algebraic varieties, such that 
\begin{itemize}
\item [$\circ$] $w(\pt)=1$,
\item [$\circ$] $w(\emptyset) = 0$,
\item [$\circ$] $w(X\times Y) = w(X)w(Y)$ for every two varieties $X$ and $Y$, and
\item [$\circ$] $w(X) = w(Y) + w(X\setminus Y)$ for every variety $X$ and closed subvariety $Y \subset X$. 
\end{itemize}  
Then there is precisely one ring homomorphism $w\colon K_0(\Var_\BC) \to S$ sending an effective class $[X]$ to $w(X) \in S$. Ring homomorphisms out of $K_0(\Var_{\BC})$ are called \emph{motivic measures}, or generalised Euler characteristics, or realisations \cite{DenefLoeser1,LooijengaMM}. The key examples are the following. If $Y$ is a complex variety, one can consider
\begin{itemize}
\item [\mylabel{mot-measure-i}{(i)}] The \emph{Hodge characteristic}
\[
\chi_{\mathrm{Hodge}}(Y) = \sum_{i\geq 0}\,(-1)^i [\HH^i_c(Y,\BQ)] \,\in\, K_0(\HS),
\]
where the vector space $\HH^i_c(Y,\BQ)$ is equipped with Deligne's mixed Hodge structure \cite{Deligne-Hodge-II,Deligne-Hodge-III}.
\item [\mylabel{mot-measure-ii}{(ii)}] The \emph{Hodge--Deligne polynomial}
\[
\mathsf E(Y;u,v)=\sum_{p,q}\,h^{p,q}(\HH^{p+q}_c(Y,\BQ))(-u)^p (-v)^q \,\in\,\BZ[u,v],
\]
where $h^{p,q}(\HH^i_c(Y,\BQ))$ is the dimension of the $(p,q)$-component of the mixed Hodge structure $\HH^i_c(Y,\BQ)$.
\item [\mylabel{mot-measure-iii}{(iii)}] The \emph{weight polynomial} ($u,v \mapsto z^{1/2}$)
\[
\mathsf w(Y,z^{1/2}) = \sum_{p,q}\,h^{p,q}(\HH^{p+q}_c(Y,\BQ))(-z^{1/2})^{p+q}\,\in\,\BZ[z^{1/2}],
\]
which coincides with the signed \emph{Poincar\'e polynomial}
\[
\mathsf p(Y,-z^{1/2}) = \sum_{i\geq 0}\,\dim_{\BQ}\HH^i(Y,\BQ) (-z^{1/2})^{i}
\]
if $Y$ is smooth and projective.
\item [\mylabel{mot-measure-iv}{(iv)}] The \emph{compactly supported Euler characteristic} ($z^{1/2} \mapsto 1$)
\[
\chi_c(Y) = \mathsf w(Y,1) = \sum_{i\geq 0} \,(-1)^i \dim_{\BQ} \HH^i_c(Y,\BQ)\,\in\,\BZ,
\]
which can be shown to equal the topological Euler characteristic $\chi(Y)$ \cite[p.~95 and pp.~141--142]{fulton_toric}.
\end{itemize}
\begin{remark}\label{rmk:pre-lambda-maps}
It is possible to obtain homomorphisms
\[
K_0(\Var_{\BC}) \to K_0(\HS) \to \BZ[u,v] \to \BZ[z^{1/2}] \to \BZ
\]
with slightly different specialisations, but with our choice we ensured that all specialisations are homomorphisms of pre-$\lambda$-rings, just as in \cite[Sec.~1.4]{DavisonR}.
Note that these homomorphisms, the way they have been defined, send
\[
\BL \mapsto [\HH^2_c(\BA^1,\BQ)]\mapsto uv \mapsto z\mapsto 1.
\]
\end{remark}
The key features of $K_0(\Var_{\BC})$ to keep in mind, which will be used constantly throughout, are the following:
\begin{itemize}
    \item [\mylabel{key-a}{(a)}] If $X \to B$ is a bijective morphism, then $[X] = [B]$ in $K_0(\Var_{\BC})$.
    \item [\mylabel{key-b}{(b)}] If $X \to B$ is a Zariski locally trivial fibration with fibre $F$, then $[X] = [B][F]$ in $K_0(\Var_{\BC})$.
\end{itemize}
As a special case of \ref{key-a}, consider a scheme $X$ along with locally closed subschemes $Z_1,\ldots, Z_r$ such that the disjoint union of the locally closed immersions $Z_1\amalg \cdots\amalg Z_r \to X$ is a bijective morphism (in this case we say that $Z_1,\ldots, Z_r$ form a \emph{stratification} of $X$, and we write $X = \coprod_i Z_i$ with a slight abuse of notation). One then has an identity $[X] = \sum_{1\leq i\leq r}[Z_i]$ in $K_0(\Var_{\BC})$.

\begin{example}
The class of the Grassmannian $\Gr(k,n)$ of $k$-planes in $\BC^n$ is
\begin{equation}\label{eqn:motive-grassmannian}
[\Gr(k,n)] = \frac{[n]_{\BL}!}{[k]_{\BL}![n-k]_{\BL}!}, \qquad [a]_{\BL}! = \prod_{k=1}^a \,\left(\BL^k-1\right).
\end{equation}
Throughout we adopt the conventions $[\Gr(k,n)] = 0 = [\BP^e]$ if $k>n$ or $e<0$.
\end{example}
Every complex variety $X$ has an associated \emph{zeta function} $\zeta_X(t)$, introduced by Kapranov \cite{Kapranov_rational_zeta}. It is defined as
\[
\zeta_X(t) = \sum_{d \geq 0}\,[\Sym^d(X)]t^d \in 1+tK_0(\Var_{\BC})\llbracket t\rrbracket.
\]
It satisfies, for every integer $n \geq 0$, the identities
\[
\zeta_{\BA^n \times X}(t) = \zeta_X(\BL^n t),\qquad \zeta_{X \amalg Y}(t) = \zeta_X(t) \zeta_Y(t),
\]
from which one deduces
\begin{align*}
\zeta_{\BA^n}(t) = \frac{1}{1-\BL^nt}, \qquad \zeta_{\BP^n}(t) = \prod_{i=0}^n\frac{1}{1-\BL^it}.
\end{align*}
\subsection{Punctual Hilbert schemes}
We exploit the first part of this subsection to set up some notation that will be used throughout. First of all, we set $R=\BC[x_1,\ldots,x_n]$ and we denote by $\Fm = (x_1,\ldots,x_n) \subset R$ the maximal ideal of the origin $0 \in \BA^n = \Spec R$.

For each $d \in \BZ_{\geq 0}$, the \emph{punctual Hilbert scheme} $\Hilb^d (\BA^n)_0$ is the fibre of the \emph{Hilbert--Chow morphism}
\[
\begin{tikzcd}
\Hilb^d (\BA^n) \arrow{r} & \Sym^d(\BA^n)
\end{tikzcd}
\]
over the point $d\cdot 0 \in \Sym^d(\BA^n)$.
It is a projective closed subscheme of $\Hilb^d(\BA^n)$, isomorphic to $\Hilb^d(\Spec \widehat{\OO}_{\BA^n,0})$, see e.g.~\cite{Fantechi-Ricolfi-motivic}. If $Z \subset \BA^n$ is a $0$-dimensional closed subscheme of length $d$, defined by an ideal $I \subset R$, we will denote by $[Z]$ or by $[I]$ the corresponding point of $\Hilb^d(\BA^n)$. A subscheme $Z=\Spec R/I \subset \BA^n$ corresponding to a point $[Z] \in \Hilb^d (\BA^n)_0$ will be called a \emph{fat point}. 
Algebraically, it corresponds to a local Artinian $\BC$-algebra $(A,\Fm_A)$ with residue field $\BC$, namely
\begin{equation}\label{eqn:local-artin}
(A,\Fm_A) = (R/I,\Fm/I).
\end{equation}

\subsection{Hilbert--Samuel functions}\label{sec:HS-functions}
We recall now the notion of the Hilbert--Samuel function. It is a classical invariant attached to  graded modules. We refer the reader to \cite[Sec.~5.1]{EISENBUD} and \cite[Lect.~20]{HARRIS} for more details.

We endow the polynomial ring $R=\BC[x_1,\ldots,x_n]$ with the standard grading $\deg(x_k)=1$, for $k=1,\ldots,n$. The $i$-th graded piece of $R$ will be denoted $R_i$.

\begin{definition}
Let $M=\bigoplus_{i\in\BZ} M_i$ be a finite graded $R$-module. The \emph{Hilbert--Samuel function} $\bh_M$ attached to $M$ is the function
\[
\begin{tikzcd}[row sep = tiny]
\BZ\arrow[r,"\bh_M"]& \BN\\
        i \arrow[r,mapsto] & \dim_{\BC} M_i. 
\end{tikzcd}
\]
Let $I\subset R$ be the ideal of a fat point. Form the local Artinian $\BC$-algebra $(A,\Fm_A)$ as in \eqref{eqn:local-artin}. The \emph{Hilbert--Samuel function} of $A$ (that we shall denote $\bh_I$ throughout) is defined to be the Hilbert--Samuel function of the graded $R$-module
\begin{equation}\label{eqn:grading-A-module}
\mathsf{gr}_{\Fm_A}(A) =\bigoplus_{i\geq 0} \mathfrak m_A^{i}/\mathfrak m_A^{i+1}.
\end{equation}
\end{definition}
 
We record in the next remark a few useful observations, that we exploit to set up some more notation. 

\begin{remark}\label{rmks-on-HS}
Suppose $A=R/I$ corresponds to a fat point $Z \subset \BA^n$. The function $\bh_I$ has the following properties:
\begin{itemize}
\item [\mylabel{HF1}{(i)}] It vanishes for almost all values (i.e.~it has finite support), and satisfies $\bh_{I}(0)=1$. Therefore we represent $\bh_I$ as a finite string of integers, namely $\bh_I = (1,h_1,\ldots,h_t)$, where $t$ is the largest index such that $h_t \neq 0$.
\item [\mylabel{HF2}{(ii)}] We have $\bh_{I}(1) = \dim_{\BC} \Fm_A/\Fm_A^2 = \embdim(Z)$, where $\embdim(Z)$ denotes the \emph{embedding dimension} of $Z$, namely the smallest dimension of a scheme containing $Z$ and smooth at the unique point of $Z$.  
\item [\mylabel{HF3}{(iii)}] The length $\chi(\OO_Z)=\dim_{\BC}A$ can be computed as the finite sum $\lvert \bh \rvert = 1+\sum_{i > 0} h_i$.
\item [\mylabel{HF4}{(iv)}] Let $\In I \subset R$ be the initial ideal of $I$ (i.e.~the ideal generated by the initial forms of the polynomials in $I$, cf.~\cite[Sec.~5.1]{EISENBUD}). Then the isomorphism of graded $R$-algebras $\mathsf{gr}_{\Fm_A}(A)\cong R/\In I$ shows that $\bh_{\In I} = \bh_I$.
\end{itemize}
\end{remark}

\subsection{Hilbert--Samuel strata}\label{subsec:HS-strata}
We are ready to introduce the \emph{Hilbert--Samuel strata}. We refer the reader to \cite{multigraded} for more details.

\begin{definition}
Given a function $\bh\colon\BZ\rightarrow \BN$ with finite support, and a nonnegative integer $n$, the \emph{Hilbert--Samuel stratum} $H_{\bh}^n$  is the (possibly empty) locally closed subset
\[
H_{\bh}^n=\Set{[I]\in \Hilb^{\left|\bh\right|} (\BA^n)_0 | \bh_I=\bh}\subset \Hilb^{\lvert \bh \rvert}(\BA^n)_0.
\]
We endow it with the reduced induced subscheme structure.
\end{definition}

The \emph{Hilbert--Samuel stratification}
\[
\Hilb^d (\BA^n)_0 = \coprod_{\lvert \bh \rvert = d} H_{\bh}^n
\]
will be our main tool in the computation of the motive of $\Hilb^d (\BA^n)_0$.

The standard scaling action $\BG_m\times \BA^n\rightarrow \BA^n$ lifts to the Hilbert scheme $\Hilb^{d}(\BA^n)_0$. 
Under this action, the Hilbert--Samuel strata are $\BG_m$-invariant locally closed subsets of the punctual Hilbert scheme. Hence, the action of the torus $\BG_m$ restricts to an action on $H_{\bh}^n$, for all $\bh\colon \BZ\rightarrow \BN$. Set theoretically, the $\BG_m$-fixed locus of the action on $H_{\bh}^n$ consists of \emph{homogeneous} ideals (with fixed Hilbert--Samuel function $\bh$). This locus defines a closed subset \cite[Prop.~1.5]{multigraded}
\[
\OH_{\bh}^n \subset H_{\bh}^n,
\]
which we endow with the reduced induced subscheme structure. 

\begin{prop}[{\cite[p.~773]{8POINTS}}]
Fix $\bh = (1, h_1, \ldots , h_t)$ and $n>0$. There is a morphism
\begin{equation}\label{initial-ideal-map}
    \begin{tikzcd}
        H_{\bh}^n \arrow[r,"\pi_{\bh}^n"] &\OH_{\bh}^n
    \end{tikzcd}
    \end{equation}
which, on closed points, sends an ideal to its initial ideal.
\end{prop}

The next lemma gives sufficient conditions on $\bh$ ensuring either that $\pi^n_{\bh}$ has affine spaces as fibres or that it is a vector bundle.
 
\begin{lemma}[{\cite[Lemmas 3.2, 3.3]{Hilb_11}}]\label{lemma:tech1}
Fix $\bh = (1, h_1, \ldots , h_t)$ and assume that
\[
h_i = \bh_{R}(i), \qquad i = 1, \ldots, t - 3,
\]
Then the fibres of $\pi_{\bh}^{h_1}\colon H_{\bh} ^{h_1}\to \OH_{\bh}^{h_1}$ are affine spaces. 
Moreover, if in addition one has $h_{t-2} = \bh_{R}(t-2)$, then $\pi_{\bh}^{h_1}$ is a vector bundle of rank 
\[
\rk H_{\bh}^{h_1} = h_t \left(  \bh_{R}(t-1) - h_{t-1}\right).
\]
\end{lemma}

The next result relates the varieties $\OH_{\bh}^{h_1}$ and $\OH_{\bh}^n$.

\begin{lemma}[{\cite[Prop.~3.1]{8POINTS}}]\label{lemma:tech2}
Fix $\bh = (1, h_1, \ldots , h_t)$. Then, for every $n\geq h_1$, the map 
\begin{equation}\label{eqn:rho-map}
\begin{tikzcd}
\OH_{\bh}^n\arrow[r,"\rho^n_{\bh}"]&\Gr(n-h_1,R_1),
\end{tikzcd}
\end{equation}
sending a homogeneous ideal $[I]\in\OH_{\bh}^n$ to its linear part $[I_1]\in \Gr(n-h_1,R_1)$ is a Zariski locally trivial fibration with fibre $\OH_{\bh}^{h_1}$.
\end{lemma}
 
While proving \Cref{thm:motive-Y-stratum} we will show that it is also possible to relate the Hilbert--Samuel strata $H_{\bh}^{h_1}$ and $H_{\bh}^n$.

Even though the map $\pi_{\bh}^n$ from \eqref{initial-ideal-map} is not well-behaved in general, post-composition with $\rho_{\bh}^n$ yields a fibration, as we prove in the following lemma, which will be used in the proof of \Cref{thm:motive-Y-stratum}.

\begin{lemma}\label{tau-is-ZLT}
 Fix $\bh = (1, h_1, \ldots , h_t)$. The composite morphism
\[
\begin{tikzcd}
\tau\colon H^n_{\bh} \arrow{r}{\pi^n_{\bh}} & \OH^n_{\bh} \arrow{r}{\rho^n_{\bh}} & \Gr(n-h_1,R_1)
\end{tikzcd}
\]
is Zariski locally trivial.
\end{lemma}

\begin{proof}
This can be proved along the same lines of \cite[Prop.~3.1]{8POINTS}. We provide full details for the sake of completeness.

Set $G=\Gr(n-h_1,R_1)$. Fix a point $p \in G$, corresponding to a linear subspace $V \subset R_1$, and an open neighbourhood $p \in U \subset G$ trivialising the inclusion
\begin{equation*}
\mathscr S \subset \OO_{G} \otimes_{\BC}R_1,
\end{equation*}
where $\mathscr S$ denotes the tautological subbundle. If $\Sigma \to G$ denotes the total space of $\mathscr S$, the trivialisation yields a commutative diagram
\begin{equation}\label{eqn: comm}
    \begin{tikzcd}[column sep=large,row sep=large]
\Sigma|_U\arrow[d]\arrow[r, hook] & U\times R_1\arrow[d,swap,"\psi"'] \\
   U\times V \arrow[r,swap, hook,"\id_U \times \iota"'] & U\times R_1
\end{tikzcd}
\end{equation}
where $\iota \colon V \into R_1$ is the natural inclusion, $\Sigma|_U = \Sigma \cap (U\times R_1)$ and the vertical maps are isomorphisms. 
Note that $\psi$ induces an isomorphism 
\[
\begin{tikzcd}
\alpha=\Sym^\bullet_{\OO_U} \psi\colon \OO_U[x_1,\ldots,x_n] \arrow{r}{\sim} & \OO_U[x_1,\ldots,x_n].
\end{tikzcd}
\]
Set $W=\tau^{-1}(U)$. It is an open subscheme of $H^n_{\bh}$. Let $\mathscr J \subset \OO_W[x_1,\ldots,x_n]$ be the ideal sheaf corresponding to the open immersion $W \into H^n_{\bh}$. Pulling $\alpha$ back along $\tau\colon W \to U$ we obtain an isomorphism
\[
\begin{tikzcd}
\beta = \tau^\ast \alpha\colon \OO_W[x_1,\ldots,x_n] \arrow{r}{\sim} & \OO_W[x_1,\ldots,x_n].
\end{tikzcd}
\]
The ideal $\beta^{-1}(\mathscr J) \subset \OO_W[x_1,\ldots,x_n]$ is $W$-flat and thus defines a morphism $f\colon W \to H^n_{\bh}$. By the commutativity of \eqref{eqn: comm}, the composition $\tau \circ f$ is constant with value $p \in G$, therefore $f$ factors through the fibre $\tau^{-1}(p)$. We have obtained a morphism 
\[
\begin{tikzcd}[column sep=large]
    W=\tau^{-1}(U) \arrow{r}{(\tau,f)} & U \times \tau^{-1}(p),
\end{tikzcd}
\]
which is the required trivialisation.
\end{proof}
    
\subsection{Apolarity}
A powerful tool for parametrising families of ideals is \emph{Macaulay duality}, also known as \emph{apolarity} \cite{Macaulay}. We now present some results on apolarity that we shall need in \Cref{sec:consequences of thm A,sec:expectation}. The reader can consult \cite{8POINTS,Hilb_11,EMSALEM,Geramita,Iarrobook} for more details. Working in characteristic $0$ is crucial in this section.

Let us set
\[
R=\BC[x_1,\ldots,x_n], \quad R^*=\BC[y_1,\ldots,y_n].
\]
We view $R^*$ as an $R$-module via the action 
\[
\begin{tikzcd}[row sep =tiny]
    R\times R^*\arrow[r]& R^* \\
    (x_1^{\alpha_1}\cdots x_n^{\alpha_n},p(y_1,\ldots,y_n))\arrow[r,mapsto] & \displaystyle\frac{\partial^{\sum_{i=1}^n\alpha_i} }{\partial^{\alpha_1} y_1\cdots \partial^{\alpha_n} y_n}p,
\end{tikzcd}\]
where $\alpha_i\in\BZ_{\geq 0}$ for $i=1,\ldots,n$. This induces, for every  $k\ge 0$, a perfect pairing
\[
\begin{tikzcd}
R_k\times R_k^* \arrow{r} & R_0^* = \BC
\end{tikzcd}
\]
and, consequently, a notion of orthogonality. 
\begin{definition}[{\cite{Hilb_11}}]
 An \emph{inverse system} is a graded vector subspace $T\subset R^*$ closed under differentiation. If $S\subset R^*$ is a finite subset then the inverse system generated by $S$ is the smallest subspace $\langle S\rangle\subset R^*$ containing $S$ and closed under differentiation. The \emph{apolar ideal} attached to $T$ is  
\[
T^\perp=\Set{r\in  R  | r\cdot T =0 }\subset R.
\]
If $I\subset R$ is a homogeneous ideal, its associated inverse system is 
\[
I^\perp =\Set{r^*\in  R^*  | I \cdot r^* =0 }\subset R^*.
\]
\end{definition} 
\begin{remark}
Notice that if $I\subset R$ is a homogeneous ideal, then $I^\perp\subset R^*$ is a graded subspace closed under differentiation. Conversely, every graded vector subspace $T\subset R^*$ closed under differentiation is orthogonal to the homogeneous ideal $T^\perp\subset R$.
Moreover, if $V\subset R^*_j$ is a vector subspace, then
\[
\dim_\BC (V^\perp)_j =\dim_\BC {R^*_j}/ V,
\]
which yields an isomorphism of graded vector spaces $R/I\cong I^\perp$, see \cite[Sec.~2]{8POINTS}.
\end{remark}
 
\begin{prop}[{\cite[Lemma 2.12]{Iarrobook}}, {\cite{CARLINI}}] \label{prop:gorapolar}
Let $f \in  R^*$ be a homogeneous form of degree
$d\ge 0$. Then the Hilbert--Samuel function of the ideal $\langle f\rangle^\perp$ is
symmetric, i.e.~$\bh_{\langle f\rangle^\perp}(i) = \bh_{\langle f\rangle^\perp}(d - i)$ for all $i=0,\ldots,d$. 

Suppose also that $\bh_{\langle f\rangle^\perp}$ has the form
\[
\bh_{\langle f\rangle^\perp}=(1,h_1,h_2,\ldots,k,1),
\]
for some $k\ge 1$. In particular, in this case $d\ge2$. Then, there are linearly
independent linear forms $\ell_1,\ldots,\ell_k\in R^*_1 $ and a homogeneous form $g$ such that
$f = g(\ell_1,\ldots,\ell_k)$.
\end{prop}

\section{Stratification by embedding dimension}\label{sec:strat-emb-dim}

The goal of this subsection is to analyse a key stratification of the punctual Hilbert scheme, defined entirely in terms of the embedding dimension.

Fix three nonnegative integers $n,d,k$. Consider the subset 
\[
Y^n_{k,d}=\Set{[Z]\in \Hilb^d(\BA^n)_0 | \embdim (Z)=k}\subset \Hilb^d(\BA^n)_0,
\]
where $\embdim(-)$ denotes the embedding dimension, cf.~\Cref{rmks-on-HS}\ref{HF2}. We extend the definition of $Y^n_{k,d}$ allowing possibly negative integers $n,k,d$, but we declare $Y^n_{k,d}=\emptyset$ as soon as at least one of these integers is negative.

\begin{remark}\label{rmk:small-cases-of-Y}
The locus $Y_{k,d}^n$ is empty when $k\geq d$ and when $k=0$ and $d\neq 1$. Moreover, one has $Y^n_{0,1}=\pt=Y^{d-1}_{d-1,d}$, 
for all $n\geq 0$ and $d>0$. Notice that the unique point of $Y_{d-1,d}^{d-1} = H_{(1,d-1)}$ corresponds to the ideal $(x_1,\ldots,x_{d-1})^2 \subset \BC[x_1,\ldots,x_{d-1}]$.
\end{remark}

\begin{lemma} \label{lem:YnkdLocallyClosed}
Fix three integers $n,k,d$. Then $Y^n_{k,d}\subset \Hilb^d(\BA^n)_0$ is locally closed.
\end{lemma}

\begin{proof}
Let $B$ be a scheme and $\CZ \subset \BA^n \times B$ a $B$-flat family of fat points. We claim that the locus of points $b \in B$ such that $\embdim(\CZ_b) = k$ is locally closed. Let $\Fm_b \subset \OO_{\CZ_b}$ be the maximal ideal of the coordinate ring of $\CZ_b\subset \BA^n$. The function $b \mapsto  \embdim(\CZ_b)$ is upper semi-continuous, since $\embdim(\CZ_b)=\dim_{\BC} \Fm_b / \Fm_b^2$ is the minimal number of generators of $\Fm_b$.
\end{proof}

We endow $Y^n_{k,d}$ with the reduced subscheme structure, to obtain a stratification 
\begin{equation}\label{eq:stratification}
    \Hilb^d(\BA^n)_0 = \begin{cases}
    Y^{n}_{0,1} &\mbox{ if }d=1,\\
    \coprod_{k=1}^{d-1}Y^{n}_{k,d} & \mbox{ if }d>1.
\end{cases}
\end{equation}
Moreover, notice that the strata $Y_{k,d}^n\subset \Hilb^d(\BA^n)_0$ admit a Hilbert--Samuel stratification themselves, namely
\begin{equation}\label{Y-stratification}
Y_{k,d}^n=\coprod_{\substack{\lvert \bh\rvert = d \\ h_1=k}} H_{\bh}^n.
\end{equation}

Before establishing a formula for the motive of $Y^n_{k,d}$, we prove the following lemma.

\begin{lemma}\label{lemma:standard-basis}
Let $I \subset R=\BC[x_1,\ldots,x_n]$ be the ideal of a fat point such that $\In I \supset (x_1,\ldots,x_k)$ for some $k \geq 1$. Then there exist polynomials $p_1,\ldots,p_k \in S= \BC[x_{k+1},\ldots,x_n]$ and an ideal $J \subset S$ such that 
\[
I = J\cdot R+(x_1+p_1,\ldots,x_k+p_k).
\]
\end{lemma}

\begin{proof}
By the results in \cite{tandardBasis} about the existence of standard bases, $I$ can be written as
\[
I=J+(x_1+\widetilde p_1,\ldots,x_k+\widetilde p_k)
\]
where $\widetilde p_i \in \BC[x_{i+1},\ldots,x_n]$ and $J\subset S$. The polynomials $p_i$ are recursively defined as $p_k = \widetilde p_k$ and $p_{i-1} = \widetilde p_{i-1}|_{x_{i}=p_i}$.
\end{proof}

\begin{theorem}\label{thm:motive-Y-stratum}
Fix $n,d,k>0$. There is an identity 
\[
[Y^n_{k,d}] = [\Gr(k,n)]\BL^{(n-k)(d-k-1)}[Y^k_{k,d}]
\]
in $K_0(\Var_{\BC})$.
\end{theorem} 
\begin{proof}
The case $d=1$ yields the trivial identity $0 = 0$, so we assume $d>1$.
Thanks to the stratification \eqref{Y-stratification}, it suffices to prove  the identity
\begin{equation}\label{claim-decomposition}
[H^n_{\bh}]=[\Gr(n-h_1,R_1)] [H^{h_1}_{\bh}]\BL^{(n-h_1)(\lvert \bh\rvert-h_1-1)}
\end{equation}
for every function $\bh\colon \BZ \rightarrow \BN$ with finite support. Since the composition $\tau = \rho^n_{\bh} \circ \pi^n_{\bh}$ is Zariski locally trivial by \Cref{tau-is-ZLT}, we have the identity of motives
\[
[H^n_{\bh}] = [\Gr(n-h_1,R_1)][F]= [\Gr(h_1,n)][F],
\]
where $F$ is a fibre of $\tau$.
To prove the sought after relation \eqref{claim-decomposition}, it therefore suffices to show that 
\[
[F] = [H_{\bh}^{h_1}]\BL^{(n-h_1)(\lvert \bh\rvert-h_1-1)}.
\]
Fix a point $[V] \in \Gr(n-h_1,R_1)$, corresponding to an affine subspace $\BA_V\subset \BA^n$ of dimension $h_1$. 
We may assume its defining ideal is $I_V=(x_1,\ldots,x_{n-h_1})$. Let us denote by $F$ the fibre of $\tau \colon H^n_{\bh} \to \Gr(n-h_1,R_1)$ over $[V]$.

Define $G \subset F$ to be the reduced closed subscheme parametrising 0-dimensional subschemes with support in $\BA_V\subset \BA^n$, or equivalently parametrising ideals in the fiber $F$ containing $I_V$.

Let $[I] \in G$ be a point.
Let us fix monomials $m_1,\ldots, m_{\lvert \bh\rvert -h_1-1}\in R_{\ge 2}$ such that the image of 
\[
   B= \Set{1,x_{n-h_1 +1},\ldots,x_{n} , m_1,  \ldots, m_{\lvert \bh \rvert -h_1-1} } \subset R
\]
in $R/I$ via the canonical projection is a $\BC$-linear basis. By \cite[Ch.~18]{CCA}, the subset
\[
U_B=\Set{K\subset R | B \mbox{ projects onto a basis of } R/K}\subset \Hilb^{\lvert \bh \rvert}(\BA^n)_0
\]
defines a Zariski open neighbourhood of $[I]$ in the punctual Hilbert scheme. Define the open subsets $F_{B} = F\cap U_B \subset F $ and $G_{B} = G\cap U_B \subset G$. Consider the action  
\[
\begin{tikzcd}
\BG_a^{(n-h_1)({\lvert \bh \rvert}-h_1-1)}\times\BA^n \arrow[r,"\phi_B"]&\BA^n
\end{tikzcd}
\]
induced by the morphism
\begin{equation}\label{eq:actionFv}
    \begin{tikzcd}[row sep=tiny]
\BC[x_1,\ldots,x_n] \arrow[r]&\BC\left[x_1,\ldots,x_n,\left(u_1^{(j)},\ldots, u_{{\lvert \bh \rvert}-h_1-1}^{(j)}\right)_{j=1}^{n-h_1} \right]\\
x_j \arrow[r,mapsto] & \left\{\begin{matrix}
    {\displaystyle x_j+\sum_{i=1}^{{\lvert \bh \rvert}-h_1-1} u_i^{(j)} m_i  }& \mbox{if }j\le n-h_1,  \\
    {\displaystyle x_j} &\mbox{otherwise},
\end{matrix}\right.
\end{tikzcd}
\end{equation}
and define the product map
 \[
    \begin{tikzcd}[row sep = tiny , column sep= large]
        G_{B}\times \BG_a^{(n-h_1)({\lvert \bh \rvert}-h_1-1)}\times\BA^n\arrow[rr,"\Phi_B=\id\times\phi_B"] & &
       G_{B}\times \BA^n. 
    \end{tikzcd}
    \]
    Let $\OI_B\subset \OO_{G_{B}\times \BA^{n}} $  be the restriction of the universal ideal sheaf on $\Hilb^{\lvert \bh \rvert} (\BA^n)\times \BA^{n}$ to $G_{B}\times \BA^{n}$, and set $\OF_B= \Phi_B^*\OI_B$. Since $\OF_B$ is a  flat family over $  G_{B}\times \BG_a^{(n-h_1)({\lvert \bh \rvert}-h_1-1)}$ of 0-dimensional closed subschemes of $\BA^n$, it induces a morphism 
    \[
\begin{tikzcd}
    G_{B}\times \BG_a^{(n-h_1)({\lvert \bh \rvert}-h_1-1)} \arrow[r]& \Hilb^{\lvert \bh \rvert}\BA^n,
\end{tikzcd}
\]
and, by construction, a factorisation
\[
\begin{tikzcd}
G_B\times \BG_a^{(n-h_1)({\lvert \bh \rvert}-h_1-1)}\arrow{r}{\Psi_B} & F_B\arrow[hook]{r} & \Hilb^{\lvert \bh \rvert}\BA^n.
\end{tikzcd}
\]
The map $\Psi_B$ sends an ideal $([I],z)$ to the pullback of $I$ along the isomorphism induced by $z$ via \eqref{eq:actionFv}.
 
We now show that $\Psi_B$ is a bijection on points, which will imply the identity of motives
\begin{align}\label{eqn: motive F G local}
    [F_B]=[G_B] \BL^{(n-h_1)({\lvert \bh \rvert}-h_1-1)}.
\end{align}
Surjectivity of $\Psi_B$ follows directly from \Cref{lemma:standard-basis} (taken with $k=n-h_1$). For injectivity, assume that $\Psi_B([I_1],z_1)=\Psi_B([I_2],z_2)$, i.e. 
\begin{multline*}
     J_1\cdot R+\left(x_1+\sum_{i=1}^{{\lvert \bh \rvert}-h_1-1} a_{i,1}^{(1)} m_i ,\ldots,x_{n-h_1}+\sum_{i=1}^{{\lvert \bh \rvert}-h_1-1} a_{i,1}^{(n-h_1-1)} m_i\right)=\\J_2\cdot R+\left(x_1+\sum_{i=1}^{{\lvert \bh \rvert}-h_1-1} a_{i,2}^{(1)} m_i ,\ldots,x_{n-h_1}+\sum_{i=1}^{{\lvert \bh \rvert}-h_1-1} a_{i,2}^{(n-h_1-1)} m_i\right)
\end{multline*}
for some ideals $J_l \subset \BC[x_{n-h_1+1},\ldots,x_n]$ and some $a_{i,l}^{(j)}\in \BC $, for $j=1, \dots, n-h_1 $, and $i=\lvert \bh \rvert -h_1-1 $ and $l=1, 2$.  It easily follows that $J_1=J_2$, and that  
\[
\sum_{i=1}^{{\lvert \bh \rvert}-h_1-1} \left(a_{i,1}^{(j)}-a_{i,2}^{(j)}\right) m_i  \in J_1, \quad \mbox{ for } j=1, \dots, n-h_1.
\]
However, by our assumption, the quotients
\[
R/{I_1}\cong \BC[x_{n-h_1+1, \ldots, n}]/{J_1}
\]
are generated by the monomials in the basis $B$, which implies that 
$a_{i,1}^{(j)}=a_{i,2}^{(j)} $, for all  $j=1, \dots, n-h_1 $, and $i=\lvert \bh \rvert -h_1-1$.

By \cite[Ch.~18]{CCA}, the subsets $\set{U_B}_B$ form an open cover of $\Hilb^{\lvert \bh \rvert}(\BA^n)_0$, which implies that $\set{F_B}_B$ (resp.~$ \set{G_B}_B$) form an open cover of $F$ (resp. of $G$). Combined with \eqref{eqn: motive F G local}, we conclude that 
\[
    [F]=[G]\BL^{(n-h_1)({\lvert \bh \rvert}-h_1-1)}.
\]
To conclude the proof, notice that, by construction of $G$, the association $[I] \mapsto [I/I_V]$ defines a bijective morphism $
G\to H_{\bh}^{h_1}$,
which implies the required identity of motives $[G]=[H_{\bh}^{h_1}]$.
\end{proof}

\section{Motivic Grassmannian identities}\label{sec:motives-of-grass}
This section is devoted to proving a series of identities, used in several places in this paper, involving motivic classes of Grassmannians, and  
could be independently regarded as \emph{$\BL$-analogues} of numerical identities involving binomial coefficients. We omit the proofs of the results standard in the literature. 

\begin{lemma}[Motivic Tartaglia's identity]\label{prop:TARTAGLIA}
Given two integers $i,d \in \BZ_{\geq 0}$, there are identities
\begin{align*}
[\Gr(i+1,d+1)]&=[\Gr(i,d)]+\BL^{i+1}[\Gr(i+1,d)]\\
[\Gr(i+1,d+1)]&=[\Gr(i+1,d)]+\BL^{d-i}[\Gr(i,d)].
\end{align*}
\end{lemma}

The next two lemmas relate the generating series of motives of Grassmanians to zeta functions of projective spaces.
\begin{lemma}[$\BL$-binomial Theorem]\label{lemma:zeta-inverse}
For every $d\in \BZ_{>0}$ there is an identity
\[
\zeta_{\BP^{d-1}}(t)^{-1}=\sum_{i=0}^d(-t)^i[\Gr(i,d)] \BL^{\binom{i}{2}}.
\]
\end{lemma}

\begin{lemma}\label{ex:gen-fun-grass}
For every $k \in \BZ_{\geq 0}$ we have
\begin{equation}\label{eqn:zeta_P^k}
\zeta_{\BP^k}(t)=\sum_{n\geq 0}\,[\Gr(k,n+k)]t^{n}.
\end{equation}
More generally, for every $\alpha\in\BZ_{>0}$ there is an identity
\begin{equation}\label{eqn:grass-shiftate}
\sum_{n\geq k} \,[\Gr(k,n)]\BL^{\alpha n}t^n =
\BL^{k\alpha}t^k\frac{\zeta_{\BP^{\alpha+k}}(t)}{\zeta_{\BP^{\alpha-1}}(t)}.
\end{equation}
\end{lemma}

\begin{proof}
The second statement follows easily from the first one, which is standard.
\end{proof}

\begin{prop}[Refined hockey stick identity]\label{lemma:refined-stick}
Fix integers $\alpha,\gamma,g \in \BZ_{\geq 0}$ such that $\gamma \leq g$. There is an identity
\begin{equation*}\label{stick-l=l}
[\Gr(g+1,g+1+\alpha)] = \sum_{k=0}^\alpha\,[\Gr(\gamma,\gamma+\alpha-k)][\Gr(g-\gamma,g-\gamma+k)]\BL^{k(\gamma+1)}.
\end{equation*}
In particular, the right hand side is independent on $\gamma$.
\end{prop}
\begin{proof}
    Notice the identity of zeta functions
\begin{align}\label{eqn: identiy of zeta}
\begin{split}
\zeta_{\BP^{g+1}}(t)&=\frac{1}{1-t}\cdots \frac{1}{1-t\BL^{\gamma}}\cdot \frac{1}{1-t\BL^{\gamma+1}}\cdots \frac{1}{1-t\BL^{\gamma+1}\BL^{g-\gamma}}\\
&= \zeta_{\BP^{\gamma}}(t)\cdot  \zeta_{\BP^{g-\gamma}}(t \BL^{\gamma+1}).
\end{split}
\end{align}
The conclusion follows by extracting the coefficient of $t^{\alpha}$ from both sides of \eqref{eqn: identiy of zeta}, and applying \Cref{ex:gen-fun-grass}.
\end{proof}

As a corollary of \Cref{lemma:refined-stick} we obtain an identity of binomial coefficients. 
\begin{corollary}\label{cor:refinedhockey}
Fix integers $\alpha,\gamma,g \in \BZ_{\geq 0}$ such that $\gamma \leq g$. Then
\[
\binom{g+1+\alpha}{g+1} = \sum_{k=0}^\alpha \binom{g-\gamma+k}{g-\gamma}\binom{\gamma+\alpha-k}{\gamma}.
\]
In particular, the right hand side is independent on $\gamma$.
\end{corollary}
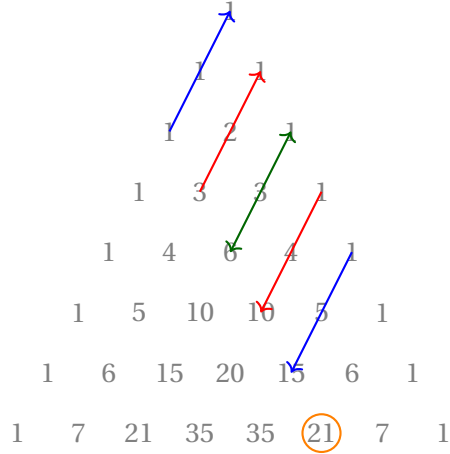
\begin{figure}[h!]
        \centering
        \begin{tikzpicture}
        \node at (0,0) {\tartaglia{7}};
        \draw[blue,thick,->] (2*0.8,-0.4)--(0.8,-2);
        
        \draw[red,thick,->] (1.5*0.8,0.4)--(0.5*0.8,-1.2);
        
        \draw[black!60!green,thick,<->] (1*0.8,1.2)--(0,-0.4);
        
        \draw[red,thick,<-] (0.5*0.8,2)--(-0.5*0.8,0.4);
        
        \draw[blue,thick,<-] (0,2.8)--(-0.8,1.2);

        \draw[orange,thick] (1.2,-2.8) circle  (0.25)   ;
    \end{tikzpicture}
\caption{Visual representation of \Cref{cor:refinedhockey}. In the picture $\alpha=2$, $g=4$. Green, red and blue illustrate the values $\gamma=2, 3, 4$, respectively. The latter (blue) recovers the classical \emph{hockey stick identity}.}

        \label{fig:refinedsticky}
    \end{figure}

\begin{lemma}\label{lemma:induzionestep}
Let $d, k\geq 1$ be integers. Then
\begin{equation*} 
\sum_{j=0}^{k}\,(-1)^{j} [\Gr(j,  j +d-1)][\Gr(k-j,d)]\BL^{\binom{j}{2} -j(k-1)}=0.
\end{equation*}
\end{lemma}
\begin{proof}
    Consider the trivial identity
    \begin{align}\label{eqn: zeta trivial}
    \zeta_{\BP^{d-1}}(t)\zeta_{\BP^{d-1}}(t)^{-1}=1.    
    \end{align}
    By Lemmas \ref{lemma:zeta-inverse}, \ref{ex:gen-fun-grass}, extracting the 
     coefficient of $t^k$ of \eqref{eqn: zeta trivial} yields
    \begin{align*}
       0&= \sum_{j=0}^k\,[\Gr(d-1,j+d-1)](-1)^{k-j}[\Gr(k-j,d)] \BL^{\binom{k-j}{2}}\\
       &= (-1)^k\BL^{\frac{k^2-k}{2}}\sum_{j=0}^k\, (-1)^{j}[\Gr(j,j+d-1)][\Gr(k-j,d)] \BL^{\binom{j}{2}-j(k-1)},
    \end{align*}
    which is the required vanishing.
\end{proof}

As a corollary of \Cref{lemma:induzionestep} we obtain an identity of binomial coefficients. 
\begin{corollary} \label{cor:crossbinomial}
Let $d\geq k\geq 1$ be integers. Then
\[
\sum_{j=0}^k {(-1)^j\binom{j+d-1}{j}\binom{d}{k-j}} = 0.
\]
\end{corollary}
\begin{figure}[h!]
        \centering
        \begin{tikzpicture}
        \node at (0,0) {\tartaglia{7}};
        \draw[blue,thick,->] (2*0.8,-0.4)--(-2*0.8,-0.4);
        \draw[blue,thick,->] (-1.5*0.8,0.4)--(0.4,-2.8);
        \draw[red,thick,->] (2*0.2,-1.2)--(-2*1,-1.2);
        \draw[red,thick,->] (-2*0.8,-0.4)--(-0.4,-2.8);
        \draw[black!60!green,thick,->] (2*0.4,1.2)--(-2*0.4,1.2);
        \draw[black!60!green,thick,->] (-2*0.2,2)--(0.4,0.4); 
    \end{tikzpicture}
\caption{Visual representation of \Cref{cor:crossbinomial}. Green, red and blue illustrate the values $(d,k)=(2,2)$, $(3,5)$, $(4,4)$, respectively.}  
    \end{figure}
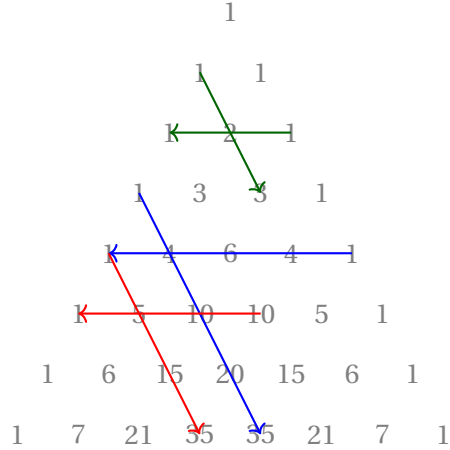
    
\begin{prop}\label{lemma:yetanotheridentity}
For every $1 \leq \varepsilon \leq d\leq m$ there is an identity
\[
[\Gr(d-\varepsilon,m-\varepsilon)][\Gr(\varepsilon-1,m)]  =\sum_{j=0}^{d-\varepsilon}\,(-1)^{j} [\Gr(d-1,  j +m )][\Gr(\varepsilon+j,d)]\BL^{\binom{j+1}{2} -(j+1)(d-\varepsilon)}.
\]
\end{prop}
\begin{proof} 
We will use an induction on $d$. The base step is the case $d=\varepsilon=1$, which yields the identity $1=1$ for every $m\geq 1$.

Now assume the statement holds for $d-1$ and prove it for $d$. In this part we want to proceed by induction on $m$. The base step is the proof of the formula for $m=d$:
\[
[\Gr(\varepsilon-1,d)]  =\sum_{j=0}^{d-\varepsilon}\,(-1)^{j} [\Gr(d-1,  j +d )][\Gr(\varepsilon+j,d)]\BL^{\binom{j+1}{2} -(j+1)(d-\varepsilon)}.
\]
This is \Cref{lemma:induzionestep} after a simple manipulation.

We are left to assume the statement holds for $m$ and prove it for $m+1$. First apply the motivic Tartaglia's identities to $[\Gr(d-1,j+m+1)]$, so that the right hand side of the sought after identity becomes
\begin{multline}\label{eqn:step1}
\sum_{j=0}^{d-\varepsilon}\,(-1)^{j}[\Gr(d-2,  j +m  )][\Gr(\varepsilon+j,d)]\BL^{\binom{j+1}{2} -(j+1)(d-\varepsilon)} \\
+\BL^{d-1}\sum_{j=0}^{d-\varepsilon}\,(-1)^{j}[\Gr(d-1,  j +m )][\Gr(\varepsilon+j,d)]\BL^{\binom{j+1}{2} -(j+1)(d-\varepsilon)}.
\end{multline}
Using the induction hypothesis on $m$ applied to the second sum, \eqref{eqn:step1} becomes
\begin{multline}\label{eqn:step2}
\BL^{d-1}[\Gr(d-\varepsilon,m-\varepsilon)][\Gr(\varepsilon-1,m)]\\
+\sum_{j=0}^{d-\varepsilon}\,(-1)^{j}[\Gr(d-2,  j +m  )][\Gr(\varepsilon+j,d)]\BL^{\binom{j+1}{2} -(j+1)(d-\varepsilon)}.
\end{multline}
Applying the motivic Tartaglia's identities to $[\Gr(\varepsilon+j,d)]$ in the second sum, \eqref{eqn:step2} becomes
\begin{multline}\label{eqn:step3}
    \BL^{d-1}[\Gr(d-\varepsilon,m-\varepsilon)][\Gr(\varepsilon-1,m)] \\ 
    +\BL^{\varepsilon-1}\sum_{j=0}^{d-\varepsilon} (-1)^{j} [\Gr(d-2,  j +m  )] [\Gr(\varepsilon+j,d-1)]\BL^{\binom{j+1}{2} -(j+1)(d-\varepsilon-1)}\\
    +\sum_{j=0}^{d-\varepsilon} (-1)^{j} [\Gr(d-2,  j +m  )] [\Gr(\varepsilon+j-1,d-1)]\BL^{\binom{j+1}{2} -(j+1)(d-\varepsilon)}.
\end{multline}
Applying the induction hypothesis on $d$ to the first sum, \eqref{eqn:step3} becomes 
\begin{multline}\label{eqn:step4}
\BL^{d-1}[\Gr(d-\varepsilon,m-\varepsilon)][\Gr(\varepsilon-1,m)]  +\BL^{\varepsilon-1}[\Gr(d-\varepsilon-1,m-\varepsilon)][\Gr(\varepsilon-1,m)]  \\ +\sum_{j=0}^{d-\varepsilon}\,(-1)^{j}  [\Gr(d-2,  j +m  )]   [\Gr(\varepsilon+j-1,d-1)] \BL^{\binom{j+1}{2} -(j+1)(d-\varepsilon)}.
\end{multline}
Extract the term in $j=0$ of the remaining sum in \eqref{eqn:step4}, and relabel the indices to start again summing from $0$, so that \eqref{eqn:step4} becomes
\begin{multline}\label{eqn:step5}
\BL^{\varepsilon-1} [\Gr(d-\varepsilon,m+1-\varepsilon)]   [\Gr(\varepsilon-1,m)] + [\Gr(d-2,  m  )]   [\Gr(\varepsilon -1,d-1)] \BL^{ - (d-\varepsilon)} \\ 
\BL^{-(d-\varepsilon)}\sum_{j=0}^{d-\varepsilon-1}\,(-1)^{j+1}  [\Gr(d-2,  j+1 +m  )]   [\Gr(\varepsilon+j ,d-1)] \BL^{\binom{j+1}{2} -(j+1)(d-\varepsilon-1)}.
\end{multline}
Using again the induction hypothesis on $d-1$, \eqref{eqn:step5} becomes
\begin{multline}\label{eqn:step6}
\BL^{\varepsilon-1} [\Gr(d-\varepsilon,m+1-\varepsilon)]   [\Gr(\varepsilon-1,m)] \\ 
+\BL^{\varepsilon-d}\big([\Gr(d-2,m)]   [\Gr(\varepsilon -1,d-1)]  -  [\Gr(d-1-\varepsilon,m+1-\varepsilon)][\Gr(\varepsilon-1,m+1)]\big).
\end{multline}
Expanding the second line of \eqref{eqn:step6} using \eqref{eqn:motive-grassmannian}, we get
\begin{multline}    \label{eqn:lunghissimaprevious}
    \BL^{\varepsilon-1} [\Gr(d-\varepsilon,m+1-\varepsilon)]   [\Gr(\varepsilon-1,m)] +\\
    \quad \quad \BL^{ \varepsilon-d}\left( \frac{[m]_{\BL}!}{[d-2]_{\BL}![m-d+2]_{\BL}!}\frac{[d-1]_{\BL}!}{[\varepsilon-1]_{\BL}![d-\varepsilon]_{\BL}!}\right.-\\
    \left.\frac{[m+1-\varepsilon]_{\BL}!}{[d-1-\varepsilon]_{\BL}![m-d+2]_{\BL}!}  \frac{[m+1]_{\BL}!}{[\varepsilon-1]_{\BL}![m-\varepsilon+2]_{\BL}!}  \right).
\end{multline}
The second summand can be simplified by using the definition of $[a]_{\BL}!$ in \eqref{eqn:motive-grassmannian}, and then arranged until a product of Grassmannian is found:
\begin{align*}
    \BL^{ \varepsilon-d}&\frac{[m]_{\BL}!}{ [m-d+2]_{\BL}![\varepsilon-1]_{\BL}![d-1-\varepsilon]_{\BL}! }\left(\frac{  \BL^{d-1}-1 }{  \BL^{d-\varepsilon}-1}-\frac{ \BL^{m+1}-1}{\BL^{m-\varepsilon+2}-1}\right) \\
    &=\frac{[m]_{\BL}!}{ [m-d+1]_{\BL}![\varepsilon-2]_{\BL}![d-\varepsilon]_{\BL}!} \frac{1}{(\BL^{m-\varepsilon+2}-1)} \\
    &=\frac{[m]_{\BL}![m-\varepsilon+1]_{\BL}!}{[m-d+1]_{\BL}![\varepsilon-2]_{\BL}![d-\varepsilon]_{\BL}! [m-\varepsilon+2]_{\BL}!}   \\
    &=[\Gr(\varepsilon-2,m)][\Gr(d-\varepsilon,m-\varepsilon+1)].
\end{align*}
Finally, we go back to \eqref{eqn:lunghissimaprevious} and apply one last time the motivic Tartaglia's identity to get the required identity
\begin{multline*}
\BL^{\varepsilon-1} [\Gr(d-\varepsilon,m+1-\varepsilon)]   [\Gr(\varepsilon-1,m)] +  [\Gr(\varepsilon-2,m )][\Gr(d-\varepsilon,m-\varepsilon+1 )]\\
=[\Gr(d-\varepsilon,m+1-\varepsilon)]   [\Gr(\varepsilon-1,m+1)].
\end{multline*}
The proof of the inductive step, and hence of the proposition, is complete.
\end{proof}
 
\begin{prop}\label{motivic-vanishing-double-grass}
Fix $h,k \in \BZ_{\geq 0}$, with $h<k$. Then
\[
\sum_{j=h}^{k}(-1)^j[\Gr(h,j)][\Gr(j,k)]\BL^{\binom{j+1}{2}-j(h+1)} = 0.
\]
\end{prop}
\begin{proof}
    Notice the identity 
    \begin{align}\label{eqn: zeta for 4.11}
        \zeta_{\BP^{h}}(t\BL^{k-h-1})\cdot\zeta_{\BP^{k-1}}(t)^{-1}=(1-t)\cdots (1-t\BL^{k-h-2}),
    \end{align}
    where the right-hand-side is a polynomial of degree $k-h-1$. By Lemmas \ref{lemma:zeta-inverse}, \ref{ex:gen-fun-grass},  extracting the coefficient of $t^{k-h}$ from \eqref{eqn: zeta for 4.11} yields
    \begin{align*}
        0&=\sum_{\alpha=0}^{k-h}\,[\Gr(\alpha,\alpha+h)]\BL^{\alpha(k-h-1)}\cdot (-1)^{k-h-\alpha}[\Gr(k-h-\alpha,k)]\BL^{\binom{k-h-\alpha}{2}}\\
        &=\sum_{j=h}^{k}\,[\Gr(j-h,j)]\BL^{(j-h)(k-h-1)}\cdot (-1)^{k-j}[\Gr(k-j,k)]\BL^{\binom{k-j}{2}}\\
        &=(-1)^{k}\BL^{\frac{k(k-1)}{2}-h(k-h-1)}\sum_{j=h}^{k}\,(-1)^j[\Gr(h,j)][\Gr(j,k)]\BL^{\binom{j+1}{2}-j(h+1)},
    \end{align*}
    which is the required vanishing.
\end{proof}

\begin{corollary}\label{lemma:triangle-vanishing}
Fix $h,k \in \BZ_{\geq 0}$, with $h<k$. Then
\[
\sum_{j=h}^{k}(-1)^j \binom{j}{h}\binom{k}{j} = 0.
\]
\end{corollary}

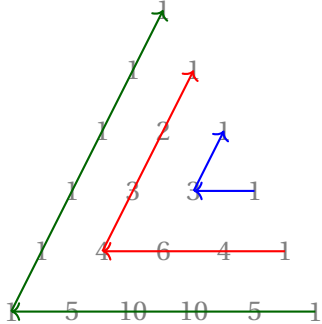
\begin{figure}[h!]
        \centering
        \begin{tikzpicture}
            \node at (0,0) {\tartaglia{5}};
            \draw[black!60!green,thick,->] (2,-2)--(-2,-2); 
            \draw[black!60!green,thick,->] (-2,-2)--(0,2);            
            \draw[red,thick,->] (1.6,-1.2)--(-0.8,-1.2);
            \draw[red,thick,->] (-0.8,-1.2)--(0.4,1.2);            
            \draw[blue,thick,->] (1.2,-0.4)--(0.4,-0.4);
            \draw[blue,thick,->] (0.4,-0.4)--(0.8,0.4);
        \end{tikzpicture}
        \caption{Visual representation of \Cref{lemma:triangle-vanishing}. In the picture $k=4$. Green, red and blue illustrate the values $h=0, 1, 2$, respectively.}
        \label{fig:triangles}
    \end{figure}

\begin{lemma}\label{lemma:horisticky}
Let $d>i\geq 0$ be integers. Then
\[
[\Gr(i,d-1)]\BL^{\binom{i+1}{2}}= \sum_{\alpha=0}^i\,(-1)^{\alpha+i}[\Gr(\alpha,d)]\BL^{\binom{\alpha}{2}}.
\]
\end{lemma}
\begin{proof}
    Consider the identity
    \begin{align}\label{eqn: zeta for 4.13}
        \zeta_{\BP^{d-2}}(t\BL)^{-1}=\zeta_{\BP^{d-1}}(t)^{-1}\cdot\frac{1}{1-t}.
    \end{align}
    Extracting the coefficient of $t^i$ from \eqref{eqn: zeta for 4.13} yields
    \[
    (-1)^i\BL^{i}[\Gr(i,d-1)]\BL^{\binom{i}{2}}=\sum_{\alpha=0}^i\,(-1)^{\alpha}[\Gr(\alpha,d)]\BL^{\binom{\alpha}{2}},
    \]
    which is equivalent to the required identity.
\end{proof}

    \begin{figure}[h!]
        \centering
        \begin{tikzpicture}
            \node at ( 0,0) {\tartaglia{5}};
            \draw[red, thick] (-1.6,-1.2)--(0,-1.2) ;         
            \draw[black!60!green, thick] (-0.8,0.4)--(0,0.4) ;
            \draw[blue, thick] (-2,-2)--(1.2,-2)  ;            
        \draw[thick, red] (0.4,-0.4) circle  (0.25)  ;
        \draw[thick, blue] (1.6,-1.2) circle  (0.25)  ;
        \draw[thick, black!60!green] (0.4,1.2) circle  (0.25)  ;
        \end{tikzpicture}
        \caption{Visual representation of the numerical counterpart of \Cref{lemma:horisticky}. Green, red and blue represent the values $(i,d)=(1,2)$, $(2,4)$, $(4,5)$, respectively.}
        \label{fig:horisticky}
    \end{figure}
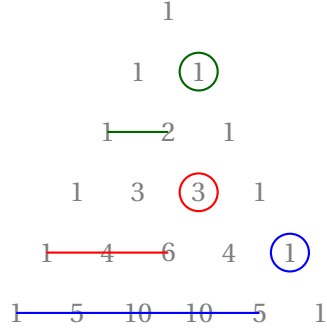
    
We present now a divisibility criterion that will be used in \Cref{subsec:indiani}.

\begin{prop}\label{lemm:congruenc}
Let $0<k\le n$ be integers. The motive $[\Gr (k,n)]$ is divisible by $[\BP^2]$ in $K_0(\Var_{\BC})$ if and only if either
\begin{itemize} 
\item [\mylabel{case-i}{\normalfont{(i)}}] $k\equiv 1 \pmod 3$, and $n\equiv 0 \pmod 3 $, or
\item  [\mylabel{case-ii}{\normalfont{(ii)}}]$k\equiv 2 \pmod 3$, and $n\not \equiv 2 \pmod 3$.
\end{itemize}
\end{prop}

\begin{proof}
Let $\omega$ be a primitive 3-rd root of unity. For every integer $a \in \BN$, we view 
\[
[a]_{\BL}! = \prod_{i=i}^a\,\left(\BL^i-1\right)
\]
as a polynomial in $\BZ[\BL]$. Since $1+\omega+\omega^2 = 0$, to prove divisibility it is enough to prove that $[\Gr (k,n)]\,\big|_{\BL = \omega} = 0$, which, by \Cref{eqn:motive-grassmannian}, is equivalent to the inequality 
\begin{equation}\label{order-vanishing}
\mult_\omega \,[n]_{\BL}! > \mult_\omega \,[k]_{\BL}!+\mult_\omega \,[n-k]_{\BL}!.
\end{equation}
For every $i \in\BN$, we have that $\omega$ is a root of $\BL^{i}-1$ if and only if $i \in 3\BN$, and moreover its multiplicity is always $1$ in this case, due to the factorisation $\BL^{3j}-1 = (\BL^3-1) \sum_{0\leq h\leq j-1} \BL^{3h}$
and to $\omega$ being primitive. Define
\[
\delta_{i,3\BN} = 
\begin{cases}
    1 & \mbox{if }i \in 3\BN \\
    0 & \mbox{if }i \notin 3\BN.
\end{cases}
\]
Then our discussion implies that
\begin{equation}\label{eqn:order}
\mult_{\omega} \,[a]_{\BL}! = \sum_{i=1}^a \delta_{i,3\BN}.
\end{equation}
Let us consider Case \ref{case-i}. Set $n=3\alpha$ and $k = 3\beta+1$. By \Cref{eqn:order}, we have
\begin{equation}\label{eqn:gesù-termosifone}
\begin{split}
  \mult_\omega [n]_{\BL}! &= \alpha \\ 
  \mult_\omega[ k]_{\BL}! &= \beta \\ 
  \mult_\omega[n-k]_{\BL}! &= \mult_\omega[3(\alpha-\beta-1)+2]_{\BL}! = \alpha-\beta-1. 
\end{split} 
\end{equation} 
Thus \eqref{order-vanishing} holds and Case \ref{case-i} is confirmed.

We move to Case \ref{case-ii}. Write $n=\varepsilon + 3\alpha$ and $k = 2+3\beta$, with $\varepsilon \in \set{0,1}$. As before, by \Cref{eqn:order}, we have the same relations as in \eqref{eqn:gesù-termosifone}, where the last identity follows after noting that $n-k = 3(\alpha-\beta-1)+1+\varepsilon$.
Thus \eqref{order-vanishing} holds and Case \ref{case-ii} is confirmed as well. 

Finally, the case $k \equiv 0 \pmod 3$ violates \eqref{order-vanishing} for every $n$, for in this case writing $k = 3 \beta$ and $n = \varepsilon + 3\alpha$ with $\varepsilon \in \set{0,1,2}$ one always gets $\mult_\omega [n-k]_{\BL}! = \alpha-\beta = \mult_\omega [n]_{\BL}! - \mult_\omega [k]_{\BL}!$.
\end{proof}

\section{Proof of \texorpdfstring{\Cref{MAIN-THEOREM-1}}{}}
In this section we prove our first main theorem. To do so, we will exploit the motivic relations obtained in \Cref{sec:motives-of-grass} and two more results proved in the next subsection.

\subsection{Preparation lemmas}
For $d \in \BZ_{>1}$, the stratification \eqref{eq:stratification} implies the identity of motives
\begin{equation}\label{eqn:strata-emb-dim}
[\Hilb^d(\BA^n)_0] =
    \sum_{k=1}^{d-1} \,[Y^{n}_{k,d}].
\end{equation}
We have the following `inversion formula'  of \eqref{eqn:strata-emb-dim} allowing one to pass from the classes $[\Hilb^d(\BA^k)_0]$ to the classes $[Y^k_{k,d}]$.

\begin{lemma}[Inversion Formula]\label{lemma:inversion}
For $k=1,\ldots,d-1$ we have
\begin{equation}\label{eqn:inversion-formula}
{[}Y_{k,d}^k{]}=\sum_{j=1}^{k} (-1)^{k+j}{[}\Gr(j,k){]}\BL^{(k-j) (d-j-1)-\binom{k-j}{2}}[\Hilb^d(\BA^j)_0].
\end{equation}
\end{lemma}

\begin{proof}
Denote by $\mathsf C_{k,d}$ the right hand side of \eqref{eqn:inversion-formula}. Then we compute
\begin{align*}
\mathsf C_{k,d} 
&=(-1)^k\sum_{j=1}^{k}(-1)^j[\Gr(j,k)]\BL^{(k-j) (d-j-1)-\binom{k-j}{2}}\sum_{h=1}^{d-1}\,[Y^j_{h,d}] \\
&=(-1)^k\sum_{j=1}^{k}(-1)^j[\Gr(j,k)]\BL^{(k-j) (d-j-1)-\binom{k-j}{2}}\sum_{h=1}^j\,[\Gr(h,j)]\BL^{(j-h)(d-h-1)}[Y^h_{h,d}]\\
&=(-1)^k\sum_{h=1}^k\,[Y^h_{h,d}]\sum_{j=h}^k (-1)^j [\Gr(j,k)][\Gr(h,j)] \BL^{(k-j) (d-j-1)-\binom{k-j}{2}+(j-h)(d-h-1)} \\
&=(-1)^k\sum_{h=1}^k \,[Y^h_{h,d}] \BL^{(d-1)(k-h)+h^2-\binom{k}{2}} \sum_{j=h}^k(-1)^j[\Gr(h,j)][\Gr(j,k)]\BL^{\binom{j+1}{2}-j(h+1)}
\end{align*}
where we used \Cref{eqn:strata-emb-dim} for the first identity and \Cref{thm:motive-Y-stratum} for the second identity. But the last sum vanishes for $h\neq k$ by \Cref{motivic-vanishing-double-grass}, thus $\mathsf C_{k,d} =  [Y^{k}_{k,d}]$, as required.
\end{proof}

The next proposition is a convenient rewriting of the $i$-th coefficient $a^{(d)}_i$ of $\mathsf P_d(t)$.

\begin{prop}\label{prop:coefficient-of-p_d}
For every $d>3$ and for every $i=0,\ldots,d-2$, define
\[
a^{(d)}_i = \sum_{\alpha=0}^i\, (-1)^\alpha[\Hilb^d(\BA^{i-\alpha+1})_0][\Gr(\alpha,d)]\BL^{\binom{\alpha}{2}}.
\]
There is a relation
\[
a^{(d)}_i = \sum_{k=0}^i\,(-1)^{i-k}[\Gr(i-k,d-k-2)]\BL^{\binom{i-k}{2}}[Y^{k+1}_{k+1,d}].
\]
\end{prop}

\begin{proof}
Let us set, for simplicity,
\[
\mathsf A_{i}^{(d)} = \sum_{k=0}^i\,(-1)^{i-k}[\Gr(i-k,d-k-2)]\BL^{\binom{i-k}{2}}[Y^{k+1}_{k+1,d}].
\]
We directly compute
\begin{align*}
\mathsf A_{i}^{(d)}
&= \sum_{k=0}^i\,(-1)^{i-k}[\Gr(i-k,d-k-2)]\BL^{\binom{i-k}{2}} \\ &\qquad\qquad\cdot\sum_{j=1}^{k+1} (-1)^{k+1+j}[\Hilb^d(\BA^j)_0][\Gr(j,k+1)]\BL^{(k+1-j)(d-j-1)-\binom{k+1-j}{2}}\\
&=\sum_{k=0}^i\sum_{j=1}^{k+1}(-1)^{i+j+1}[\Hilb^d(\BA^j)_0][\Gr(j,k+1)] \\
&\qquad\qquad \cdot[\Gr(i-k,d-k-2)]\BL^{(k+1-j)(d-j-1)-\binom{k+1-j}{2}+\binom{i-k}{2}} \\
&=\sum_{j=1}^{i+1} (-1)^{i+j+1}[\Hilb^d(\BA^j)_0] \\
&\qquad\qquad\cdot \sum_{k=j-1}^i\,[\Gr(j,k+1)][\Gr(i-k,d-k-2)]\BL^{(k+1-j)(d-j-1)-\binom{k+1-j}{2}+\binom{i-k}{2}}.
\end{align*}
Using the change of variables $\alpha = i+1-j$, we obtain
\begin{align*}
\mathsf A_{i}^{(d)} &= \sum_{\alpha=0}^i\,(-1)^\alpha [\Hilb^{d}(\BA^{i-\alpha+1})_0] \\
&\qquad\qquad \cdot\sum_{k=i-\alpha}^i\,[\Gr(i-\alpha+1,k+1)][\Gr(i-k,d-k-2)]\BL^{(k+\alpha-i)(d-i+\alpha-2)-\binom{k+\alpha-i}{2}+\binom{i-k}{2}}.
\end{align*}
Using the change of variables $h = k+\alpha-i$, we obtain
\begin{align*}
\mathsf A_{i}^{(d)} 
&= \sum_{\alpha=0}^i\,(-1)^\alpha [\Hilb^{d}(\BA^{i-\alpha+1})_0] \\
&\qquad\cdot\sum_{h=0}^\alpha\,[\Gr(h,h+i-\alpha+1)][\Gr(\alpha-h,\alpha-h+i-d-2)]\BL^{h(d-i+\alpha-2)-\binom{h}{2}+\binom{\alpha-h}{2}}.
\end{align*}
Using the change of variables $k=\alpha-h$, we obtain
\begin{align*}
\mathsf A_{i}^{(d)} 
&= \sum_{\alpha=0}^i\,(-1)^\alpha [\Hilb^{d}(\BA^{i-\alpha+1})_0]\\
&\qquad \cdot\sum_{k=0}^\alpha\,[\Gr(\alpha-k,i+1-k)][\Gr(k,d+k-i-2)]\BL^{(\alpha-k)(d-i+\alpha-2)-\binom{\alpha-k}{2}+\binom{k}{2}}\\
&=\sum_{\alpha=0}^i\,(-1)^\alpha [\Hilb^{d}(\BA^{i-\alpha+1})_0]\BL^{\binom{\alpha}{2}} \\
&\qquad \cdot\sum_{k=0}^\alpha\,[\Gr(\alpha-k,i+1-k)][\Gr(k,d+k-i-2)]\BL^{(\alpha-k)(d-i-1)}.
\end{align*}
Finally, we claim that 
\[
[\Gr(\alpha,d)] = \sum_{k=0}^\alpha \,[\Gr(\alpha-k,i+1-k)][\Gr(k,d+k-i-2)]\BL^{(\alpha-k)(d-i-1)}.
\]
This is a direct application of \Cref{lemma:refined-stick} setting $\gamma = d-i-2$.
This proves the sought after identity $\mathsf A_{i}^{(d)} = a^{(d)}_{i}$.
\end{proof}

\subsection{Proof of \texorpdfstring{\Cref{MAIN-THEOREM-1}}{}}\label{sec:proof-thm-A}
We can now prove \Cref{MAIN-THEOREM-1}, which we restate here for convenience. Recall that our focus is on the generating function
\[
\mathsf Z_d(t) = \sum_{n\geq 0}\,{\motive{\Hilb^d(\BA^{n+1})_0}}  t^n \,\in\,K_0(\Var_{\BC})\llbracket t \rrbracket.
\]
\begin{theorem}\label{main-in-body}
Fix an integer $d>0$. Then
\[
\mathsf Z_d(t) = \zeta_{\BP^{d-1}}(t) \cdot \mathsf{P}_d(t)
\]
in $K_0(\Var_{\BC})\llbracket t \rrbracket$, where   $\mathsf P_d(t) \in K_0(\Var_{\BC})[t]$ is a polynomial such that
\begin{itemize}
\item [\mylabel{mainthm-body-1}{\normalfont{(1)}}] $\mathsf{P}_d(t)=1$ if $d\leq 3$, and,
\item [\mylabel{mainthm-body-2}{\normalfont{(2)}}] if $d>3$, then
\[
\mathsf{P}_d(t) = \sum_{i=0}^{d-2} a_i^{(d)}t^i
\] 
has degree at most $d-2$, and its coefficients are given by
\[
a_i^{(d)} = \sum_{\alpha=0}^i\,(-1)^\alpha [\Hilb^d(\BA^{i-\alpha+1})_0][\Gr(\alpha,d)]\BL^{\binom{\alpha}{2}}, \qquad 0\leq i \leq d-2.
\]
\end{itemize}
In particular, $\mathsf Z_d(t)$ is a rational function.
\end{theorem}

\begin{proof}
We start with the case $d\leq 3$. If $d=1$, then $\Hilb^1(\BA^{n+1})_0=\Spec \BC$ for every $n\geq 0$, so $\mathsf Z_1(t)=(1-t)^{-1} = \zeta_{\BP^0}(t)$. If $d=2$, then $\Hilb^2(\BA^{n+1})_0 = \BP(\mathfrak m/\mathfrak m^2) = \BP^n = \Sym^n(\BP^1)$, where $\mathfrak m = (x_1,\ldots,x_{n+1})$ is the ideal of the origin $0 \in \BA^{n+1}$. Thus $\mathsf Z_2(t) = \zeta_{\BP^1}(t)$. If $d=3$, then a punctual subscheme $Z \subset \BA^{n+1}$ is either curvilinear or abstractly isomorphic to $\Spec \BC[u,v] / (u^2,uv,v^2)$. The latter case contributes the choice of a plane $\BA^2 \subset \BA^{n+1}$, whereas the curvilinenar locus $\mathscr{C}_3^{n+1}\subset \Hilb^3(\BA^{n+1})_0$ is the total space of an $\BA^n$-fibration over $\BP(\Fm/\Fm^2) = \BP^n$,  so that 
\begin{align*}
[\Hilb^3(\BA^{n+1})_0] 
&= [\Gr(2,n+1)] + \BL^n[\BP^n] \\
&= [\Gr(2,n+2)] & \mbox{by \Cref{prop:TARTAGLIA}} \\
&= [\Sym^n(\BP^2)]  & \mbox{by \Cref{ex:gen-fun-grass}} 
\end{align*}
and we obtain $\mathsf Z_3(t) = \zeta_{\BP^2}(t)$, which proves \ref{mainthm-body-1}. 

\smallbreak
For $d>3$, we compute directly
\begin{align*}
\mathsf Z_d(t) 
&= \sum_{n\geq 0}\,[\Hilb^d(\BA^{n+1})_0]t^n \\
&= \sum_{n\geq 0}\,\left(\sum_{k=1}^{d-1}\,[Y^{n+1}_{k,d}] \right)t^n & \mbox{by \eqref{eqn:strata-emb-dim}} \\ 
&= \sum_{k=1}^{d-1}\,\sum_{n\geq 0}\,[\Gr(k,n+1)]\BL^{(n+1-k)(d-k-1)}[Y^k_{k,d}]t^n & \mbox{by \Cref{thm:motive-Y-stratum}}\\
&= \sum_{k=1}^{d-1}\,\BL^{-k(d-k-1)}[Y^k_{k,d}]t^{-1}\sum_{n+1\geq k}\BL^{(n+1)(d-k-1)}[\Gr(k,n+1)]t^{n+1} \\
&= \sum_{k=1}^{d-1}\,\BL^{-k(d-k-1)}[Y^k_{k,d}]\BL^{k(d-k-1)}t^{k-1}\zeta_{\BP^{d-1}}(t)\zeta_{\BP^{d-k-2}}(t)^{-1} & \mbox{by \eqref{eqn:grass-shiftate}}\\
&=\zeta_{\BP^{d-1}}(t)\sum_{k=1}^{d-1}\,[Y^k_{k,d}]t^{k-1}\zeta_{\BP^{d-k-2}}(t)^{-1} & \\
&=\zeta_{\BP^{d-1}}(t)\sum_{k=1}^{d-1}\,[Y^k_{k,d}]t^{k-1} \sum_{i=0}^{d-k-1}(-t)^i[\Gr(i,d-k-1)]\BL^{\binom{i}{2}}  & \mbox{by \Cref{lemma:zeta-inverse}}.
\end{align*}
Let us set
\[
\widetilde{\mathsf P}_d(t) = \sum_{k=1}^{d-1}\,[Y^k_{k,d}]t^{k-1} \sum_{i=0}^{d-k-1}(-t)^i[\Gr(i,d-k-1)]\BL^{\binom{i}{2}}.
\]
Our claim is that 
\begin{equation}\label{eqn:equality-polynomials}
\widetilde{\mathsf P}_d(t) = \mathsf P_d(t).
\end{equation}
We calculate
\begin{align*}
\widetilde{\mathsf P}_d(t) 
&= \sum_{k=1}^{d-1}\,[Y^k_{k,d}] \sum_{i=0}^{d-k-1}(-1)^i[\Gr(i,d-k-1)]\BL^{\binom{i}{2}}t^{k+i-1} \\
&= \sum_{k=0}^{d-2}\,[Y^{k+1}_{k+1,d}]\sum_{i=0}^{d-k-2}\,(-1)^i[\Gr(i,d-k-2)]\BL^{\binom{i}{2}}t^{k+i} & \mbox{shift }k \mapsto k+1\\
&= \sum_{k=0}^{d-2}\,[Y^{k+1}_{k+1,d}] \sum_{i=k}^{d-2} \, (-1)^{i-k} [\Gr(i-k,d-k-2)]\BL^{\binom{i-k}{2}}    t^i& \mbox{shift }i \mapsto i-k \\
&= \sum_{i=0}^{d-2} \left(\sum_{k=0}^i \, (-1)^{i-k}  [\Gr(i-k,d-k-2)]\BL^{\binom{i-k}{2}}[Y^{k+1}_{k+1,d}]\right) t^i.
\end{align*}
Thus the final claim, which would confirm \eqref{eqn:equality-polynomials}, is that 
\[
\sum_{k=0}^i \, (-1)^{i-k} [\Gr(i-k,d-k-2)]\BL^{\binom{i-k}{2}}[Y^{k+1}_{k+1,d}] 
\]
agrees with the $i$-th coefficient of $\mathsf P_d(t)$. But this is precisely \Cref{prop:coefficient-of-p_d}.
\end{proof}

\subsection{Corollaries of \texorpdfstring{\Cref{MAIN-THEOREM-1}}{}}
We give two immediate corollaries of our main theorem.

The following result shows that the whole generating function $\mathsf Z_d(t)$ is determined by the finitely many motives
\begin{equation}\label{eqn:d-1 motives}
[\Hilb^d(\BA^1)_0],\,\,[\Hilb^d(\BA^2)_0],\,\,\ldots\,\,,\,\,[\Hilb^d(\BA^{d-1})_0].
\end{equation}

\begin{corollary}\label{cor:hilbrel}
For every $m\geq d > 1$ there is a relation
\[
[\Hilb^d(\BA^{m})_0]=\sum_{\gamma =1}^{d-1}\,(-1)^{d+\gamma+1}[\Hilb^d(\BA^{\gamma})_0][\Gr(m-d,m-\gamma-1)][\Gr(\gamma,m)]\BL^{\binom{d-\gamma}{2}}.
\]
\end{corollary}
\begin{proof}
For every $n \geq 0$ have
    \begin{align*}
&[\Hilb^d(\BA^{n+1})_0]\\ 
&=\sum_{h=0}^{d-2}[\Gr(d-1,h+n+1)]\cdot a^{(d)}_{d-2-h} &\mbox{by \Cref{ex:gen-fun-grass},} \\
&=\sum_{h=0}^{d-2}[\Gr(d-1,h+n+1)]\cdot \left(\sum_{\alpha=0}^{d-2-h}\,(-1)^\alpha [\Hilb^d(\BA^{{d-1-h}-\alpha})_0][\Gr(\alpha,d)]\BL^{\binom{\alpha}{2}}\right)  &  \mbox{by \Cref{main-in-body},}  \\ 
&=\sum_{k=0}^{d-2}\sum_{\alpha=0}^{k}\,(-1)^\alpha [\Hilb^d(\BA^{1+k-\alpha})_0][\Gr(d-1,d-k+n-1)]  [\Gr(\alpha,d)]\BL^{\binom{\alpha}{2}}  &  h=d-2-k,\\
&=\sum_{k=0}^{d-2}\sum_{\beta=0}^{k}\,(-1)^{k-\beta} [\Hilb^d(\BA^{1+\beta})_0][\Gr(d-1,d-k+n-1)][\Gr(k-\beta,d)]\BL^{\binom{k-\beta}{2}}  &  \alpha=k-\beta, \\
&=\sum_{\beta =0}^{d-2}\sum_{k=\beta}^{d-2}\,(-1)^{k-\beta} [\Hilb^d(\BA^{1+\beta})_0][\Gr(d-1,d-k+n-1)][\Gr(k-\beta,d)]\BL^{\binom{k-\beta}{2}}  &  \\
&=\sum_{\beta =0}^{d-2}[\Hilb^d(\BA^{1+\beta})_0]\sum_{j=0}^{d-2-\beta}\,(-1)^{j} [\Gr(d-1,d-j-\beta+n-1)][\Gr(j,d)]\BL^{\binom{j}{2}}  &  k=j+\beta. 
\end{align*}
Applying the further substitution $\beta=\gamma -1$ and $n=m-1$ one obtains the identity
\[
[\Hilb^d(\BA^{m})_0]=\sum_{\gamma =1}^{d-1}\,[\Hilb^d(\BA^{\gamma})_0]\sum_{j=0}^{d-1-\gamma}\,(-1)^{j} [\Gr(d-1,d-j-\gamma+m-1)][\Gr(j,d)]\BL^{\binom{j}{2}}.
\]
The conclusion then follows from \Cref{lemma:yetanotheridentity} after substituting  $j$ with $d-\gamma-j-1$ and then $\gamma+1$ with $\varepsilon$.
\end{proof}

\begin{remark}\label{rmk:first-2-motives}
The relation found in \Cref{cor:hilbrel} is equivalent to \Cref{MAIN-THEOREM-1}. Moreover, the first two motives in \eqref{eqn:d-1 motives}, out of the $d-1$ needed to determine all the next ones, are known, as
\[
[\Hilb^d(\BA^1)_0] = 1,\qquad [\Hilb^d(\BA^2)_0] = \Coef_{t^d} \left(\prod_{m\geq 1}\,\left(1-\BL^{m-1}t^m\right)^{-1}\right).
\]
The second identity is G\"{o}ttsche's formula \cite{Gottsche-motivic}.
\end{remark}

\smallbreak
Another immediate corollary of \Cref{MAIN-THEOREM-1} is the following.

\begin{corollary}\label{Cor:P(1)=1}
For every $d\ge 1$, we have
    \[
    \mathsf{P}_d(1)=1.
    \] 
In particular, at the level of Euler characteristics, one has
\[
\sum_{n\geq 0}\,\chi(\Hilb^{d}(\BA^{n+1}))t^n = \frac{q_d(t)}{(1-t)^d},
\]
where $q_d(t) \in \BZ[t]$ is a polynomial whose coefficients sum up to $1$.
\end{corollary}
\begin{proof} 
The statement is obvious for $1\leq d\leq 3$, since $\mathsf P_d (t) = 1$ in this case. For $d>3$ we compute
\begin{align*}
    1&={[}Y_{d-1,d}^{d-1}{]} & \mbox{by \Cref{rmk:small-cases-of-Y}} \\
    &=\sum_{i=1}^{d-1} (-1)^{d-1+i}{[}\Gr(i,d-1){]}\BL^{(d-1-i) (d-i-1)-\binom{d-1-i}{2}}[\Hilb^d(\BA^i)_0] &\mbox{by \Cref{lemma:inversion}}\\
    &=\sum_{i=0}^{d-2} (-1)^{d + i}{[}\Gr(i+1,d-1){]}\BL^{(d-i-2) (d-i-2)-\binom{d-i-2}{2}}[\Hilb^d(\BA^{i+1})_0] &\mbox{shift}\\
    &=\sum_{i=0}^{d-2} (-1)^{d + i}{[}\Gr(d-i-2,d-1){]}\BL^{ \binom{d-i-1}{2}}[\Hilb^d(\BA^{i+1})_0] &\mbox{rearranging}\\
    &=\sum_{\beta=2}^{d} (-1)^{\beta}{[}\Gr(\beta-2,d-1){]}\BL^{ \binom{\beta-1}{2}}[\Hilb^d(\BA^{d-\beta+1})_0] &\mbox{$\beta= d-i$}\\
    &=\sum_{\beta=0}^{d-2} (-1)^{\beta}{[}\Gr(\beta,d-1){]}\BL^{ \binom{\beta+1}{2}}[\Hilb^d(\BA^{d-\beta-1})_0] &\mbox{shift}\\
    &=\sum_{\beta=0}^{d-2} (-1)^{\beta}\left( \sum_{\alpha=0}^\beta (-1)^{\alpha+\beta}[\Gr(\alpha,d)]\BL^{\binom{\alpha}{2}}\right)[\Hilb^d(\BA^{d-\beta-1})_0]  &\mbox{by \Cref{lemma:horisticky}}\\ 
    &=   \sum_{\alpha=0}^{d-2}  \sum_{\beta=\alpha}^{d-2} (-1)^{\alpha}[\Gr(\alpha,d)]\BL^{\binom{\alpha}{2}} [\Hilb^d(\BA^{d-\beta-1})_0] &\mbox{swap}\\
    &=   \sum_{\alpha=0}^{d-2}  \sum_{i =\alpha}^{d-2} (-1)^{\alpha}[\Gr(\alpha,d)]\BL^{\binom{\alpha}{2}} [\Hilb^d(\BA^{i-\alpha+1})_0] &\mbox{$\beta= d-2+\alpha-i$}\\
    &=     \sum_{i = 0}^{d-2} \sum_{\alpha=0}^{i} (-1)^{\alpha}[\Gr(\alpha,d)]\BL^{\binom{\alpha}{2}} [\Hilb^d(\BA^{i-\alpha+1})_0]
    &\mbox{swap.}
\end{align*} 
But this agrees with $\sum_{0\leq i\leq d-2} a^{(d)}_i = \mathsf P_d(1)$ by \Cref{main-in-body}, as required.
\end{proof}

\begin{remark} \label{rem: relations from Y}
\Cref{Cor:P(1)=1} translates into a relation among the coefficients of $\mathsf{P}_d(t)$, i.e.~among some motives of Hilbert scheme of points, and comes directly from the relation $[Y^{d-1}_{d-1,d}] = 1$ (cf.~\Cref{rmk:small-cases-of-Y}). Similarly, one can generate more relations analysing the contribution of each $[Y_{k,d}^k]$ to $\mathsf P_d(t)$. In fact, each stratum $Y_{k,d}^k$ provides a linear relation on the $d-1$ coefficients of the polynomial, and all these linear relations are independent. For instance, for $d\ge 3$, the relation corresponding to $k=d-2$ yields the identity
\begin{equation}\label{relation-d-2}
\BL^{d-2}[Y^{d-2}_{d-2,d}]+\sum_{i=1}^{d-2}\BL^{d-2-i}\left[\BP^{i-1}\right]a_i^{(d)}=0.
\end{equation}
Note that, as we confirm in \Cref{prop:easyY}, $Y^{d-2}_{d-2,d}=\BP^{\binom{d-2}{2}-1}$. Specialising \eqref{relation-d-2} to $\BL=1$ we get the numerical relation 
\[
\binom{d-2}{2} +\sum_{i=1}^{d-2}i\chi (a_i^{(d)}) = 0.
\]
\end{remark}

\section{Stabilisation to Infinite Grassmannian}\label{sec:Stabilisation}
In this section we prove that the motive of $\Hilb^d(\BA^n)_0$ stabilises to the motive of the infinite Grassmannian $\Gr(d-1,\infty)$, in a suitable sense.

\subsection{Completion of \texorpdfstring{$K_0(\Var_\BC)$}{}}
The descending filtration of $K_0(\Var_\BC)$ given by the ideals
\[
K_0(\Var_\BC)= (\BL^0)\supset (\BL^1)\supset (\BL^2)\supset \cdots
\]
gives rise to the $\BL$-adic completion 
\[
\widehat{K}_0(\Var_\BC)=\varprojlim K_0(\Var_\BC)/(\BL^\ell),
\]
which is naturally a topological ring (with respect to the Krull topology).

Let $X_{\infty}=\varinjlim X_n$ be an ind-scheme. If the sequence of motives $([X_n])_n\in K_0(\Var_\BC) $ converges with respect to the induced topology in $\widehat{K}_0(\Var_\BC)$, we associate a motivic class
\[
[X_\infty]=\lim_n\, [X_n]\,\in\, \widehat{K}_0(\Var_\BC).
\]
 Fix $d\geq 1$ and define the \emph{infinite Grassmannian} by
\[
\Gr(d-1,\infty)=\varinjlim \Gr(d-1,n).
\]
By \Cref{prop:TARTAGLIA} the sequence of motives $([\Gr(d-1,n)])_n$ converges in $ \widehat{K}_0(\Var_\BC)$, and 
by \cite[Eqn. (2.6)]{GLMstacks} its limit can be extracted by the closed formula
\begin{align}\label{eqn: inft grass formula}
    \sum_{k=0}^\infty \,[\Gr(k,\infty)]t^k=\prod_{i=0}^\infty\frac{1}{1-\BL^{i}t}\in \widehat{K}_0(\Var_\BC)\llbracket t \rrbracket.
\end{align}
For instance, the motive of the infinite projective space $\BP^\infty = \Gr(1,\infty)$ is
\[
[\BP^\infty]=\sum_{i=0}^\infty \BL^i =\frac{1}{1-\BL}\in \widehat{K}_0(\Var_\BC).
\]

\subsection{Infinite punctual Hilbert scheme}
We exploit the stratification \eqref{eqn:strata-emb-dim} to prove that the motive of $\Hilb^d(\BA^n)_0$ coincides with the motive of $[\Gr(d-1,n)]$, modulo a suitable power of $\BL$. This will enable us to deduce a formula for $[\Hilb^d(\BA^\infty)_0]$ in the $\BL$-adic completion $\widehat{K}_0(\Var_\BC)$.

\begin{theorem}\label{thm: stab motives}
Fix integers $d,n\geq 1$.
Then 
    \[
   [\Hilb^d(\BA^n)_0] =[\Gr(d-1,n)] \in K_0(\Var_\BC)/(\BL^{n-d+2}).
    \]
\end{theorem}
\begin{proof}
 We have 
    \begin{align*}
         [\Hilb^d(\BA^n)_0]&=  \sum_{k=1}^{d-1} \,[Y^{n}_{k,d}] & \mbox{by \eqref{eqn:strata-emb-dim}}\\
         &=\sum_{k=1}^{d-1} [\Gr(k,n)]\BL^{(n-k)(d-k-1)}[Y^k_{k,d}]& \mbox{by \Cref{thm:motive-Y-stratum}}\\
         &=  [\Gr(d-1,n)]+ \sum_{k=1}^{d-2} [\Gr(k,n)]\BL^{(n-k)(d-k-1)}[Y^k_{k,d}] &\\
         &=[\Gr(d-1,n)]+ \BL^{n-d+2}A_{n,d},&
    \end{align*}
    where $A_{n,d}\in K_0(\Var_\BC)$ are effective classes. Quotienting by $\BL^{n-d+2}$ yields the result. 
\end{proof}
As a corollary, we obtain a closed formula for the motive of the `infinite punctual Hilbert scheme' (again viewed as an ind-scheme)
\[
\Hilb^d(\BA^\infty)_0=\varinjlim \Hilb^d(\BA^n)_0.
\]
\begin{corollary}\label{cor:limit-Hilb}
 The  sequence of motives $([\Hilb^d(\BA^n)_0])_n$ converges in $ \widehat{K}_0(\Var_\BC)$ and  we have 
    the identity
    \[
    \sum_{d\geq 0}\,[\Hilb^d(\BA^\infty)_0]t^d =t\cdot \prod_{k=0}^\infty\frac{1}{1-\BL^{k}t}\in \widehat{K}_0(\Var_\BC)\llbracket t \rrbracket.
    \]
\end{corollary}
\begin{proof}
Combine \Cref{thm: stab motives} and \eqref{eqn: inft grass formula} with one  another.
\end{proof}
\Cref{cor:limit-Hilb} is related to the $\BA^1$-homotopy equivalence
\begin{equation}\label{A^1-equivalence}
\Hilb^d(\BA^\infty) \simeq \Gr(d-1,\infty)
\end{equation}
proved in \cite[Thm.~2.1]{Hilb^infinity}, which implies the identity of the two ind-schemes in the category of \emph{Voevodsky's geometrical effective motives} \cite{Gillet-Soulé}, and as such it carries functorial information that is lacking in \Cref{cor:limit-Hilb}. On the other hand, the relation obtained in \Cref{cor:limit-Hilb} is stronger than the K-theoretic specialisation of \cite[Thm.~2.1]{Hilb^infinity}. Note also that the order of truncation $n-d+1$ realised in \Cref{thm: stab motives} agrees with the cohomological degree  under which the canonical homomorphism in integral cohomology
\[
\begin{tikzcd}
\HH^\ast (\Gr(d-1,\infty),\BZ) = \BZ[c_1,\ldots,c_{d-1}] \arrow{r} & \HH^\ast (\Hilb^d(\BA^n),\BZ)
\end{tikzcd}
\]
is an isomorphism, as proved in \cite{Hilb^infinity}. In particular, since $\Hilb^d(\BA^n)_0 \into \Hilb^d(\BA^n)$ is a homotopy equivalence by \cite[Cor.~2.3]{Totaro-Hilb-homotopic}, and $\Hilb^d(\BA^n)_0$ is projective, we deduce an isomorphism
\[
\begin{tikzcd}
    \HH^\ast (\Gr(d-1,\infty),\BZ) \arrow{r}{\sim} & \HH^\ast_c (\Hilb^d(\BA^n)_0,\BZ)
\end{tikzcd}
\]
in (even) degrees $\leq 2n-2d+2$. As a consequence of this isomorphism, we obtain the following.

\begin{corollary}\label{cor:wt-polynomials}
For every $n$ and $d$ there are identities
\begin{align*}
\mathsf w(\Hilb^d(\BA^n)_0,z^{1/2}) 
&= \mathsf p(\Hilb^d(\BA^n)_0,-z^{1/2})\\
&=\frac{\prod_{k=1}^n (z^k-1)}{\prod_{k=1}^{d-1} (z^k-1) \prod_{k=1}^{n-d+1} (z^k-1)}
\end{align*}
in $\BZ[z^{1/2}]/z^{n-d+2}$.
\end{corollary}
This completes the proof of \Cref{MAIN-THM-STAB}.

\section{Explicit computation of some \texorpdfstring{$Y$-classes}{}}\label{sec:consequences of thm A} 
In this section we give explicit formulas for the motives  of the fundamental strata $Y^k_{k,d}$, for all $d>1$ and some values of $k$, cf.~\Cref{prop:easyY}. Notice that for a fixed $d>1$, by  combining Theorems \ref{main-in-body}, \ref{thm:motive-Y-stratum} and \eqref{eq:stratification}, the polynomial  $\mathsf P_d(t)$ can be computed in terms of the motives $[Y_{k,d}^k]$, for $k=1,\ldots,d-1$. We use the results in this section to provide explicit formulas for $\mathsf P_d(t)$ when $d \leq 8$, which we collect in  \Cref{app:Pd-explicit}.

The next lemma computes the motives of infinitely many \emph{reducible} (homogeneous) Hilbert--Samuel strata $\OH_{\bh}^k$, whose associated (nonhomogeneous) Hilbert--Samuel strata $H^k_{\bh}$ are vector bundles by \Cref{lemma:tech1}.

\begin{lemma}\label{lemma:h2}
Let $\bh=(1,k,r,1)$, for $k,r\geq 1$. There is an identity
\begin{equation*}
\left[\OH_{\bh}^k\right] = \sum_{i=1}^k \,[\Gr(i,k)]\left[\Gr\left(r-i,\binom{k+1}{2}-i\right) \right] [\mathscr D_i],
\end{equation*} 
where $[\mathscr D_i]\in K_0(\Var_{\BC})$ are effective classes  recursively defined by
\begin{equation}\label{eq:Di}
[\mathscr D_i] = 
 \begin{cases}
    1 & \mbox{if }i=1,\\
    \left[ \BP^{\binom{i+2}{3}-1}\right] - \displaystyle\sum_{j=1}^{i-1}\,[\Gr(j,i)] [\mathscr D_j] & \mbox{if }i>1.
\end{cases}
\end{equation} 
\end{lemma}
\begin{proof} 
Set $R=\BC[x_1,\ldots,x_k]$. For $[I]\in \OH_{\bh}^k$, let us denote by $f_I\in R^*$ a homogeneous generator of $I^\perp$ of degree 3. We decompose $\OH_{\bh}^k$ as a disjoint union
\[
\OH_{\bh}^k=\coprod_{i=1}^k\OZ_i,
\]
where by \Cref{prop:gorapolar} we can define
\[
\OZ_i=\Set{[I]\in \OH_{\bh}^k | \bh_{\langle f_I \rangle^\perp}(2)=i }.
\]
These loci are locally closed (the proof is identical to \Cref{lem:YnkdLocallyClosed}).

Define, for $i=1,\ldots,k$, the locally closed subsets
\[
\mathscr D_i=\Set{[f]\in \BP( \BC[x_1,\ldots,x_i]_3) | \bh_{\langle f \rangle^\perp}(2)=i }\subset \BP( \BC[x_1,\ldots,x_i]_3).
\]
Clearly, for $i=1,\ldots,k$, we have 
\[
[\OZ_i]=[\Gr(i,k)][\mathscr D_i]\left[\Gr\left(r-i,\binom{k+1}{2}-i\right) \right],
\]
where the first Grassmannian corresponds to the choice of $i$ linear forms and the second corresponds to the choice of $r-i$ quadratic forms. Indeed, an element in $\mathscr Z_i$ is uniquely determined by a choice of a $i$-th dimensional vector subspace $V_i\subset R_1^*$, an element $f$ in the 3-rd power $I_{V_i}^3\subset R^*$  of the ideal generated by $V_i$, such that $\bh_{\langle f \rangle^\perp}(2)=i$ and a $r-i$ dimensional vector subspace of $R_2^*/\langle f\rangle $.  Notice that \eqref{eq:Di} holds true.  Indeed, we have
\[
\mathscr D_i=\BP^{\binom{i+2}{3}-1}\smallsetminus \coprod_{j=1}^{i-1} \mathscr E_{i,j}
\]
where
\[
\mathscr E_{i,j}=\Set{[f]\in \BP (\BC[x_1,\ldots,x_i]_3) | \bh_{\langle f \rangle^\perp}(2)=j }\subset \BP (\BC[x_1,\ldots,x_i]_3).
\]
Finally \eqref{eq:Di} follows observing that $[\mathscr E_{i,j}]=[\Gr(j,i)][\mathscr D_j]$ by \Cref{lemma:tech2}.
\end{proof}

\begin{prop}\label{prop:easyY} 
For every $d > 1$ we have the identities
\begin{align*}
[Y_{1,d}^1]&=1\\
[Y_{2,d}^2]&=[\Hilb^d(\BA^2)_0]-[\BP^1]\BL^ {d-2}\\
[Y_{d-5,d}^{d-5}]&=\left[\Gr\left(4,s_{d-5}\right)\right]+[\BP^{d-6}]\BL^{s_{d-3}-6}+[\Gr(2,d-5)][\BP^2]\BL^{2(s_{d-5}-2)} \\
    &\qquad +\left( [\BP^{d-6}]\left[ \BP^{s_{d-5}-2} \right] -[\Gr(2,d-5)](-\BL^3 - \BL^2 + \BL + 1)  \right)\BL^{s_{d-4}-4} \\
    &\qquad +\big( [\Gr(3,d-5)]\left(\BL^9 + \BL^8 + \BL^7 + \BL^6 - \BL^4 - \BL^3 - \BL^2\right) \\
    &\qquad + [\Gr(2,d-5)]\left(\BL^3+\BL^2\right) \left[\BP^{s_{d-5}-3}\right] + [\BP^{d-6}]\left[ \Gr\left( 2,s_{d-5}-1\right) \right] \big)\BL^{s_{d-5}-3}\\
[Y_{d-4,d}^{d-4}]&=\left[\Gr\left(3,s_{d-4}\right)\right] +[\BP^{d-5}]\BL^{s_{d-3}-3} \\
    &\qquad + \left([\BP^{d-5}]\left[\BP^{s_{d-4}-2}\right]+[\Gr(2,d-4)](\BL^3+\BL^2)\right) \BL^{s_{d-4}-2}  \\
[Y_{d-3,d}^{d-3}]&=\left[\Gr\left(2,s_{d-3}\right)\right]+[\BP^{d-4}]\BL^{s_{d-3}-1}\\
[Y_{d-2,d}^{d-2}]&=\left[\Gr\left(1,s_{d-2}\right)\right]\\
[Y_{d-1,d}^{d-1}]&=1,
\end{align*} 
where $s_k=\binom{k+1}{2}$ is the number of linearly independent quadrics in $k$ variables and we set  $[\BP^i]=0$ for $i<0$.
 \end{prop}
 \begin{proof}
By \eqref{Y-stratification}, to compute the motive $[Y_{k,d}^k]$ it is enough to compute the motives  $[H_{\bh}^k]$ for the relevant Hilbert--Samuel functions.
 
The cases $k=1,2$ are trivial. The remaining five identities are proven by computing the motives of the Hilbert--Samuel strata listed in \Cref{tab:H--Smeno8}.
\begin{table}[H]
    \centering 
    \begin{tabular}{ccccc}
    \hline 
        $Y_{d-1,d}^{d-1}$ &$Y_{d-2,d}^{d-2}$ &$Y_{d-3,d}^{d-3}$ &$Y_{d-4,d}^{d-4}$ &$Y_{d-5,d}^{d-5}$ \\\hline $(1,d-1)$ & $(1,d-2,1)$    &
         $(1,d-3,1,1)$  & $(1,d-4,1,1,1)$  &
         $(1,d-5,1,1,1,1)$\\
         &&$(1,d-3,2)$& $(1,d-4,2,1)$  &$(1,d-5,2,1,1)$
         \\
         &&& $(1,d-4,3)$ & $(1,d-5,2,2)$ \\
         &&&& $(1,d-5,3,1)$   \\
         &&&& $(1,d-5,4)$    
    \end{tabular}
    \caption{Hilbert--Samuel functions occurring in $Y_{d-i,d}^{d-i}$, for $i=1,\ldots,5$.}
    \label{tab:H--Smeno8}
\end{table}
The Hilbert--Samuel functions of the form $\bh=(1,k,r) $ are special instances of \emph{very compressed} functions and the associated strata are known to be isomorphic to the  Grassmannians $\Gr(r,k(k+1)/2)$, see \cite{Iarrobino-78-punti,Hilb_11}. The cases $\bh=(1,k,r,1)$ are obtained by combining \Cref{lemma:h2}, which computes $[\OH_{\bh}^k]$, with \Cref{lemma:tech1}, which computes  $[H_{\bh}^k]$.

For $\bh_1 = (1,k,1,1,1)$ and $\bh_2 = (1,k,1,1,1,1)$ we clearly have $\OH_{\bh_i}^k \cong \BP \BC[x_1,\ldots,x_k]_1\cong \BP^{k-1} $, for $i=1,2$,  by \Cref{prop:gorapolar}.  Moreover, the morphisms $\pi_{\bh_i}^k\colon H_{\bh_i}^k\to\OH_{\bh_i}^k$ are Zariski locally trivial affine bundles of ranks respectively
\begin{equation}\label{eq:rankprrof7.3}
    \rk H_{\bh_1}^k=\frac{(k-1)(k+4)}{2} \,\,\,\mbox{ and }\,\,\, \rk H_{\bh_2}^k=\frac{(k-1)(k+6)}{2}.
\end{equation}
Local triviality is obtained similarly to \Cref{tau-is-ZLT}. To compute the ranks, fix a point $[I]\in \OH_{\bh_1}^k$ (the case $\OH_{\bh_2}^k$ is analogous). Without loss of generality we can suppose 
\[
I=(x_1^5,x_2^2,\ldots,x_k^2)+( x_ix_j\,|\,1\le i<j\le k)+\Fm^5. 
\]
then the fibre of $\pi _{\bh_1}^k$ over $[I]$ parametrises ideals of the form
\[
(x_1^5,x_2^2+a_2 x_1^3+b_2 x_1^4,\ldots,x_k^2+a_k x_1^3+b_k x_1^4)+( x_ix_j+c_{i,j} x_1^3+d_{i,j} x_1^4\,|\,1\le i<j\le k)+\Fm^5,
\]
where, for $h=2,\ldots,k$ and $1\le i<j\le k$, the parameters $a_h,b_h,c_{i,j},d_{i,j}\in\BC$  are compatible with the syzygies of $I$. But the only syzygies of $I$ are of the form
\[
x_1\cdot x_ix_j = x_j \cdot x_1x_i \qquad \mbox{ for } 1 < i\leq j\leq k,
\]
which implies that $a_h=c_{i,j}=0$, for $h\ge 1 $ and $1< i<j\le k$. The equalities in \eqref{eq:rankprrof7.3} then follow by a simple dimension count.

The case $\bh = (1,k,2,1,1)$ is covered by the identity
\[
[H_{\bh}^k] =\left( [\BP^{k-1}]\left[ \BP^{\binom{k+1}{2}-2} \right] -[\Gr(2,k)][\BP^1]^2 \right)\BL^{\frac{(k-1) (k+4)}{2}-1} + [\Gr(2,k)][\BP^1]^2\BL^{\frac{(k-1) (k+4)}{2}},
\] 
whose proof follows by an analysis analogous to the one carried out in the proofs of
\cite[Lemmas 4.5--4.8]{8POINTS}.

Finally, for $\bh=(1,k,2,2)$, we have $[\OH_{\bh}^k] = [\Gr(2,k)][\BP^2]$ by the analysis carried out in the  proof of  \cite[Prop.~3.6]{8POINTS}, and once more we compute the motive  
$[H_{\bh}^k]$ by using \Cref{lemma:tech1}. 
\end{proof}

\begin{remark}\label{rmk:339}
Using \Cref{prop:easyY}, one can compute explicitly $[\Hilb^d(\BA^n)_0]$ for $d\leq 8$ for all $n\geq 1$. For $d=9$, one would only need a formula for the motive $[Y_{3,9}^3]$, whose computation is now reduced to the computation of the motive $[H_{\bh_i}^3]$ for $\bh_1=(1,3,3,1,1)$ and $\bh_2=(1,3,3,2)$, as the other strata can be deduced by similar arguments to the ones employed in the proof of \Cref{prop:easyY}. The difficulty for $\bh_1$ consists in determining a stratification of $H_{\bh_1}^3$ over which which $\pi_{\bh_1}^3$ a Zariski locally trivial fibration. On the other hand, the stratum corresponding to $\bh_2$ is a vector bundle over the homogeneous locus by \Cref{lemma:tech1} and the difficulty lies  in the computation of $\OH_{\bh_2}^3$.

We have computed an upper bound for the dimension of $Y_{3,9}^{5}$, namely 
$$\dim Y_{3,9}^{5} \leq 31 < \dim Y_{5,9}^{5}=36.$$
\end{remark}

\begin{example}\label{punctual-A^3}
Combining \Cref{prop:easyY} with the inversion formula \Cref{lemma:inversion}, we find, for a 3-fold,
\begin{align*}\label{eqn:punctual-A^3}
[\Hilb^6(\BA^3)_0] &= \BL^{10}+3\BL^9+7\BL^8+9\BL^7+9\BL^6+7\BL^5+5\BL^4+3\BL^3+2\BL^2+\BL+1 \\
[\Hilb^7(\BA^3)_0] &= \BL^{12} + 3\BL^{11} + 8\BL^{10} + 14\BL^9 + 16\BL^8 + 14\BL^7 + 11\BL^6 + 7\BL^5 + 5\BL^4 + 3\BL^3 + 2\BL^2 + \BL + 1 \\
[\Hilb^8(\BA^3)_0] &= \BL^{14} + 4\BL^{13} + 12\BL^{12} + 22\BL^{11} + 28\BL^{10} + 27\BL^9 + 21\BL^8 + 15\BL^7 + 11\BL^6 \\ &\qquad +7\BL^5 + 5\BL^4 + 3\BL^3 + 2\BL^2 + \BL + 1.
\end{align*}
And, for a 4-fold,
\begin{align*}
[\Hilb^4(\BA^4)_0] &= \BL^{9} + 2\BL^{8} + 3\BL^{7} + 5\BL^{6} + 4\BL^{5} + 4\BL^{4} + 3\BL^{3} + 2\BL^{2} + \BL + 1\\
[\Hilb^5(\BA^4)_0] &= \BL^{12} + 2\BL^{11} + 4\BL^{10} + 7\BL^{9} + 9\BL^{8} + 9\BL^{7} + 9\BL^{6} + 6\BL^{5} + 5\BL^{4} + 3\BL^{3} + 2\BL^{2} + \BL + 1\\
[\Hilb^6(\BA^4)_0] &= \BL^{15} + 3\BL^{14} + 7\BL^{13} + 13\BL^{12} + 17\BL^{11} + 20\BL^{10} + 20\BL^{9} + 17\BL^{8} + 13\BL^{7} + 10\BL^{6} \\
& \quad + 7\BL^{5} + 5\BL^{4} + 3\BL^{3} + 2\BL^{2} + \BL + 1\\
[\Hilb^7(\BA^4)_0] &= \BL^{18} + 3\BL^{17} + 9\BL^{16} + 19\BL^{15} + 30\BL^{14} + 38\BL^{13} + 44\BL^{12} + 39\BL^{11} + 34\BL^{10} + 26\BL^{9} \\
&\quad + 20\BL^{8} + 14\BL^{7} + 11\BL^{6} + 7\BL^{5} + 5\BL^{4} + 3\BL^{3} + 2\BL^{2} + \BL + 1\\
[\Hilb^8(\BA^4)_0] &= 2\BL^{21} + 5\BL^{20} + 15\BL^{19} + 34\BL^{18} + 55\BL^{17} + 76\BL^{16} + 88\BL^{15} + 87\BL^{14} + 77\BL^{13} + 64\BL^{12} \\
&\quad+ 49\BL^{11} + 38\BL^{10} + 28\BL^{9} + 21\BL^{8} + 15\BL^{7} + 11\BL^{6} + 7\BL^{5} + 5\BL^{4} + 3\BL^{3} + 2\BL^{2} + \BL + 1.
\end{align*}
Notice that $[\Hilb^8(\BA^4)_0]$ is the first motive which is not monic as a  polynomial in $\BL$. This reflects the fact that the punctual Hilbert scheme of $8$ points in $\BA^n$ always has exactly two irreducible components of the same dimension \cite{8POINTS}. One of them is the closure of the curvilinear locus $H_{(1,\ldots,1)}^n$ and the other one is the stratum $H_{(1,4,3)}^n$. On the other hand, the motive $[\Hilb^9(\BA^5)_0]$ is monic as one can immediately see looking at the degree of $[Y_{5,9}^5]$ seen as a polynomial in $\BL$ and estimating the degree of $[Y_{3,9}^5]$ as suggested in \Cref{rmk:339}. This confirms the existence of the \emph{Shafarevich component} \cite{sha90}.
\end{example}

\section{\texorpdfstring{From $\BA^n$ to a smooth $n$-fold}{}}\label{sec:Exp-and-stuff}
\subsection{Hilbert series in the language of power structures}
A \emph{power structure} on a ring $R$ is a rule allowing one to raise a power series $A(t) \in 1+t R\llbracket t \rrbracket$ to an arbitrary element $r \in R$, in such a way that all of the familiar properties of `raising to a power' are satisfied. Gusein-Zade, Luengo and Melle-Hern{\'a}ndez  proved in \cite[Thm.~2]{GLMps} that there is a canonical power structure on $K_0(\Var_{\BC})$, determined by the rule
\[
(1-t)^{-[X]} = \zeta_X(t),
\]
for all effective classes $[X]$. See \cite{GLMps,GLMHilb,ricolfi2019motive} for more details.

We formally introduce, for an arbitrary class $\mu \in K_0(\Var_{\BC})$, the generating function
\[
\zeta_\mu(t) = (1-t)^{-\mu},
\]
which we call the \emph{zeta function} of $\mu$. Attached to a power structure is the \emph{plethystic exponential}. In the case of the Grothendieck ring of varieties, this is the group isomorphism
\begin{equation}\label{Exp-zeta}
\begin{tikzcd}[row sep=tiny]
\Exp\colon tK_0(\Var_{\BC})\llbracket t \rrbracket \arrow{r}{\sim} & 1+tK_0(\Var_{\BC})\llbracket t \rrbracket \\
\displaystyle\sum_{d>0}\,A_dt^d \arrow[mapsto]{r} & \displaystyle\prod_{d > 0}\zeta_{A_d}(t^d).
\end{tikzcd}
\end{equation}
For instance, if $X$ is a $\BC$-variety, one has
\[
\Exp([X]t) = (1-t)^{-[X]} = \zeta_X(t).
\]

\smallbreak
Let us move to motivic generating functions attached to Hilbert schemes.
Fix $n \in \BZ_{>0}$. Consider the generating functions
\begin{align*}
\mathsf{Hilb}_{n,0}(t) &= \sum_{d \geq 0}\,[\Hilb^d(\BA^n)_0] t^d\\
\mathsf{Hilb}_{X}(t) &= \sum_{d \geq 0}\,[\Hilb^d(X)] t^d
\end{align*}
in $1+tK_0(\Var_{\BC})\llbracket t \rrbracket$, where $X$ is a smooth quasiprojective variety of dimension $n$. Note that here we sum over the number of points. Since $\Exp$ is an isomorphism, there are uniquely determined motivic classes $\Omega^{n}_{d} \in K_0(\Var_\BC)$, necessarily effective, such that 
\begin{equation}\label{formula-for-H_n0}
\mathsf{Hilb}_{n,0}(t) = \Exp \left(\sum_{d>0} \Omega^{n}_{d}t^d\right)=\prod_{d>0}\, \zeta_{\Omega^n_d}(t^d).
\end{equation}

\begin{remark}\label{rmk:omega-ps}
If $\Omega^n_d = \sum_j a_j\BL^{e_j}\in\BZ[\BL]$, then
\[
\zeta_{\Omega^n_d}(t) = (1-t)^{-\Omega^n_d} = \frac{\prod_{j\,|\,a_j<0} \, \left(1-\BL^{e_j}t\right)^{a_j}}{\prod_{j\,|\,a_j>0} \, \left(1-\BL^{e_j}t\right)^{a_j}}.
\]
\end{remark}

By the main result of \cite{GLMHilb}, for every smooth quasiprojective $n$-dimensional variety $X$, we have
\begin{equation}\label{eqn:power-struct-formula}
\mathsf{Hilb}_{X}(t) 
= \mathsf{Hilb}_{n,0}(t)^{[X]} 
= \Exp\left([X]\sum_{d>0}\Omega^{n}_{d}t^d\right)
= \prod_{d>0} \,\zeta_{\Omega^{n}_{d}[X]}(t^d).   
\end{equation}

\begin{example}\label{sec:curves-and-surfaces}
We give here the known formulas in the case $n\leq 2$. In this case,  the Hilbert scheme $\Hilb^d(\BA^n)$ is smooth for all $d\geq 0$. We have
\begin{align*}
\mathsf{Hilb}_{n,0}(t) 
&= 
\begin{cases}
    (1-t)^{-1} & \mbox{if }n=1 \\
    \displaystyle\prod_{d>0} (1-\BL^{d-1}t^d)^{-1} & \mbox{if }n=2
\end{cases}\\
&=
\begin{cases}
\Exp(t) & \mbox{if }n=1 \\
\Exp\left(\frac{t}{1-\BL t} \right)& \mbox{if }n=2.
\end{cases}
\end{align*}
The case $n=2$ is G\"{o}ttsche's formula \cite{Gottsche-motivic}. In other words, 
\[
\Omega^{n}_{d} = 
\begin{cases}
1 & \mbox{if }d=n=1 \\
0 & \mbox{if }n=1,\,d>1\\
\BL^{d-1} & \mbox{if }n=2,\,d>0.
\end{cases}
\]
\end{example}

The $\Omega$-classes are, in some sense, the `most important' classes, despite their (apparent) lack of geometric meaning, for they generate the Hilb-classes under Exp. The notation `$\Omega$' is borrowed from motivic Donaldson--Thomas theory. 

\subsection{The Hilbert series of a smooth 3-fold}
By \Cref{punctual-A^3} we know the series $\mathsf{Hilb}_{3,0}(t)$ to order $8$.  Comparing it with the expansion of \eqref{formula-for-H_n0} in the case $n=3$ we obtain 
\begin{align*}
\Omega^3_1 
&= 1,\\
\Omega^3_2 
&= \BL^2+\BL \\
\Omega_3^3 
&= \BL^4+\BL^3+\BL^2 \\
\Omega_4^3
&= \BL^6+2\BL^5 + \BL^4 + \BL^3 - \BL^2 \\
\Omega^3_5 
&= \BL^8+2\BL^7+2\BL^6+\BL^5-\BL^3 \\
\Omega^3_6 
&= \BL^{10} + 3\BL^9 + 4\BL^8 + 3\BL^7 - \BL^6 - 2\BL^5 - 2\BL^4 \\
\Omega^3_7 
&= \BL^{12} + 3\BL^{11} + 5\BL^{10} + 5\BL^9 - 4\BL^7 - 3\BL^6 - \BL^5 + \BL^4 \\
\Omega^3_8 
&= \BL^{14} + 4\BL^{13} + 8\BL^{12} + 10\BL^{11} + 2\BL^{10} - 7\BL^9 - 10\BL^8 - 3\BL^7 + \BL^6 + 2\BL^5.
\end{align*} 
The zeta functions of these $\Omega$-classes can be computed as explained in \Cref{rmk:omega-ps}. By \Cref{eqn:power-struct-formula}, the classes $\Omega^3_1,\ldots,\Omega^3_8$ are enough to compute $[\Hilb^i(X)]$ for every $i \leq 8$ and every smooth 3-fold $X$.

\begin{example}
If $X=\BP^3$, we have
\begin{align*}
\mathsf{Hilb}_{\BP^3}(t)
&= \prod_{i=0}^3\prod_{d>0}\,\zeta_{\Omega^3_d}\left(\BL^it^d\right).
\end{align*}
To compute $[\Hilb^{\leq 8}(\BP^3)]$ we need to extract the coefficient of $t^i$ in
\[
\mathsf{Hilb}_{\BP^3}(t)  \equiv \prod_{i=0}^3\prod_{d=1}^8\,\zeta_{\Omega^3_d}\left(\BL^it^d\right) \pmod{t^9}.
\]
For instance, we find
\begin{align*}
[\Hilb^5(\BP^3)] 
&= \BL^{15} + 2 \BL^{14}+ 7 \BL^{13} + 17 \BL^{12} + 39 \BL^{11}+ 67 \BL^{10} + 97 \BL^9     +  114 \BL^8 \\
&\quad + 111 \BL^7+ 90 \BL^6 + 59 \BL^5 + 33 \BL^4+ 14 \BL^3 + 6 \BL^2+ 2 \BL  + 1   , \\
[\Hilb^6(\BP^3)] 
&= \BL^{18} +2\BL^{17}+7\BL^{16}+18\BL^{15}+45\BL^{14}+92\BL^{13}+167\BL^{12}+242\BL^{11}+ 306\BL^{10} \\
&\quad +316\BL^9+282\BL^8 +206\BL^7+131\BL^6+68\BL^5+32\BL^4+14\BL^3+6\BL^2+2 \BL+  1 ,\\
[\Hilb^7(\BP^3)] 
&=  \BL^{21}+2\BL^{20} +7\BL^{19} +18\BL^{18} +47\BL^{17}+105\BL^{16} +220\BL^{15} +385\BL^{14}+587\BL^{13}   \\
&\quad+761\BL^{12}+843\BL^{11}   +799\BL^{10}+647\BL^{9}+449\BL^{8}+266\BL^{7} +142\BL^{6} +66\BL^{5}   \\ & \quad +32\BL^{4}+14\BL^{3}+6\BL^{2}+2\BL + 1   ,\\
[\Hilb^8(\BP^3)] 
&=\BL^{24}
 + 2 \BL^{23} 
 + 7 \BL^{22} 
 + 18 \BL^{21} 
 + 48 \BL^{20} 
 + 111 \BL^{19} 
 + 251 \BL^{18} 
 + 498 \BL^{17} 
 +891 \BL^{16}  \\
&\quad 
 + 1368 \BL^{15} 
 + 1847 \BL^{14}  
 + 2132 \BL^{13} 
 + 2150 \BL^{12} 
 + 1853 \BL^{11}
 + 1395 \BL^{10} 
 + 904 \BL^9 \\
&\quad + 
 522 \BL^8 + 272 \BL^7+ 136 \BL^6+ 66 \BL^5 + 32 \BL^4+ 14 \BL^3 + 6 \BL^2+ 2 \BL   +1   .
\end{align*}
\end{example}

\section{Further directions and conjectures}\label{sec:expectation} 
\subsection{Motives of punctual Quot schemes}\label{sec:quot}
Fix $r,n,d>0$. Let us consider the \emph{punctual Quot scheme}
\[
\Quot_{\BA^n}(\OO^{\oplus r},d)_0 \subset \Quot_{\BA^n}(\OO^{\oplus r},d),
\]
parametrising isomorphism classes of quotients $\OO^{\oplus r} \onto F$, with $F$ a $0$-dimensional coherent sheaf on $\BA^n$ (entirely supported at the origin) such that $\chi(F)=d$.
This scheme only depends on $\widehat{\OO}_{\BA^n,0}$, see e.g.~\cite[Thm.~B]{Fantechi-Ricolfi-motivic}. 

We expect higher rank formulas structurally similar to those of \Cref{MAIN-THEOREM-1} to hold. We mostly leave this for future research, but we motivate this expectation with a series of examples, leading to the proof of \Cref{MAIN-THM-QUOT}. Consider, for $d \in \BZ_{>0}$, the generating function
\[ 
\mathsf{Quot}_d (x,y)=\sum_{{r}\ge 0}\sum_{n\ge 0} \,\bigl[\Quot_{\BA^{n+1}}(\OO^{\oplus r+1},d)_0 \bigr]x^ny^r.
\]
A direct computation in the cases $d=1,2$ gives 
\begin{align*}
      \mathsf{Quot}_1 (x,y)&=\sum_{{r}\ge 0}\sum_{n\ge 0} \,[\BP^{r} ]x^ny^r = (1-x)^{-1}\zeta_{\BP^1}(y) = \zeta_{\BP^0}(x)\zeta_{\BP^1}(y)\\
      \mathsf{Quot}_2 (x,y)&=\sum_{{r}\ge 0}\sum_{n\ge 0} \,([\BP^{r}][\BP^n]\BL^{r}+[\Gr(2,r+1)])x^ny^r=\zeta_{\BP^1}(x)\zeta_{\BP^2}(y)(1-\BL xy).
\end{align*} 
We next compute $\mathsf{Quot}_d(x,y)$ for $d=3,4$. In this case a more detailed analysis is required. We exploit the general stratification of the punctual Quot scheme given by the locally closed subsets
\begin{equation}\label{Y-for-MODULES}
Y^{n,r}_{s,d}=\Set{[\OO^{\oplus r} \onto F]\in \Quot_{\BA^n}(\OO^{\oplus r},d)_0|   \bh_F(0)=s}\subset \Quot_{\BA^n}(\OO^{\oplus r},d)_0,
\end{equation}
indexed by $s=1,\ldots,d$. The computations in the rest of this subsection are obtained via the so-called \emph{apolarity for modules} \cite{Joachim-Sivic} and using the computer software Macaulay2 \cite{M2,HilbQuotPaoloLella}.   

\subsubsection{The case $d=3$}
The schematic support of a module of length $3$ has embedding dimension at most $2$, and there are six isomorphism classes of monomial $\BC[x_1,x_2]$-modules of length 3 (see \cite{MR18}) distributed in four Hilbert--Samuel strata. The complete classification of the strata \eqref{Y-for-MODULES} is given in \Cref{tab:stratQuot}.

\begin{table}[h!]
    \centering
    \begin{tabular}{cccc}
       $s$  & H--S functions & $\chi(Y^{n,r}_{s,3})$ &    $[Y^{n,r}_{s,3}]$ \\
\toprule
         1 & (1,1,1),(1,2) & $\binom{n+1}{2}r$  &     $ [\Hilb^3(\BA^n)_0 ][\BP^{r-1}] \BL^{2(r-1)}$ \\
         \hline 
         2 & (2,1) & $2n\binom{r}{2}$  & $[\Gr(2,r)] [\BP^{2n-1}] \BL^{r-2} $ \\
         \hline 
         3 & (3) & $\binom{r}{3}$  & $[\Gr(3,r)]$  \\
         \hline 
    \end{tabular}
    \caption{Euler characteristics and motives of the strata $Y_{s,3}^{n,r}$ for $s=1,2,3$.}
    \label{tab:stratQuot}
\end{table}
Summing over the entries of the last column in \Cref{tab:stratQuot} (and then over $n$ and $r$) gives
\begin{equation}\label{Quot_3}
\mathsf{Quot}_3 (x,y)=\zeta_{\BP^2}(x)\zeta_{\BP^3}(y)(1-\BL xy)(1-\BL^2xy).
\end{equation}

\subsubsection{The case $d=4$}
In order to compute the generating function $\mathsf{Quot}_4(x,y)$ we computed again the required Hilbert--Samuel strata. They are listed in \cref{tab:stratQuot4}. 

\begin{table}[H]
    \centering
    \begin{tabular}{cccc}
       $s$  & H--S functions & $\chi(Y^{n,r}_{s,4})$ &    $[Y^{n,r}_{s,4}]$\\
         \toprule 
         1 & $(1,1,1,1),(1,2,1),(1,3)$ & $\left(n+3\binom{n}{2} +\binom{n}{3} \right)r$   &     $ [\Hilb^4(\BA^n)_0 ][\BP^{r-1}] \BL^{3(r-1)}$ \\
         \hline 
         2 & (2,1,1) & $2n\binom{r}{2}$   &$[\Gr(2,r)][\BP^{n-1}][\BP^1]\BL^{2(n+r-2)-1}$   \\
         \hline 
         2 & (2,2) &  $\binom{2n}{2}\binom{r}{2}$  &$[\Gr(2,r)][\Gr(2,2n)]\BL^{2(r-2)}$   \\
         \hline 
         3 & (3,1) & $3n\binom{r}{3}$  &$[\Gr(3,r)][\BP^{3n-1}]\BL^{r-3}$    \\
         \hline 
         4 & (4) & $\binom{r}{4}$  & $[\Gr(4,r)]$  \\
         \hline 
    \end{tabular}
    \caption{Euler characteristics and motives of the strata $Y_{s,4}^{n,r}$ for $s=1,2,3,4$.}
    \label{tab:stratQuot4}
\end{table}
Summing over the entries of the last column in \Cref{tab:stratQuot4} (and then over $n$ and $r$) gives
\begin{equation}\label{quot_4}
\mathsf{Quot}_4 (x,y) 
=  \zeta_{\BP^3}(x)\zeta_{\BP^4}(y)\zeta_{\BL^4}(x)\mathsf U_4(x,y),
\end{equation}
where $\mathsf U_4(x,y) \in \BZ[\BL,x,y] \subset K_0(\Var_{\BC})[x,y]$ is the explicit polynomial 
\begin{multline}\label{U-polynomial}
    \mathsf U_4(x,y) =
     \BL^{10} x^{4} y^{3} - \BL^{9} x^{3} y^{3} + \BL^{6} x^{3} + \BL^{3} x y^{2} 
-{(\BL^{6} + \BL^{5})} x^{3} y^{2} \\
+ {(\BL^{9} - \BL^{6})} x^{2} y^{3} 
- {(\BL^{8} + \BL^{7} + \BL^{6} - \BL^{5})} x^{3} y \\
- {(\BL^{9} + \BL^{8} + \BL^{7} - \BL^{6} - 2\BL^{5} - \BL^{4})} x^{2} y^{2} 
+ {(\BL^{8} + 2\BL^{7} + 2\BL^{6} + \BL^{3} + \BL^{2})} x^{2}y \\
- {(\BL^{6} + \BL^{2})} x^{2} 
- {(2\BL^{3} + 2\BL^{2} + \BL)} x y - {(\BL^{4} - \BL^{2})} x + 1.
\end{multline}
Summing up, we have the explicit formulas
\[
\mathsf{Quot}_d (x,y) = 
\begin{cases}
\zeta_{\BP^0}(x)\zeta_{\BP^1}(y) & \mbox{if }d=1\\
\zeta_{\BP^1}(x)\zeta_{\BP^2}(y)(1-\BL xy) & \mbox{if }d=2\\
\zeta_{\BP^2}(x)\zeta_{\BP^3}(y)(1-\BL xy)(1-\BL^2xy) & \mbox{if }d=3 \\
\zeta_{\BP^3}(x)\zeta_{\BP^4}(y)\zeta_{\BL^4}(x)\mathsf U_4(x,y) & \mbox{if }d=4.
\end{cases}
\]
This completes the proof of \Cref{MAIN-THM-QUOT} from the introduction. At the level of Euler characteristics, we have
\[
\chi\mathsf{Quot}_d (x,y) = 
\begin{cases}
\frac{1}{(1-x)(1-y)^2}   & \mbox{if }d=1 \\
\frac{1-xy}{(1-x)^2(1-y)^3}    & \mbox{if }d=2 \\
\frac{(1-xy)^2}{(1-x)^3(1-y)^4}    & \mbox{if }d=3\\
\frac{-x^3y^3+(2 x^2 + x) y^2+(2 x^2 -5 x )y-x^2+x+1}{(1-x)^4(1-y)^5} & \mbox{if }d=4.
\end{cases}
\]

\subsubsection{Explicit elementary components in the case $d=4$}
The Quot scheme of points $\Quot_{\BA^n}(\OO^{\oplus r},d)$ contains several interesting loci. The first one is
the \emph{standard component} 
\[
W_{r,n,d}^{\mathsf{st}} \subset \Quot_{\BA^n}(\OO^{\oplus r},d),
\]
namely the closure of the locus of quotients supported on $d$ distinct points. It has dimension $(n+r-1)d$. Another interesting locus is the closure of the \emph{curvilinear locus} $H^{n,r}_{(1,\ldots,1)}$, which has dimension $d(r-1)+(d-1)(n-1)$ and is contained in the standard component. The motive of the curvilinear locus is 
\[
[H^{n,r}_{(1,\ldots,1)}]=[\BP^{r-1}]\BL^{(d-1)(r-1)}[\BP^{n-1}]\BL^{(d-2)(n-1)}. 
\]
Let us go back to the case $d=4$.
The Quot scheme of $4$ points $\Quot_{\BA^n}(\OO^{\oplus r},4)$ is known to have exactly two irreducible components as soon as $r\geq 2$ and $n\geq 4$, and to be irreducible in the other cases \cite[Table 2]{Joachim-Sivic}. In all the reducible cases, the unique nonstandard component is \emph{elementary}, i.e.~it parametrises quotients entirely supported on a single point. As an application of our analysis we get the following result.

\begin{theorem}
Fix integers $n \geq 4$, $r \geq 2$. Let $V_{r,n}\subset \Quot_{\BA^n}(\OO^{\oplus r},4)$ be the unique elementary component. Then $V_{r,n}$ is the closure of the Hilbert--Samuel stratum $H_{(2,2)}^{n,r}$. As a consequence, 
\[
\dim V_{r,n} = 4(r+n-3).
\]
\end{theorem}

\begin{proof}
A direct check\footnote{This check can be for instance performed via the Macaulay2 package \href{http://www.paololella.it/publications/lm2/}{\tt HilbertAndQuotSchemesOfPoints.m2} developed by P.~Lella \cite{HilbQuotPaoloLella}.} shows that the generic point of $H_{(2,2)}^{n,r}$ has trivial negative tangents (in the sense of \cite{ELEMENTARY} and \cite[Eqn.~(5.4)]{Joachim-Sivic}). This shows that the closure of $H_{(2,2)}^{n,r}$ is an elementary component, which must then agree with $V_{r,n}$. Its dimension is computed via \Cref{tab:stratQuot4}. 
\end{proof}

\begin{remark}
Thanks to our motivic computations, we can immediately spot that the irreducible component $V_{r,n}$ has dimension strictly greater than the dimension of the closure of $H_{(1,1,1,1)}^{n,r}$ starting from $n=6$ and for any $r\ge 2$.
\end{remark}

\subsubsection{A worked out example for the Quot scheme of $\BP^3$}\label{subsubsec:Quot-P^3}
Fix $n,r>0$. Define the generating function
\[
\mathsf{Quot}_{n,r,0}(t) = \sum_{d \geq 0}\,[\Quot_{\BA^n}(\OO^{\oplus r},d)_0] t^d.
\]
Following \Cref{sec:Exp-and-stuff}, one can define (effective) classes $\Omega^{n,r}_{d}\in K_0(\Var_{\BC})$ such that
\[
\mathsf{Quot}_{n,r,0}(t) = \Exp\left( \sum_{d > 0}\Omega^{n,r}_{d} t^d\right).
 \]
This determines the motives $[\Quot_X(E,d)]\in K_0(\Var_{\BC})$ for any locally free sheaf $E$ of rank $r$ on a smooth $n$-dimensional quasiprojective variety $X$, via the power structure formula \cite[Thm.~A]{ricolfi2019motive}
\[
\sum_{d \geq 0}\,[\Quot_X(E,d)] t^d = \mathsf{Quot}_{n,r,0}(t)^{[X]}.
\] 
Note the independence of the right hand side on the bundle $E$. To fix ideas, set $n=r=3$. The four $\Omega$-classes that one can extract from \Cref{quot_4} are
\begin{align*}
    \Omega^{3,3}_1 &= [\BP^2]\\
    \Omega^{3,3}_2 &= [\BP^2](\BL+1)\BL^3\\
    \Omega^{3,3}_3 &= [\BP^2](\BL^4+\BL^3+2\BL^2-1)\BL^4\\
    \Omega^{3,3}_4 &= (\BL^6+2\BL^5+3\BL^4+3\BL^3-2\BL^2-2\BL-1)\BL^6.
\end{align*}
If we set $X=\BP^3$, exponentiating we immediately get the formulas
\begin{align*}
[\Quot_{\BP^3}(E,2)] &= \BL^{10} + 3\BL^9 + 9\BL^8 + 14\BL^7 + 20\BL^6 + 19\BL^5 + 17\BL^4 + 10\BL^3 + 6\BL^2 + 2\BL + 1, \\
[\Quot_{\BP^3}(E,3)] &= \BL^{15} + 3\BL^{14} + 12\BL^{13} + 30\BL^{12} + 58\BL^{11} + 88\BL^{10} + 111\BL^9 + 114\BL^8 + 99\BL^7 \\
&\quad+ 75\BL^6 + 47\BL^5 + 27\BL^4 + 14\BL^3 + 6\BL^2 + 2\BL + 1, \\
[\Quot_{\BP^3}(E,4)] &= \BL^{20} + 3\BL^{19} + 13\BL^{18} + 39\BL^{17} + 102\BL^{16} + 202\BL^{15} + 346L^{14} + 480\BL^{13} \\
&\quad+ 581\BL^{12} + 590\BL^{11} + 533\BL^{10} + 415\BL^9 + 297\BL^8 + 187\BL^7 
+ 113\BL^6 \\
&\quad+ 60\BL^5 + 32\BL^4 + 14\BL^3 + 6\BL^2 + 2\BL + 1
\end{align*}
for any locally free sheaf $E$ of rank $3$ on $\BP^3$. We stress that with the data at our disposal --- i.e.~the classes $\Omega^{r,n}_{\leq 4}$ --- this procedure can be performed for any pair $(X,E)$.

\subsection{A conjectural formula for the generating function of \texorpdfstring{$\Omega$}{}-classes}
Recall the $\Omega$-classes defined via \Cref{formula-for-H_n0}. Consider the generating function
\[
\Omega_d(t) = \sum_{n\geq 0}\,{{\Omega_d^{n+2}}}  t^n \,\in\,K_0(\Var_{\BC})\llbracket t \rrbracket.
\] 
We propose the following conjecture, which should be regarded as the analogue of \Cref{MAIN-THEOREM-1} for $\Omega$-classes. 

\begin{conjecture}\label{conjecture-on-Omega-GF}
Fix $d\geq 1$. Then 
\begin{itemize}
\item [\mylabel{it:omega-1}{\normalfont{(1)}}] if $d=1,2,3$, we have 
\[
\Omega_d(t) = \zeta_{\BP^{d-1}}(t)\cdot\BL^{d-1},
\]
\item [\mylabel{it:omega-2}{\normalfont{(2)}}] if $d>3$, we have
\[
\Omega_d(t) = \zeta_{\BP^{d-1}}(t) \cdot \mathsf{Q}_d(t),
\]
where $\mathsf{Q}_d(t)\in K_0(\Var_{\BC})[t]$ takes the explicit form
\[
\mathsf{Q}_d(t) = \sum_{i=0}^{d-3} \left(\sum_{j=0}^{i}\,(-1)^{j} \Omega^{i-j+2}_d[\Gr(j,d)]\BL^{\binom{j}{2}}\right)t^i.
\]
\end{itemize} 
\end{conjecture}
For $d\ge 3$, $m  \ge d $, the conjecture is easily seen to be equivalent to the recursive formula 
\begin{equation}\label{FormulaOmega}
    \Omega_d^m=\sum_{\gamma =1}^{d-1}(-1)^{d+\gamma+1}\Omega_d^\gamma [\Gr(m-d,m-\gamma-1)][\Gr(\gamma,m)]\BL^{\binom{d-\gamma}{2}},
\end{equation}
by an argument similar to  the chain of identities proving \Cref{cor:hilbrel}. Exploiting the motivic formulas presented in \Cref{app:Pd-explicit}, we were able to check the conjecture for $d\leq 8$. Explicit formulas for the polynomials $\mathsf Q_{\leq 8}(t)$ are collected in \Cref{subsec:motiveomega}.

\begin{prop}
    \Cref{conjecture-on-Omega-GF} holds true for $d\leq 8$.
\end{prop}
 
\subsection{Numerical and motivic failure of MacMahon's prediction}\label{sec:higher-dim-partitions}
In this section we revisit the celebrated guess by MacMahon on the generating function of higher dimensional partitions. We first do so in the language of the $\Omega$-classes, defined in \Cref{formula-for-H_n0}, and then at the level of $Y$-classes (and their Euler characteristic), highlighting the relationship with the existing literature on the subject \cite{Atkin}.

For the sake of completeness, we give a short summary on the theory of partitions, see e.g. \cite[Ch.~7]{RPStanley}.

\begin{definition}\label{def:partition}
Let $n\geq 1$ and $d \geq 0$ be integers. An $(n-1)$-dimensional partition of $d$ is a collection of $d$ points $\CA=\set{\ba_1,\ldots,\ba_d}\subset \BZ_{\geq 0}^{n}$ with the following property: if $\ba_i = (a_{i1},\ldots,a_{in}) \in \CA$, then whenever a point $\by = (y_1,\ldots,y_n)\in\BZ_{\geq 0}^{n}$ satisfies $0\leq y_j\leq a_{ij}$ for all $j=1,\ldots,n$, one has $\by \in \CA$. The integer $d = \lvert \mathcal A \rvert$ is called the \emph{size} of the partition. 
\end{definition}

Notice that $n=1$ is a trivial case: it produces only one object for each $d \geq 0$.
The case $n=2$ is the theory of \emph{Young diagrams} (ordinary partitions), whereas $n=3$, resp.~$n=4$, corresponds to the theory of \emph{plane partitions}, resp.~\emph{solid partitions}. Plane partitions were introduced by MacMahon in \cite[Art.~43]{MacMahon-def-PP}, see also \cite[Ch.~IX, X]{MacMahon-Comb-anal} for an account of his results.

Let $p_{n-1}(d)$ denote the number of $(n-1)$-dimensional partitions of size $d$. Consider the generating function $\Pi_n(t) = \sum_{d \geq 0}\,p_{n-1}(d)t^d$. The identities
\begin{equation}\label{Pi-small-n}
\Pi_n(t) = 
\begin{cases}
    (1-t)^{-1} & \mbox{if }n=1 \\
    \prod_{d>0}\,\left(1-t^d\right)^{-1} & \mbox{if }n=2 \\
    \prod_{d>0}\,\left(1-t^d\right)^{-d} & \mbox{if }n=3 
\end{cases}
\end{equation} 
are well-known.\footnote{The case $n=2$ is due to Euler: see \cite[Ch.~16]{Euler_partitions} if you read latin, or \cite[Thm.~1.1]{GEAndrews} for a modern account; the case $n=3$ is due to MacMahon, see \cite[Cor.~7.20.3]{RPStanley} for a modern treatment.} No such product expansions currently exists for $\Pi_n(t)$ if $n>3$. MacMahon suggested\footnote{In fact, MacMahon might have known his suggestion to be incorrect, cf.~\cite[vol.~2, footnote on p.~175]{MacMahon-Comb-anal}.} the general identity
\begin{equation}\label{eqn:MacMahon-suggestion}
\Pi_n(t) \,\overset{?}{=}\, \prod_{d>0}\,\left(1-t^d\right)^{-\binom{d+n-3}{n-2}}.
\end{equation}
The right hand side of the formula, which is clearly compatible with \eqref{Pi-small-n} for $n=2,3$, does in fact compute the correct coefficients of $\Pi_n(t)$ for $d \leq 5$ also for $n>3$, but fails for $d \geq 6$, see \cite{Atkin,Wright-partitions}. It was recently proved in \cite[Thm.~1.1]{Amanov-Yeliussizov} that the right hand side of \eqref{eqn:MacMahon-suggestion} enumerates partitions of given \emph{corner-hook volume}. The theory of partitions connects to the theory of Hilbert schemes via the well-known identities
\[
\Pi_n(t) =  \chi \mathsf{Hilb}_{n,0}(t). 
\]
On the other hand, by the very definition of the $\Omega$-classes, we have
\begin{equation}\label{eqn:chi-hilb}
\chi \mathsf{Hilb}_{n,0}(t) 
= \chi \prod_{d>0} \left( 1-t^d \right)^{-\Omega^n_d} 
= \prod_{d>0} \left(1-t^d \right)^{-\chi(\Omega^n_d)},  
\end{equation}
where the last identity uses that $\chi \colon K_0(\Var_{\BC}) \to \BZ$ is a homomorphism of pre-$\lambda$-rings, cf.~\Cref{rmk:pre-lambda-maps}.
Comparing \eqref{eqn:chi-hilb} with \eqref{eqn:MacMahon-suggestion}, we deduce that 
\[
\chi(\Omega^n_d) = \binom{d+n-3}{n-2}
\]
if either $n=2,3$ or $d\leq 5$. It is natural to ask whether the `error'
\begin{align}\label{eqn: error num}
    \varepsilon_d(n) = \binom{d+n-3}{n-2}-\chi(\Omega^n_d) 
\end{align}
hides some kind of regularity.
By using the results in \cite{GOVINDARAJAN2013600}, we have computed $\varepsilon_d(n)$ for $d\le 26$. This led us to formulate the following conjecture.
\begin{conjecture}\label{conj:erroreMac}
For every $d\ge 6 $ there exists an \emph{irreducible} polynomial $r_d(t)\in \BQ[t]$  of degree $d-6$ such that
\[
\varepsilon_d(n)=\binom{n}{4}r_d(n)
\]
for all $n\ge 1$. 
\end{conjecture}
We have verified \Cref{conj:erroreMac} for $d\le 26$, using the computer softwares SAGE and Macaulay2 \cite{sagemath,M2} and the database of partitions \cite{ThePartitionsProjectWebsite,Indiani-asintotici}. See  \Cref{app:macmahon-polynomials} foe explicit formulas for $\varepsilon_{\leq 26}(n)$.

\begin{prop}\label{thm:conjfino26}
\Cref{conj:erroreMac} holds true for $d\le 26$.
\end{prop}

Define the motivic classes
\[
\overline{\Omega}_d^n = [\Gr(n-2,d-3+n)] \BL^{d-1}\,\in\,K_0(\Var_{\BC}).
\]
A naive motivic version of the prediction  \eqref{eqn:MacMahon-suggestion} is
\begin{equation}\label{eq:motiviconjfalse}
\Omega_d^n\,\overset{?}{=}\,\overline{\Omega}_d^n.
\end{equation}
Clearly, for $n\geq 4$, equality \eqref{eq:motiviconjfalse} is false for $d\ge 6$, but surprisingly it starts failing already for $d=4$ and $n=3$. In other words, \eqref{eq:motiviconjfalse} holds true if and only if $\Hilb^d(\BA^n)$ is smooth. 

In analogy to \eqref{eqn: error num}, we introduce the refined version of the error $\varepsilon_d(n)$, namely
\[
\varepsilon^{\mathsf{mot}}_d(n) = \overline{\Omega}^n_d-\Omega^n_d \,\in\, K_0(\Var_{\BC}).
\]
We record in \Cref{table: mot err Omega} this motivic discrepancy for small values of $d$ and $n$. Notice that, as it should be, one has $\varepsilon^{\mathsf{mot}}_d(n)|_{\BL=1}= 0$ except in the case $(d,n)=(6,4)$.

\begin{table}[h!]
    \centering 
    \begin{tabular}{cl}
    $(d,n)$ & \hspace{2.6em} $\varepsilon^{\mathsf{mot}}_d(n)$\\
    \toprule
    $(4,3)$ &  $\hspace{1.6em}(1-\BL)[\BP^2]\BL^2$ \\
    $(4,4)$ &  $\hspace{1.6em}(1-\BL)[\BP^2](\BL + 1)(\BL^2 + 1)\BL^2$ \\
    $(5,3)$ &  $\hspace{1.6em}(1-\BL)[\BP^2](\BL + 1)\BL^3$ \\
    $(5,4)$ &  $ \hspace{1.6em}(1-\BL)[\BP^2](\BL^2 + \BL + 1)(\BL + 1)(\BL^2 + 1)\BL^3$ \\
    $(6,3)$ &  $\hspace{1.6em}(1-\BL)[\BP^2](2\BL^2 + 3\BL + 2)\BL^4$\\
    $(6,4)$ &  $1+(1-\BL)[\BP^2]\big((2\BL^6 + 5\BL^5 + 6\BL^4 + 5\BL^3 + 3\BL^2 + \BL - 1)(\BL^2 + 1)\BL^3-1 \big)$\\
    \end{tabular}
    \caption{Some motivic discrepancies $\varepsilon^{\mathsf{mot}}_d(n)$.}
    \label{table: mot err Omega}
\end{table}

A possible approach towards estimating the error in MacMahon's prediction is to compute the discrepancy between the coefficients of the generating function
\[
\overline{\mathsf{Hilb}}_{n,0}(t)=\Exp \left(\sum_{d > 0}\overline{\Omega}_d^nt^d\right) = \sum_{d\geq 0}\,\overline{\Hilb}_d^n\cdot t^d
\]
and the coefficients of $\mathsf{Hilb}_{n,0}(t)$. A numerical version of this approach was proposed for instance in \cite{Atkin}, where several related questions were raised. 

It is worth  mentioning that the collections of motivic classes $(\overline{\Hilb}^n_d)_{n,d}$, resp.~$(\overline{\Omega}^n_d)_{n,d}$, verify the relations in \Cref{cor:hilbrel}, resp.~\Cref{FormulaOmega}. 
 
By formally applying the inversion formula (\Cref{lemma:inversion}), we can define motives $\overline{Y}_{d}^k$ satisfying
\[
\overline{\Hilb}^n_d=\sum_{k=1}^{d-1}\, [\Gr(k,n)]\BL^{(n-k)(d-k-1)}\overline{Y}^k_{ d}.
\]
Although the motives $\overline{Y}_d^k$ verify the numerical identity $\chi (\overline{Y}_d^{d-1})=\chi (Y_{d-1,d}^{d-1})=1$ \cite{Atkin}, it is not true that $\overline{Y}_d^{d-1}=1$ while, we have $[Y_{d-1,d}^{d-1}]=1$. 

Consider the motivic error 
\[
e_{k,d}^{\mathsf{mot}}= \overline{Y}_{ d}^k -[{Y}_{k,d}^k]\,\in\,K_0(\Var_{\BC}). 
\]
Its Euler characteristic $\chi(e_{k,d}^{\mathsf{mot}})$ agrees with the number denoted by $e_{k,d}$ in \cite{Atkin}. We record in \Cref{table: mot err HILB} some of these motivic discrepancies, which, as expected, vanish in $\BL=1$ except for the case $(d,k)=(6,4)$.
\begin{table}[h!]
    \centering 
    \begin{tabular}{cl}
    $(d,k)$ & \hspace{2.6em} $e^{\mathsf{mot}}_{k,d}$\\
    \hline
    $(4,3)$ &  $\hspace{1.6em}(1-\BL)[\BP^2]\BL^2$ \\
    $(5,3)$ &  $\hspace{1.6em}(1-\BL)[\BP^2](\BL^2 + \BL + 1)\BL^2$ \\
    $(5,4)$ &  $\hspace{1.6em}(1-\BL)[\BP^2](\BL+1)(\BL^2+1)\BL^2$ \\
    $(6,3)$ &  $\hspace{1.6em}(1-\BL)[\BP^2](2\BL^2 + \BL + 1)(\BL^2 + \BL + 1)\BL^2$ \\
    $(6,4)$ &  $1-(1-\BL)[\BP^2](\BL^9 - \BL^8 - 3\BL^7 - 5\BL^6 - 6\BL^5 - 5\BL^4 - 2\BL^3 - \BL^2 + 1)$\\
    $(6,5)$ &  $\hspace{0.9em}-(1-\BL)[\BP^4](\BL^5 - \BL^4 - \BL^3 - 3\BL^2 - 2\BL - 1)\BL^2$\\ 
    \end{tabular}
    \caption{Some motivic discrepancies $e^{\mathsf{mot}}_{k,d}$.}  
    \label{table: mot err HILB}
\end{table}

Up to date, little is known\footnote{To quote \cite[A007326]{oeis}, referred to the sequence $(e_{0,d})_d$: ``Understanding this sequence is a famous unsolved problem in the theory of partitions''.} about the structure of the discrepancy $e_{k,d}$. Based on our computational evidence, we propose the following conjecture.
\begin{conjecture}\label{conj: Andrews}
For every $k\ge 0 $ there exists a polynomial $\mathsf M_k(t)\in\BQ[t]$ such that
   \[
\sum_{i\ge 0} e_{6+k+i,4+i} t^i=\frac{\mathsf M_k(t)}{(1-t)^{2k+3}\prod_{i=0}^k(1-(2+i)t)^{k+1-i}}.
   \]

\end{conjecture}
To our knowledge, only one of the sequences $(e_{d-2,d+k})_d$, namely the one corresponding to $k=0$, is recorded in OEIS \cite[A002662]{oeis}.

We checked numerically \Cref{conj: Andrews} for low orders. More precisely, we shall exhibit in \Cref{subs:macmahondiscrepancy} three polynomials $\mathsf M_0, \mathsf M_1$ and $\mathsf M_2$ which confirm \Cref{conj: Andrews} to the maximum order testable thanks to \cite{ThePartitionsProjectWebsite}.

\subsection{Determining \texorpdfstring{$\mathsf P_d(t)$}{} with less data}\label{subsec:indiani} 
In this subsection we motivate \Cref{conj:halfdata} by introducing a stratification inspired by \cite{GOVINDARAJAN2013600}. We begin with a quick overview on the ideas behind the proof of the following result.

\begin{theorem}[Govindarajan {\cite{GOVINDARAJAN2013600}}]\label{thm:indiani}
The number of \emph{new data} needed to compute the 
 number $p_{n-1}(d)$ of $(n-1)$-dimensional partitions of size $d$  for every $n$ is bounded above by $\left\lceil\frac{d-1}{2}\right\rceil$.
\end{theorem}

The phrasing `new data' means that we work under the assumption of having already computed $p_{n-1}(e)$ for $1 \leq e < d$ and every $n$. 

The number $p_{n-1}(d)$ can be computed as a sum of individual contributions, each corresponding to $(n-1)$-dimensional partitions of `embedding dimension' precisely $k$, for $k=1, \dots, d-1$. This is the numerical counterpart of our stratification of $\Hilb^d(\BA^n)_0$ into $Y$-strata. Notice, in fact, that specialising $\BL=1$ in \Cref{thm:motive-Y-stratum} and \eqref{eq:stratification} recovers the well-known relation between partitions and partitions of exact embedding dimension, see e.g. \cite{Atkin}.

One of the key steps in the proof of \Cref{thm:indiani} is to subdivide the collection of $(n-1)$-dimensional partitions of embedding dimension exactly  $k$ into a disjoint union of smaller subsets, 
and then to compute their cardinality individually. Following this idea we produce a stratification of $Y_{k,k+d}^d$ whose strata are the geometric counterpart of the subsets introduced in \cite{GOVINDARAJAN2013600}.  

The stratification goes as follows. Let $[I]\in \Hilb^d(\BA^n)_0$ be the ideal of a fat point and let $\Fm_A$ denote the maximal ideal $\Fm_A=\Fm/I\subset A=R/I$. Denote by $\iota(I)$ the minimum number of linearly independent classes of linear forms $\ell_1,\ldots,\ell_i\in A$ satisfying $\Fm_A^2/\Fm_A^3 = (\ell_1,\ldots,\ell_i)^2/\Fm_A^3$.
Consider the stratification
\begin{equation} \label{eqn:stratificationC}
Y^k_{k,k+e+1}=\coprod_{i=1}^{2e} C_{e,i}^k,
\end{equation}
where the locally closed subset $C_{e,i}^k\subset Y^k_{k,k+e+1}$ is the locus of points $[I] \in Y^k_{k,k+e+1}$ such that $\iota(I)=i$. Taking Euler characteristics and summing over these strata $ C_{e,i}^k$, we recover  \cite[Eqn.~(2.4)]{GOVINDARAJAN2013600}.

Our general expectation is to be able to find a collection of finitely many motives $\Gamma_{e,i}\in K_0(\Var_\BC)$ along with universal expressions for the motives $[C_{e,i}^k]$ in terms of these auxiliary classes. We motivate this via a working example.
\begin{example} \label{example:definition C}
We have computed the classes $[C^k_{e,i}]$ for $e \leq 2$ and arbitrary $k$. We exhibit here universal expressions for them in terms of finitely many classes $\Gamma_{e,i}$. Define
\begin{align*}
    \Gamma_{0,0}&=\Gamma_{1,1} = \Gamma_{1,2} = \Gamma_{2,1} =1,\\
    \Gamma_{2,2}&=\Gamma_{2,4} = [\BP^2],\\
    \Gamma_{2,3}&= \BL^8+\BL^7+2\BL^6+2\BL^5+2\BL^4-\BL^2-\BL. 
\end{align*}
Then we have
    \begin{align*}
    [C_{1,1}^k]&=\Gamma_{1,1} [\BP^{k-2}], \\
    [C_{1,2}^k]&=\Gamma_{1,2} \BL^{k-1} [\BP^{\binom{k}{2}}],\\
    [C_{2,1}^k]&= \Gamma_{2,1} [\Gr(1,k)]\BL^{\frac{(k-1)(k+2)}{2}} ,\\ [C_{2,2}^k]&=\Gamma_{2,2} [\Gr(2,k)],\\
    [C_{2,3}^k]&=\Gamma_{2,3} [\Gr(3,k)],\\
    [C_{2,4}^k]&=\Gamma_{2,4} \left( \frac{\left[ \Gr\left(2,\binom{k+1}{2}\right) \right]}{[\BP^2]} - [\Gr(3,k)]\BL(\BL^5 + \BL^3 + \BL^2 - 1) - [\Gr(2,k)] \right),
    \end{align*}
    where the expression for  $[C_{2,4}^k]$ is a well defined element of $K_0(\Var_{\BC})$ by \Cref{lemm:congruenc}, as  for every $k\ge 1$, we have $\binom{k+1}{2}\not\equiv 2 \pmod 3$. We highlight that this example perfectly matches, after setting $\BL=1$, the combinatorial results in \cite[Sec.~2.5]{GOVINDARAJAN2013600}.
\end{example}
To gather evidence for \Cref{conj:halfdata}, we now explain how to determine the motives of the Hilbert schemes with half of the data prescribed by \Cref{cor:hilbrel}.
Say we want to compute the motive $[\Hilb^{4}(\BA^4)_0]$. By \Cref{cor:hilbrel} we need \emph{six} independent data: three for $[\Hilb^{4}(\BA^m)_0]$, two for $[\Hilb^{3}(\BA^m)_0]$ and one for $[\Hilb^{2}(\BA^m)_0]$. 

By using \eqref{Y-stratification} and \Cref{thm:motive-Y-stratum} we can write the motive $[\Hilb^{d}(\BA^m)_0]$ in terms of the appropriate $Y^k_{k,d}$, and then by the stratification \eqref{eqn:stratificationC} we can write everything in terms of the appropriate $\Gamma_{e,i}$. The reader can easily verify that writing $[\Hilb^{d}(\BA^m)_0]$ for $d \leq 4$ requires only $\Gamma_{1,1}, \Gamma_{1,2}$ and $\Gamma_{2,1}$ (excluding $\Gamma_{0,0}$, in the same spirit of excluding $[\Hilb^{1}(\BA^m)_0]$). The number of data required is therefore halved.

\begin{remark}\label{rmk:d-8-data}
Another perspective on \Cref{ques:howmanydata} can be given by working directly with $Y$-classes instead of Hilb-classes. Indeed, by \Cref{prop:easyY} we know $[Y_{k,d}^k]$ for $k=1, 2$ and $k=d-5, \ldots, d-1$, so that $\mathsf P_d(t)$ can be determined by $d-8$ values.
\end{remark}

\appendix
\section{Explicit formulas}\label{app:gen-functions}
In this appendix we collect several explicit formulas:
\begin{itemize}
    \item [\mylabel{app1}{(1)}] In \Cref{app:Pd-explicit} we compute the polynomials $\mathsf P_d(t)$ explicitly for $d \leq 8$, thus leading to  closed formulas for the series $\mathsf Z_d(t)$ from \Cref{MAIN-THEOREM-1} and, as a byproduct, for all motivic measures $w\colon K_0(\Var_{\BC}) \to S$ listed in \Cref{sec:K0(Var)}.
    \item [\mylabel{app2}{(2)}] In \Cref{subsec:motiveomega} we compute the polynomials $\mathsf Q_d(t)$ appearing in \Cref{conjecture-on-Omega-GF} for $d\leq 8$. These polynomials govern the motivic classes $\Omega^\bullet_d$.
    \item [\mylabel{app3}{(3)}] In \Cref{app:macmahon-polynomials} we compute the polynomials $r_d(t)$ predicted by \Cref{conj:erroreMac} for $d \leq 26$. These polynomial capture the discrepancy between $\chi(\Omega^n_d)$ and Macmahon's mysterious exponent $\binom{d+n-3}{n-2}$.
    \item [\mylabel{app4}{(4)}] In \Cref{subs:macmahondiscrepancy} we exhibit polynomials  $\mathsf M_0, \mathsf M_1$ and $\mathsf M_2$ confirming \Cref{conj: Andrews} up to order $26,25,24$ respectively.
\end{itemize}

\subsection{The series \texorpdfstring{$\mathsf Z_d(t)$}{} for \texorpdfstring{$d\leq 8$}{}}\label{app:Pd-explicit}
Recall that by \Cref{MAIN-THEOREM-1} we have
\[
\mathsf Z_d(t) = \sum_{n\geq 0}\,[\Hilb^{d}(\BA^{n+1})_0] t^n = \zeta_{\BP^{d-1}}(t) \cdot \mathsf{P}_d(t),
\]
where $\mathsf{P}_d(t)$ is $1$ for $d\leq 3$ and a polynomial of degree at most $d-2$ if $d>3$. We list here the polynomial $\mathsf{P}_d(t)$ up $d\leq 8$ points, computed using the computer software SAGE \cite{sagemath}.
\begin{itemize}
    \item[${d=1}:$] $1$
    \item[${d=2}:$] $1$
    \item[${d=3}:$] $1$
    \item[${d=4}:$] $ 1+\BL^2t-\BL^2t^2$
    \item[${d=5}:$] $ 1+(\BL^3+\BL^2)t+(\BL^5-\BL^3-\BL^2)t^2-\BL^5t^3$
    \item[${d=6}:$] $  1+ (2\BL^4 + 2\BL^3 + \BL^2)t+ (\BL^8 + 2\BL^7 + \BL^6 - \BL^5 - 3\BL^4 - 2\BL^3 - \BL^2)t^2 - (2\BL^8 + 2\BL^7 + \BL^6 - \BL^5 - \BL^4)t^3+ \BL^8t^4 $
    \item[${d=7}:$] $  1+(\BL^2 + 2 \BL^3 + 3 \BL^4 + 2 \BL^5)t 
+(-\BL^2 - 2 \BL^3 - 4 \BL^4 - 4 \BL^5 - \BL^6 + 3 \BL^7 + 5 \BL^8 + 5 \BL^9 + \BL^{10})t^2 
+(\BL^4 + 2 \BL^5 + \BL^6 - 3 \BL^7 - 7 \BL^8 - 10 \BL^9 - 6 \BL^{10} - 2 \BL^{11} + 
 3 \BL^{12} + 2 \BL^{13} + 2 \BL^{14} + \BL^{15} + \BL^{16})t^3 
+(2 \BL^8 + 5 \BL^9 + 6 \BL^{10} + 3 \BL^{11} - \BL^{12} - 4\BL^{13} - 3 \BL^{14} - 3 \BL^{15} - 
 2 \BL^{16} - \BL^{17})t^4 
+(-\BL^{10} - \BL^{11} - 2 \BL^{12} + 2 \BL^{13} + \BL^{14} + 2 \BL^{15} + \BL^{16} + \BL^{17})t^5$
    \item[${d=8}:$] $1+ (3\BL^6 + 4\BL^5 + 4\BL^4 + 2\BL^3 + \BL^2)t + (3\BL^{12} + 9\BL^{11} + 12\BL^{10} + 10\BL^9 + 3\BL^8 - 3\BL^7 - 7\BL^6 - 7\BL^5 - 5\BL^4 - 2\BL^3 - \BL^2)t^2 + (\BL^{21} + \BL^{20} + 3\BL^{19} + 5\BL^{18} + 7\BL^{17} + 9\BL^{16} + 10\BL^{15} + 3\BL^{14} - 7\BL^{13} - 19\BL^{12} - 25\BL^{11} - 22\BL^{10} - 12\BL^9 - 2\BL^8 + 4\BL^7 + 4\BL^6 + 3\BL^5 + \BL^4)t^3 + (\BL^{26} + \BL^{25} + 2\BL^{24} + \BL^{23} + \BL^{22} - 2\BL^{21} - 5\BL^{20} - 11\BL^{19} - 15\BL^{18} - 20\BL^{17} - 18\BL^{16} - 11\BL^{15} + 4\BL^{14} + 15\BL^{13} + 21\BL^{12} + 17\BL^{11} + 10\BL^{10} + 2\BL^9 - \BL^8 - \BL^7)t^4 - (\BL^{27} + 2\BL^{26} + 3\BL^{25} + 3\BL^{24} + 3\BL^{23} + \BL^{22} - 2\BL^{21} - 8\BL^{20} - 12\BL^{19} - 15\BL^{18} - 13\BL^{17} - 8\BL^{16} + \BL^{15} + 7\BL^{14} + 8\BL^{13} + 5\BL^{12} + \BL^{11})t^5 +(\BL^{27} + \BL^{26} + 2\BL^{25} + \BL^{24} + 2\BL^{23} - \BL^{21} - 4\BL^{20} - 4\BL^{19} - 5\BL^{18} + \BL^{16} + 2\BL^{15})t^6.$
\end{itemize}

This completes the calculation of the series $\mathsf Z_d(t)$ for $d \leq 8$. 
Note that, in all these examples, the polynomial $\mathsf P_d(t)$ has degree exactly $d-2$, cf.~\Cref{conj:degree-of-P_d}.

\subsection{The polynomials \texorpdfstring{$\mathsf{Q}_d(t)$ for \texorpdfstring{$d\leq 8$}{}}{}}\label{subsec:motiveomega}
We list here the  polynomials $\mathsf{Q}_d(t)$ (cf.~\Cref{conjecture-on-Omega-GF}) for $d\leq 8$ points, computed directly inverting the plethystic exponential. 

\begin{itemize}
    \item[${d=1}:$] $1$
    \item[${d=2}:$] $\BL$
    \item[${d=3}:$] $\BL^2$
    \item[${d=4}:$] $\BL^3+(-\BL^2+\BL^5)t $
    \item[${d=5}:$] $\BL^4+t(\BL^7 + \BL^6 - \BL^4 - \BL^3)+(\BL^9 + \BL^8 - \BL^6 - \BL^5)t^2 $
    \item[${d=6}:$] $\BL^5+(2\BL^9 + 3\BL^8 + 2\BL^7 - 2\BL^6 - 3\BL^5 - 2\BL^4)t+ (\BL^{12} + 3\BL^{11} + 2\BL^{10} - 2\BL^9 - 6\BL^8 - 3\BL^7 + 3\BL^5 + \BL^4)t^2+ (\BL^{14} + 2\BL^{13} - \BL^{11} - 2\BL^{10} + \BL^8)t^3 $
    \item[${d=7}:$] $\BL^6+(2\BL^{11} + 4\BL^{10} + 4\BL^9 - \BL^8 - 5\BL^7 - 4\BL^6 - \BL^5 + \BL^4)t+ (\BL^{16} + 4\BL^{15} + 9\BL^{14} + 10\BL^{13} + 4\BL^{12} - 10\BL^{11} - 17\BL^{10} - 13\BL^9 - \BL^8 + 5\BL^7 + 5\BL^6 + \BL^5)t^2+ (2\BL^{18} + 4\BL^{17} + 4\BL^{16} - 2\BL^{15} - 11\BL^{14} - 14\BL^{13} - 6\BL^{12} + 5\BL^{11} + 10\BL^{10} + 6\BL^9 + \BL^8)t^3+ (\BL^{20} + 2\BL^{19} + \BL^{18} - 2\BL^{16} + \BL^{15} + 2\BL^{14} + 2\BL^{13} - 2\BL^{12} - \BL^{11} - \BL^{10})t^4 $
    \item[${d=8}:$] $\BL^7+(3\BL^{13} + 7\BL^{12} + 9\BL^{11} + \BL^{10} - 8\BL^9 - 11\BL^8 - 4\BL^7 + \BL^6 + 2\BL^5)t+ ( \BL^{21} + \BL^{20} + 3\BL^{19} + 9\BL^{18} + 20\BL^{17} + 27\BL^{16} + 19\BL^{15} - 12\BL^{14} - 39\BL^{13} - 51\BL^{12} - 28\BL^{11} + \BL^{10} + 22\BL^9 + 19\BL^8 + 8\BL^7 - \BL^6 - \BL^5)t^2+ (\BL^{26} + \BL^{25} + 2\BL^{24} + 2\BL^{23} + 6\BL^{22} + 14\BL^{21} + 19\BL^{20} + 12\BL^{19} -20\BL^{18} - 52\BL^{17} -  64\BL^{16} - 37\BL^{15} + 9\BL^{14} + 43\BL^{13} + 44\BL^{12} + 21\BL^{11} - 9\BL^9 - 4\BL^8 - \BL^7)t^3+ (-\BL^{27} - 2\BL^{26} + 4\BL^{24} + 8\BL^{23} + 4\BL^{22} - 7\BL^{21} - 16\BL^{20} - 9\BL^{19} + 10\BL^{18} + 23\BL^{17} + 22\BL^{16} + 3\BL^{15} - 9\BL^{14} - 12\BL^{13} - 3\BL^{12} + \BL^{10} )t^4+ (2\BL^{27} + 3\BL^{26} + 4\BL^{25} + \BL^{24} - \BL^{23} - \BL^{22} + \BL^{21} - 6\BL^{19} - 6\BL^{18} - \BL^{17} + \BL^{16} + 2\BL^{15})t^5$.
\end{itemize}  

\subsection{MacMahon's discrepancy for \texorpdfstring{$d\leq 26$}{}}\label{app:macmahon-polynomials}
We list here the irreducible polynomials $r_d(n)\in\BQ[n]$ (cf.~\Cref{conj:erroreMac}) for $d=6,\ldots,26$.

\begin{itemize}
    \item[${d=6}:$] $1$
    \item[${d=7}:$] $ n - 2$
    \item[${d=8}:$] $\frac{1}{15}( 8 n^2 - 15 n - 38)$
    \item[${d=9}:$] $\frac{1}{5}( n^3 - n^2 - 26 n + 51 )$
    \item[${d=10}:$] $\frac{1}{1680}(99 n^4 - 170 n^3 - 4223 n^2 + 1766 n + 24200 )$
    \item[${d=11}:$] $\frac{1}{5040}(73 n^5 - 597 n^4 - 1043 n^3 - 42051 n^2 + 235834 n - 281928)$
    \item[${d=12}:$] $\frac{1}{75600}(233 n^6 - 6387 n^5 + 23405 n^4 - 430245 n^3 + 1452722 n^2 + 2747472 n - 10278720)$
    \item[${d=13}:$] $\frac{1}{151200}( 88 n^7 - 6245 n^6 + 31097 n^5 - 236165 n^4 - 1469593 n^3 + 21899170 n^2 - 64976192 n + 53149440 )$
    \item[${d=14}:$] $\frac{1}{19958400}(1981 n^8 - 311988 n^7 + 1935414 n^6 - 3845172 n^5 - 193010871 n^4 + 1493337048 n^3 - 359131804 n^2 - 17294094288 n + 31388071680)$
    \item[${d=15}:$] $\frac{1}{6652800}( 103 n^9 - 32632 n^8 + 318242 n^7 - 599228 n^6 - 16650853 n^5 - 46403708 n^4 + 2470215868 n^3 - 13748769232 n^2 + 26565820320 n - 14059278720)$
    \item[${d=16}:$] $\frac{1}{1816214400}(4050 n^{10} - 2415235 n^9 + 44202720 n^8 - 198998202 n^7 + 543569502 n^6 - 33951478911 n^5 + 406407860400 n^4 - 886475686148 n^3 - 5844164792832 n^2  + \\ 29344150425216 n - 36495822424320 )$
    \item[${d=17}:$] $\frac{1}{9081072000}(2713 n^{11} - 2905932 n^{10} + 102936245 n^9 - 696362890 n^8 + 3912831339 n^7 - 60103579116 n^6 + 105375461075 n^5 + 7975305525090 n^4 - 72894798639452 n^3 + 241427250677248 n^2 - 287098838112720 n + 13072622068800)$
    \item[${d=18}:$] $\frac{1}{871782912000}(32647 n^{12} - 60701262 n^{11} + 4011851771 n^{10} - 35040317130 n^9 + \\ 168916525761 n^8 - 467520183306 n^7 - 30485333857327 n^6 + 537458552700810 n^5 - 2257598433391508 n^4 - 7856165277517032 n^3 + 88701471042013056 n^2 - \\ 249606482614519680 n + 236887198095744000)$
    \item[${d=19}:$] $\frac{1}{871782912000} (3847 n^{13} - 12105581 n^{12} + 1419885673 n^{11} - 16448633149 n^{10} + \\ 77399715691 n^9 + 148861849677 n^8 - 12298839695941 n^7 + \\ 69445773950033 n^6 + 1726606727061122 n^5 - 24456018469726196 n^4 + \\
    120768098399825768 n^3 - 234712358509584384 n^2 + 20565176385939840 n +\\ 334664360427801600)$
    \item[${d=20}:$] $\frac{1}{133382785536000}(65459 n^{14} - 341997999 n^{13} + 67993482881 n^{12} - 1137851310231 n^{11} + 6863878273427 n^{10} - 11397386114337 n^9 - 94916460499517 n^8 - 9137497186611453 n^7 + 219397451904763774 n^6 - 1394004307912199964 n^5 - 2536756760556061064 n^4 + 68645406748809408384 n^3 - 323391855641276684160 n^2 + 649398353859605376000 n - 483528616380614246400)$
    \item[${d=21}:$] $\frac{1}{29640619008000}(1532 n^{15} - 13095959 n^{14} + 4247443551 n^{13} - 109750132475 n^{12} + \\ 885524398299 n^{11} - 4126838161517 n^{10} + 35011946515013 n^9 - \\ 1008817082259145 n^8 + 8866308376448781 n^7 + 127237613967358136 n^6 - \\ 2741156531918400024 n^5 + 18625194630887394320 n^4 - 51388145341071878512 n^3 + 3220254428031661440 n^2 + 261708337903419912960 n - 378213429350798899200 )$
    \item[${d=22}:$] $\frac{1}{101370917007360000}(524097 n^{16} - 7247511368 n^{15} + 3717237160520 n^{14} - \\ 152184186299660 n^{13} + 1570134610133714 n^{12} - 9035640753821276 n^{11} + \\ 50458962657912320 n^{10} - 402805530960448180 n^9 - 13424282139388183019 n^8 + 439767764755867019156 n^7 - 3802421532685450383200 n^6 - 2095537058812293622960 n^5 + 
    234480780397771174644208 n^4 - 1620421765907622698032512 n^3 + \\
    5090557538890888381359360 n^2 - 7614777869397455680051200 n + \\
    4157088781168635992064000 )$
    \item[${d=23}:$] $\frac{1}{709596419051520000}(349455 n^{17} - 7748519514 n^{16} + 6128700383656 n^{15} - 396196709710600 n^{14} + 5160723726554410 n^{13} - 33056117183302108 n^{12} + 106247070348532972 n^{11} + \\1038864681729035240 n^{10} - 56225386232920907185 n^9 + 680070895686740877958 n^8 + 6320546409234867610388 n^7 - 209510306821027467249680 n^6 +\\ 1856164786465966273736720 n^5 - 6687218730778391669954336 n^4 -\\  626845029527650429502016 n^3 + 81776578356238362723671040 n^2 -\\ 241130010521642210056166400n + 230953469310039778735104000)$
    \item[${d=24}:$] $\frac{1}{23416681828700160000} ( 1048460 n^{18} - 37016385183 n^{17} + 44238216645036 n^{16} -\\ 4438506626037516 n^{15} + 74937699997317060 n^{14} - 562591452104653566 n^{13} + \\1958819321299862252 n^{12} + 8888255431431499608 n^{11} - 243421963109949028680 n^{10} - 6031308305222396144619 n^9 + 267803044865827855728888 n^8 - \\2966075995664738789286468 n^7  + 1622425560699019170248840 n^6 + \\
    227322542130051193367587968 n^5 - 2127700548933581500558150176 n^4 + \\
    9169544904833319462449093376 n^3 - 20256824429334656867442055680 n^2 +\\  19649874106231223185607270400 n - 3240290359379905083850752000)$
    \item[${d=25}:$] $\frac{1}{4683336365740032000} (18235 n^{19} - 1019399825 n^{18} + 1810365879804 n^{17} -\\ 275863055020407 n^{16} + 6287711297318214 n^{15} - 57901707711619446 n^{14} + \\
    289219084251460720 n^{13} - 1009283386616377538 n^{12} + 19676156331158908191 n^{11} - 899803001299779118113 n^{10} + 13526465786506229882988 n^9 + \\
    81138404334872119471977 n^8 - 4340287998236452397458904 n^7 +\\  48200906266152323175663784 n^6 - 214116222702857039223114912 n^5 -\\  96255084553962509937965232 n^4 + 5251379089177169319784011264 n^3 -\\  23362801119060581617375852800 n^2 + 45371842164302310664455168000 n - \\34189093301675630118532300800 )$
    \item[${d=26}:$] $\frac{1}{25852016738884976640000} ( 8388331 n^{20} - 739204394610 n^{19} + 1924250747704469 n^{18} -\\
    435800072551872714 n^{17} + 13867745325067836750 n^{16} - 157526730215754733908 n^{15} +
    1001864136149464848298 n^{14} - 5103149030391359492868 n^{13} +\\
    52400005207020194241479 n^{12} - 825597681149570318033226 n^{11} -\\
    12605867718739135261848543 n^{10}+  783086200071350660814554478 n^9 - \\
    10676618251495171555430595632 n^8 + 16369272655785444492325116816 n^7 + \\
    1011212771005471932089807675216 n^6 - 12110261856446073932069935204896 n^5 + 66250904426343966025992362067072 n^4 - 187864832852428838553998453187072 n^3 + 217829258262474632939407900354560 n^2 + 111832291754129290393687615488000 n - 368278371840755647668884078592000)$.
\end{itemize} 

\subsection{Generating polynomial for MacMahon's discrepancy} \label{subs:macmahondiscrepancy}
We exhibit polynomials  $\mathsf M_0, \mathsf M_1$ and $\mathsf M_2$ confirming \Cref{conj: Andrews} up to order $26,25,24$ respectively.
\begin{itemize}
    \item[${k=0}:$] $1$ 
    \item[${k=1}:$] $6t^5 - 34t^4 + 58t^3 - 20t^2 - 7t + 3$ 
    \item[${k=2}:$] $-144t^{11}+1356t^{10}-5770t^9+13965t^8-19993t^7+15064t^6-5170t^5+545t^4-206t^3+244t^2-79t+8$ .
\end{itemize} 

\bibliographystyle{amsplain-nodash}
\bibliography{The_Bible}

\bigskip
\noindent
{\small{Michele Graffeo \\
\address{SISSA, Via Bonomea 265, 34136, Trieste (Italy)} \\
\href{mailto:mgraffeo@sissa.it}{\texttt{mgraffeo@sissa.it}}
}}

\bigskip
\noindent
{\small{Sergej Monavari \\
\address{\'Ecole Polytechnique F\'ed\'erale de Lausanne (EPFL),  CH-1015 Lausanne (Switzerland)} \\
\href{mailto:sergej.monavari@epfl.ch}{\texttt{sergej.monavari@epfl.ch}}
}}

\bigskip
\noindent
{\small{Riccardo Moschetti \\
\address{Università di Torino, Via Carlo Alberto 10, 10123, Torino (Italy)} \\
\href{mailto:riccardo.moschetti@unito.it}{\texttt{riccardo.moschetti@unito.it}}
}}

\bigskip
\noindent
{\small Andrea T. Ricolfi \\
\address{SISSA, Via Bonomea 265, 34136, Trieste (Italy)} \\
\href{mailto:aricolfi@sissa.it}{\texttt{aricolfi@sissa.it}}}
\end{document}